\documentclass[12pt,twoside,reqno]{amsart}
\usepackage[colorlinks=true,citecolor=blue]{hyperref}
\usepackage{mathptmx, amsmath,dsfont, amssymb, amsfonts, amsthm, mathptmx, enumerate, color, mathrsfs, soul, tabularx, xcolor}
\def\grad{\nabla}

\def\bb{\mathbf{b}}

\def\bp{\mathbf{p}}

\def\bs{\mathbf{s}}

\def\bw{\mathbf{w}}
\def\bx{\mathbf{x}}  
\def\by{\mathbf{y}}
\def\bz{\mathbf{z}}

\def\bD{\mathbf{D}}

\def\bI{\mathbf{I}}

\def\cB{\mathcal{B}}

\def\cD{\mathcal{D}}
\def\cE{\mathcal{E}}

\def\cG{\mathcal{G}}
\def\cH{\mathcal{H}}
\def\cI{\mathcal{I}}

\def\cK{\mathcal{K}}
\def\cL{\mathcal{L}}

\def\cN{\mathcal{N}}
\def\cO{\mathcal{O}}
\def\cP{\mathcal{P}}

\def\cS{\mathcal{S}}

\def\cX{\mathcal{X}}
\def\cY{\mathcal{Y}}

\def\smskip{\smallskip}

\def\texitem#1{\par\smskip\noindent\hangindent 25pt
               \hbox to 25pt {\hss #1 ~}\ignorespaces}

\def\norm#1{\|#1\|}

\newcommand{\BEAS}{\begin{eqnarray*}}
\newcommand{\EEAS}{\end{eqnarray*}}
\newcommand{\BEA}{\begin{eqnarray}}
\newcommand{\EEA}{\end{eqnarray}}
\newcommand{\BEQ}{\begin{eqnarray}}
\newcommand{\EEQ}{\end{eqnarray}}
\newcommand{\BIT}{\begin{itemize}}
\newcommand{\EIT}{\end{itemize}}
\newcommand{\BNUM}{\begin{enumerate}}
\newcommand{\ENUM}{\end{enumerate}}

\newcommand{\BA}{\begin{array}}
\newcommand{\EA}{\end{array}}

\newcommand{\reals}{\mathbb{R}}
\newcommand{\integers}{\mathbb{Z}}

\newcommand{\diag}{\mathop{\bf diag}}

\def\fprod#1{\left\langle#1\right\rangle}

\DeclareMathOperator*{\argmax}{\mathbf{argmax}}
\DeclareMathOperator*{\argmin}{\mathbf{argmin}}

\newcommand{\dom}{\mathop{\bf dom}}

\usepackage{xcolor}

\newif\ifpagenumbering
\pagenumberingtrue

\pagenumberingfalse
\newsavebox{\theorembox}
\newsavebox{\lemmabox}
\newsavebox{\defnbox}
\newsavebox{\assbox}
\savebox{\theorembox}{\noindent\bf Theorem}
\savebox{\lemmabox}{\noindent\bf Lemma}
\savebox{\defnbox}{\noindent\bf Definition}

\usepackage{soul}
\usepackage[normalem]{ulem}
\DeclareUnicodeCharacter{2212}{~}

\usepackage{hyperref}
\usepackage{cleveref}
\crefname{assumption}{Assumption}{Assumptions}
\crefname{theorem}{Theorem}{Theorems}
\crefname{lemma}{Lemma}{Lemmas}

\usepackage{relsize}

\def\mg#1{\textcolor{magenta}{#1}}

\definecolor{darkgreen}{rgb}{0.2,0.8,0.1}
\def\mg#1{\textcolor{black}{#1}}

\def\sa#1{\textcolor{black}{#1}}

\newcommand{\mgb}[1]{{\color{black} #1}} 
\newcommand{\mgc}[1]{{\color{black} #1}}

\def\mg#1{\textcolor{black}{#1}}

\DeclareFontFamily{U}{stix2bb}{}
\DeclareFontShape{U}{stix2bb}{m}{n} {<-> stix2-mathbb}{}

\def\J{\mathbf{J}}
\def\alg{\texttt{D-APDB}}
\def\algz{\texttt{D-APDB0}}
\def\lin#1{\texttt{Line}~\ref{#1}}
\usepackage{mathtools}

\usepackage{graphicx}
\usepackage{subcaption}
\usepackage{caption}
 \usepackage{hyperref}

\usepackage{multicol}
\usepackage{algorithm}

\usepackage{epstopdf}

\newtheorem{assumption}{Assumption}
\newtheorem{condition}{Condition}

\usepackage{algpseudocode}

\usepackage{todonotes}
\usepackage{etoolbox}

\newtheorem{theorem}{Theorem}[section]

\newtheorem{corollary}[theorem]{Corollary}
\newtheorem{lemma}[theorem]{Lemma}

\theoremstyle{definition}
\newtheorem{definition}[theorem]{Definition}

\newtheorem{remark}[theorem]{Remark}

\numberwithin{equation}{section}

\begin{document}
\setcounter{page}{1}

\vspace*{1.0cm}
\title[Decentralized Gradient Descent-Ascent with Backtracking]
{AN ACCELERATED PRIMAL DUAL ALGORITHM WITH BACKTRACKING FOR DECENTRALIZED CONSTRAINED OPTIMIZATION}
\author[Q. Xu, N.S. Aybat, M. G\"urb\"uzbalaban]{Qiushui Xu$^1$, Necdet S. Aybat$^{1,*}$, Mert G\"urb\"uzbalaban$^2$}
\maketitle
\vspace*{-0.6cm}

\begin{center}
{\footnotesize {\it

$^1$Department of Industrial Engineering, Pennsylvania State University, University Park, PA 16802, USA \\
$^2$Department of Management Science \& Information Systems, Rutgers University, Piscataway, NJ 08854, USA

}

}\end{center}

\vskip 4mm {\small \noindent {\bf Abstract.}
We propose a distributed accelerated primal-dual method with backtracking (\texttt{D-APDB}) for cooperative multi-agent constrained consensus optimization problems over an undirected network of agents, where only those agents connected by an edge can directly communicate \sa{to exchange large-volume data vectors using a high-speed, short-range communication protocol, e.g., WiFi, and we also assume that the network allows for one-hop simple information exchange beyond immediate neighbors as in LoRaWAN protocol}. The objective is to minimize the sum of agent-specific composite convex functions \sa{over agent-specific private constraint sets}. Unlike existing decentralized primal-dual methods that require knowledge of the Lipschitz constants, \texttt{D-APDB} automatically adapts to local smoothness by employing a distributed backtracking step-size search. Each agent relies only on first-order oracles associated with its own objective and constraint functions and on local communications with the neighboring agents, without any prior knowledge of Lipschitz constants. We establish $\cO(1/K)$ convergence guarantees for sub-optimality, infeasibility and consensus violation, under standard assumptions on smoothness and on the connectivity of the communication graph. To our knowledge, \sa{when nodes have private constraints, especially when they are nonlinear convex constraints onto which projections are not cheap to compute,} \texttt{D-APDB} is the first distributed method with backtracking that achieves the optimal convergence rate for \sa{the class of constrained composite convex optimization problems.} 
We provide numerical results for \texttt{D-APDB} on a distributed QCQP problem and distributed primal SVM training, illustrating the potential performance gains that can be achieved by \texttt{D-APDB}.


 }

\renewcommand{\thefootnote}{}
\footnotetext{ $^*$Corresponding author.
\par
E-mail addresses: qjx5019@psu.edu  (Q. Xu), nsa10@psu.edu  (N.S. Aybat), mg1366@rutgers.edu  (M. G\"urb\"uzbalaban). }

\section{Introduction}

\mgb{Modern datasets are large, and data are often acquired by computational agents connected over a communication network. 
\sa{In this context,} an agent denotes any computational entity (e.g., device, node, sensor, or processor) capable of local computation and communication. In such systems, information is inherently distributed across agents that exchange messages only with neighbors, rendering centralized processing impractical or infeasible due to bandwidth, latency, energy, and privacy constraints. These factors motivate \emph{decentralized optimization} algorithms, wherein data and computation reside across multiple agents and the global objective is optimized via local updates and limited message passing, without sharing raw data. Such methods arise in a wide range of applications, including machine learning with decentralized data, control and coordination in multi-robot systems, smart grids, signal processing and estimation over sensor networks \cite{ideal,aybat2017distributed,dall2013distributed,Lu2021DecentralizedPG,rabbat2004distributed}. We refer the reader to \cite{SYS-004,sayed2014adaptation} for additional examples
\sa{in a} broader context.
In these settings, the goal is \sa{to collaboratively solve an optimization problem defined by agent-specific objectives and/or constraints through utilizing the computing capability of agents across the network}—without a central coordinator—using only local computations and peer-to-peer communication \sa{among the neighboring nodes}.}

\sa{Let $\cN\triangleq\{1,\ldots, N\}$ denote the set of agents in the network.} In many applications, the global feasible set is given by the intersection of agent-specific nonlinear constraint sets that encode local information or individual operational requirements. These constraints are often naturally expressed as convex functional constraints of the form $-g_i(x)\in \cK_i$ in the decision variable $x$, where each $g_i$ is a vector-valued mapping and $\cK_i$ is a closed convex cone, \sa{for each $i\in\cN$}. Objective functions are frequently composite as well, combining smooth losses with nonsmooth regularizers or indicator functions that model various constraints. In this \sa{paper,} we consider the following \sa{conic} constrained optimization problems:
\begin{equation}\label{eq:opt}
    \sa{\varphi^*}\triangleq \min_{\sa{x\in\reals^n}} \sum_{i\in\cN} \sa{\varphi_i(x)\triangleq\phi_i(x)+}f_i(x) \quad {\rm s.t.}\quad -g_i(x)\in\cK_i \quad \forall i\in\cN\triangleq\{1,\cdots,N\},
\end{equation}
where $\phi_i:\reals^n\to\reals\cup\{+\infty\}$ is a proper, closed convex (possibly nonsmooth) function, $f_i:\reals^n\to\reals$ is a smooth convex function, $g_i:\reals^n\to\reals^{m_i}$ is a smooth $\cK_i$-convex function$^1$\footnote{$^1$Given a closed convex cone $\cK\subset\reals^m$, and $f:\reals^n\to\reals^m$ is $\cK$-convex if $f(\lambda x'+(1-\lambda)x'')\preceq_{\cK}\lambda f(x')+(1-\lambda)f(x'')$ holds for all $x',x''$ and $\lambda\in [0,1]$, where $\preceq_{\cK}$ denotes the partial order induced by $\cK$, i.e., $y'\preceq_{\cK} y''$ for $y',y''\in\reals^m$ when $y''-y'\in\cK$.}, and $\cK_i \subseteq \reals^{m_i}$ is a closed convex cone for $i\in\cN$. \sa{The  data of agent $i$ defining $\varphi_i$, $g_i$ and $\cK_i$} are private and not shared globally, and agents cooperate to solve \eqref{eq:opt} using only local computation and limited communication. This is a general class of problems that include many \sa{applications and important special cases}~\cite{Boyd2011DistributedOA, lee-nedich,Lu2021DecentralizedPG,Nedi2009DistributedSM,SYS-004,sayed2014adaptation}.

\mgb{Over the past \sa{few} decades, a rich literature has emerged on decentralized constrained optimization problems \sa{that are special cases of the class of problems in \eqref{eq:opt}. In particular,} distributed (sub)gradient schemes, gradient tracking methods, and their accelerated and proximal variants have been extensively studied under a variety of assumptions on smoothness, convexity, and network connectivity, e.g., \cite{ghaderyan2023fast,Jakoveti2011FastDG,8752033, 5404774, shi2015extra,7169615,NEURIPS2020_d4b5b5c1}. Other approaches include \sa{primal–dual type methods based on Augmented Lagrangian and ADMM formulations, e.g.,} \cite{ideal,aybat2015asynchronous,aybat2016primal,aybat2017distributed,Boyd2011DistributedOA,ghaderyan2023fast,hamedani2021decentralized,koshal2011multiuser}. The latter class of algorithms typically rely on Lagrangian formulations and exploit separability across agents to achieve fully distributed implementations. \sa{Except very few~\cite{aybat2016primal,hamedani2021decentralized}, none of the aforementioned methods can handle \eqref{eq:opt} in its full generality without employing projections onto $\{x: -g_i(x)\in\cK_i\}$. More importantly,} a persistent challenge in nearly all of these methods is \sa{to appropriately choose stepsizes with theoretical convergence guarantees}.
\sa{Such} guarantees often require stepsizes that depend on global problem parameters—such as \sa{global Lipschitz constants for
$\grad f_i$ and Jacobians of
$g_i$ for $i\in\cN$}, that are difficult to estimate, rarely known a-priori, and may vary widely across agents.
In practice, this leads practitioners to grid-search or adopt conservative, globally synchronized stepsizes that can significantly 
\sa{degrade the convergence speed}. The difficulty is further exacerbated in the presence of nonlinear functional constraints, where the relevant Lipschitz constants for the Lagrangian gradients depend on the dual iterates, which typically evolve within a cone; therefore, the Lipschitz constants may be only locally finite but not globally.}

\mgb{The literature on parameter-free decentralized algorithms—i.e., methods that do not require prior knowledge of Lipschitz constants—remains limited. A small body of work studies decentralized adaptive gradient methods that exploit past gradient information to choose stepsizes adaptively across iterations. For example, \cite{nazari2022dadam} develops a distributed adaptive moment estimation method for online smooth convex and nonconvex minimization and establishes regret bounds, while \cite{chen2023decentralizedadaptive} proposes a general framework for converting centralized adaptive gradient schemes such as RMSProp, Adam, and AdaGrad into decentralized counterparts.} \sa{We should emphasize that for the methods proposed in~\cite{chen2023decentralizedadaptive} to have convergence guarantees, it is required that the step size is sufficiently small inversely proportional to $1/\max_{i=1,\ldots,N}\{L_{f_i}\}$,} \mgb{where $L_{f_i}$ is the Lipschitz constant of \sa{$\grad f_i$}, and this would require some information exchange \sa{beyond} the immediate neighbors on the network. Furthermore, 
\sa{both} \cite{nazari2022dadam} and \cite{chen2023decentralizedadaptive} consider unconstrained smooth optimization problems and do not handle constraints.}

\mgb{The work \cite{li2024problem} proposes a parameter-free decentralized algorithm for nonconvex stochastic optimization and the proposed method
\sa{does not require} knowledge of 
\sa{$\{L_{f_i}\}_{i\in\cN}$} or topological information about the communication network. Another line of work \cite{aldana2024towards} introduces a port-Hamiltonian systems framework for the design and analysis of distributed optimization algorithms, and develops methods for smooth, strongly convex, and unconstrained decentralized problems. Convergence of the proposed decentralized methods is guaranteed for certain special graph classes or when the stepsize is chosen below a threshold that depends on global network information \sa{\cite[Corollaries 2 and 3]{aldana2024towards}. However, we should emphasize that the proposed approaches in \cite{aldana2024towards,li2024problem} both focus on smooth minimization problems and they do not address either the node-specific constraints or the non-smooth terms in the objective.}}

\mgb{More recently, \cite{kuruzov2024achieving} \sa{considered distributed smooth strongly convex unconstrained minimization problems and} proposed a parameter-free decentralized optimization algorithm that employs a local backtracking line search to select stepsizes adaptively—without requiring global information or heavy communication, \sa{and the authors analyze two different implementations of the proposed method. The first implementation} uses a global min-consensus to synchronize agents’ stepsizes, 
\sa{which can be computed in practice} via flooding protocols over low-power wide-area networks, 
e.g., \sa{LoRa~\cite{askhedkar2020hardware,janssen2020lora,kim2016low}} (low power, long range), \sa{as it requires simple} information exchange beyond immediate neighbors. \sa{On the other hand, the second one} replaces the global min-consensus \sa{update} with a local min-consensus \sa{update}, restricting communication to immediate neighbors; however, \sa{as mentioned in the follow-up work~\cite{kuruzov2025adaptive},} this increased locality comes at the cost of weaker theoretical guarantees 
\sa{due to} stronger assumptions 
\sa{requiring that the iterates stay} bounded, 
and potentially non-monotone convergence trajectories. In~\cite{kuruzov2025adaptive}, the authors introduce  a   fully decentralized algorithm in which each agent adaptively selects its own stepsize using only neighbor-to-neighbor communication and no global information \sa{is required for its implementation}—agents need not even know whether the problem is strongly convex. The algorithm retains strong guarantees: it converges at a \emph{linear} rate under strong convexity and at a \sa{$\cO(1/T)$-\emph{sublinear} rate under mere convexity}, matching the best-known rates for parameter-dependent, nonadaptive \sa{distributed} methods. \sa{In the merely convex setting, the convergence guarantees for the method in~\cite{kuruzov2025adaptive} also require that the primal-dual iterate sequence stays bounded (see~\cite[Theorem 10]{kuruzov2025adaptive}) as in their earlier work~\cite{kuruzov2024achieving} --since neither work \cite{kuruzov2025adaptive,kuruzov2024achieving} can handle indicator functions to directly impose boundedness, the authors argue that this requirement can be satisfied by some particular choice of algorithmic parameter sequences.}

\sa{All of the previously mentioned 
works are designed to} address \emph{unconstrained} decentralized problems of the form $\min_{x\in\reals^n}\sum_{i\in\cN} f_i(x)$, where each $f_i$ is smooth with a Lipschitz-continuous gradient. Consequently, they do not apply to the more general constrained problems we consider in \eqref{eq:opt}, which feature agent-specific local constraints and nonsmooth terms $\phi_i(x)$ in the objective. \sa{A more} recent work \cite{chen2025parameter} is the first to accommodate nonsmoothness, focusing on problems of the form $\min \sum_{i\in\cN} f_i(x)+\phi(x)$, where $\phi$ is a convex, nonsmooth function known to \textit{all} agents \sa{in the network}. 
\sa{The work in \cite{chen2025parameter}} introduces a novel three-operator splitting technique and proposes \sa{a method} that requires neither global network information nor extensive inter-agent communication. The resulting adaptive decentralized method enjoys robust convergence guarantees and outperforms existing nonadaptive approaches.

\sa{These aforementioned results on adaptive distributed methods} are exciting developments \sa{for decentralized optimization}; nevertheless, none of these existing methods \sa{that we discussed above} can handle a problem of the form \eqref{eq:opt} \sa{in its full generality. Indeed, to our knowledge, there is no decentralized optimization method that does not rely on the a priori knowledge of Lipschitz constants and that can handle agent-specific local constraints defined by nonlinear convex
functions.}}

\mgb{We should mention that for \emph{centralized} \emph{unconstrained} optimization, where a central node can aggregate and process all data, a wide range of parameter-free methods is available. These include centralized adaptive gradient schemes such as AdaGrad \cite{duchi2011adagrad}, RMSProp \cite{tieleman2012rmsprop}, Adam \cite{kingma2015adam}, and their variants \cite{luo2019adabound,reddi2018convergence}, as well as methods based on Barzilai–Borwein stepsizes \cite{barzilai1988two,burdakov2019stabilized,zhou2025adabb} and adaptive techniques that estimate local curvature \cite{malitsky2020adaptive,malitsky2024adaptive}. For centralized constrained \sa{convex} optimization, backtracking schemes can estimate local Lipschitz constants on the fly, yielding convergence-rate guarantees without prior knowledge of problem-specific smoothness parameters. In particular, Lagrangian reformulations of \eqref{eq:opt} lead to convex–concave saddle--point (SP) problems, for which primal–dual methods with backtracking have been developed \cite{hamedani2021primal,jiang_mokthari} --\sa{on a different note, there are also backtracking-based algorithms for centralized non-convex min–max formulations~\cite{xu2024stochastic,zhang2024agda+}, which can be used for \eqref{eq:opt} as well.} However, these methods fundamentally rely on centralized aggregation; therefore, they do not extend directly to \sa{the decentralized computation setting we consider in this paper,} where no single node has access to global information. There are also adaptive distributed methods in federated learning or central-server settings, e.g., \cite{chen2020toward,reddi2021adaptive,Xie2020}; but these algorithms still require a central node to aggregate information from the network and are thus not applicable to the fully decentralized setting \sa{over an arbitrary connected undirected communication network} considered in this paper.}

\mgb{To address this gap, we propose \texttt{D-APDB}, a novel decentralized primal--dual algorithm that incorporates a backtracking mechanism for \emph{local}, agent-specific stepsize selection to solve \eqref{eq:opt}. Unlike existing decentralized primal-dual methods that require knowledge of the Lipschitz constants, \texttt{D-APDB} automatically adapts to local smoothness by employing a distributed backtracking step-size search. Each agent relies only on first-order oracles associated with its own objective and constraint functions, \sa{without any prior knowledge of Lipschitz constants. At each iteration of \texttt{D-APDB}, each agent locally communicates a $n$-dimensional vector one time with the neighboring agents, and the agents across the network collectively implement a max-consensus one time. As explained in \cite{kuruzov2024achieving}, the max-consensus   protocol is well-suited to existing wireless mesh network technologies; more precisely, LoRa~\cite{askhedkar2020hardware,janssen2020lora,kim2016low} enables wide-area coverage at low data rates, which is advantageous for network-wide flooding where each transmission reaches all nodes in one hop but conveys limited information.} We establish $\cO(1/K)$ convergence guarantees for sub-optimality, infeasibility and consensus violation, under standard assumptions on smoothness and on the connectivity of the communication graph. To our knowledge, \texttt{D-APDB} is the first distributed method with backtracking that achieves the optimal convergence rate for the class of composite convex optimization problems subject to functional agent-specific convex constraints. \mg{Furthermore, we propose a 
\sa{variant} of our method, which we call \algz{}, tailored to the setting \sa{with $g_i(\cdot)=0$
for all $i\in\cN$ while agents can still have closed convex functions $\phi_i$ in their local objectives}. \sa{\algz{}} can also achieve $\cO(1/K)$ convergence guarantees for sub-optimality and consensus violation \sa{--to the best of our knowledge,
 this is the first time a rate result is provided for a decentralized method that can handle node-specific closed convex functions without requiring a priori knowledge on Lipschitz constants.}} Finally, we present numerical results for \sa{\algz{} and \alg{} on a distributed $\ell_1$-norm regularized QP and QCQP problems},
illustrating the potential performance gains achievable \sa{with our proposed backtracking framework.}}

\paragraph{\textbf{Notation}}

Throughout the text $\norm{\cdot}$ denotes the Euclidean norm and the spectral norm when the argument is a vector and a matrix, respectively. Given two sets $A$ and $B$, $A\times B$ denotes the Cartesian product, and $\Pi_{i=1}^MA_i$ denotes the Cartesian product of a collection of sets $\{A_i\}_{i=i}^M$. In few places, we abuse the notation $\Pi_{i=1}^M a_i$ to denote the multiplication of a collection of real numbers $\{a_i\}_{i=1}^M\subset\reals$.
Given a set $S$, let $\mathds{1}_{S}(\cdot)$ denote  its indicator function, i.e., $\mathds{1}_{S}(x)=0$ if $x\in S$ and is equal to $+\infty$ otherwise. For $S\subset\reals^n$, $\cP_{S}(\cdot)$ denotes the Euclidean projection onto $S$; moreover, $d_S:\reals^n\to\reals_+$ denotes the distance function, i.e., $d_S(x)=\|x-\cP_S(x)\|$.
Given a convex cone $\cK\subseteq \reals^{m}$, let $\cK^*$ denote its dual cone, i.e., $\cK^* \triangleq \{\theta\in \reals^m: \langle \theta, w\rangle \geq 0,\quad \forall w\in\cK\}$.
\sa{Given a discrete set $\cS$, $|\cS|$ denotes the cardinality of $\cS$. \mgb{The set $\mathbb{S}^{n_x}_{++}$ denotes the set of $n_x\times n_x$ positive definite matrices.} Throughout the text, we use $\diag(\cdot)$ to construct block-diagonal matrices, i.e., given $G_i\in\reals^{m_i\times n}$ for $i\in\cN$, $G=\diag\left([G_i]_{i\in\cN}\right)$ denotes a block diagonal matrix with diagonal blocks being $G_i$ for $i\in\cN$; hence, $G\in\reals^{m\times n|\cN|}$ where $m=\sum_{i\in\cN}m_i$.} \mgb{We use $I_m$ to denote the $m\times m$ identity matrix,} \mgc{and $\mathbf{1}_m$ to denote the $m$-dimensional vector of ones.}

\section{Assumptions and the Main Results}

We investigate \sa{distributed methods for efficiently solving the problem in \eqref{eq:opt}} over a network of $\cN$ agents \sa{with computing and message passing capability. 
\subsection{Assumptions}
This network is} modeled as an undirected, static, connected
graph $\cG = (\cN, \cE)$, where \sa{$\cE\subset\cN\times \cN$} and $(i, j) \in \cE$ if there is communication link (edge) between \sa{$i\in\cN$ and $j\in\cN$. Our assumptions on $\cG$ are formally stated below.}
\begin{assumption}\label{assmp:N}
Let $\cG = (\cN, \mathcal{E}) $ denote a connected undirected graph of $N$ computing nodes, where $\cN \triangleq \{1, \cdots, N\}$ and $\cE\subseteq \cN \times \cN$ denotes the set of edges -- without loss of generality assume that $(i,j)\in\cE$ implies $i < j$. Suppose that nodes $i\in\cN$ and $j\in\cN$ can exchange 
\sa{$n$-dimensional data vectors} only if $(i,j)\in\cE$. \sa{Moreover, given arbitrary $\{\eta_i\}_{i\in\cN}\subset\reals$ such that each $\eta_i$ is only known to node $i\in\cN$, suppose that the network is capable of computing $\max_{i\in\cN}\{\eta_i\}$ in such a way that this quantity would be available to all the nodes. Finally, each agent $i\in\cN$ has only access to the agent-specific functions $\varphi_i(\cdot)$, $g_i(\cdot)$, and the cone $\cK_i$.}
\end{assumption}
\begin{definition}
    For $i\in\cN$, let $\cN_i \triangleq \{\sa{j\in\cN}:\ (i,j)\in\cE\ {\rm or\ } (j,i)\in\cE\}$ 
denote the set of neighboring nodes and $d_i\triangleq |\cN_i|$ is its degree, and also let $d_{\rm max} \triangleq \max_{i\in\cN} \{d_i\}$.
\end{definition}

\sa{According to Assumption~\ref{assmp:N}, each agent $i\in\cN$ can communicate (send/receive vector data) only with $j\in\cN_i$, which denotes the set of immediate neighbors of agent $i$. 
We assume that the agents are \textit{collaborative} and the objective is to solve \eqref{eq:opt} in a distributed manner. For this purpose each agent $i\in\cN$ stores/updates its own copy of the decision vector, i.e., $x_i\in\reals^n$ corresponds to agent-$i$; hence, \eqref{eq:opt} can be equivalently formulated as}
\begin{equation}
    \min_{\bx} \sa{\varphi(\bx)\triangleq\phi(\bx)+}f(\bx) \quad {\rm s.t.}\quad \sa{-G(\bx)}\in \cK, \quad 
    A\bx =0,
    \label{pbm-to-solve}
\end{equation}
where $\bx = [x_i]_{i\in\cN} \in \cX\triangleq \reals^{n|\cN|}$ \mgb{denotes the \sa{long} vector obtained by \sa{vertically} concatenating the local decision variables
\sa{$x_i\in\reals^n$} for all agents $i \in \cN$}, \sa{$\phi(\bx)=\sum_{i\in\cN} \phi_i(x_i)$,} $f(\bx)=\sum_{i\in\cN} f_i(x_i)$, \sa{$G(\bx)=[g_i(x_i)]_{i\in\cN}\in\reals^m$ with $m=\sum_{i\in\cN}m_i$, $\cK= \Pi_{i\in\cN}[\cK_i]_{i\in\cN}$ denotes the Cartesian product,} and $A\in \reals^{n|\cE| \times n |\cN|}$ is a block matrix such that $A= H\otimes \mathbf{I}_n$ where \sa{$H\in\reals^{|\cE|\times |\cN|}$} is the oriented edge-node incidence matrix, i.e.,
the entry $H_{(i,j), l}$, corresponding to edge $(i,j)\in\cE$ and $l\in\cN$, is equal to 1 if $l=i$, -1 if $l=j$, and 0 otherwise. Note that $A^\top A = H^\top H \otimes \mathbf{I}_n = \Omega \otimes \mathbf{I}_n$, where $\Omega\in\reals^{|\cN|\times |\cN|}$ denotes the graph
Laplacian of $ \cG$, i.e., $\Omega_{ii}=d_i$, $\Omega_{ij}=-1$ if $(i,j)\in\cE$ or $(j,i)\in\cE$, and equal to 0 otherwise.\looseness=-1

\mgb{We make the following assumptions on \sa{$\{\phi_i\}_{i\in\cN}$, $\{f_i\}_{i\in\cN}$ and $\{g_i\}_{i\in\cN}$} throughout the paper.}
\begin{assumption}\label{assmp:f}
\sa{For all $i\in\cN$, $\phi_i:\reals^n\to\reals\cup\{+\infty\}$ is a proper closed convex function with a compact domain, i.e., there exists $D_i>0$ such that $\norm{x_i}\leq D_{i}$ for all $x_i\in\dom \phi_i$, and $f_i:\reals^n\to\reals$ is a convex function that is differentiable on an open set containing $\dom \phi_i$. Suppose \mgb{there exists a constant \(\sa{L_{f_i}} > 0\)} such that the gradient $\grad f_i$ satisfies} 
$$\norm{\grad f_i(x) - \grad f_i(\bar x)} \leq \sa{L_{f_i}} \norm{x-  \bar x}, \quad \forall x,\bar x\in \sa{\dom \phi_i}. $$
\end{assumption}
\begin{assumption}\label{assmp:g}
\sa{For all $i\in\cN$, $g_i:\reals^n\to\reals^{m_i}$ is a $\cK_i$-convex function that is differentiable on an open set containing $\dom \phi_i$. Suppose \mgb{there exists a constant \(L_{g_i} \geq 0\)} such that the Jacobian $\J g_i:\reals^n\to \reals^{m_i\times n}$ satisfies}
$$ \norm{\sa{\J g_i(x) - \J g_i(\bar x)}} \leq L_{g_i}\norm{x - \bar{ x}}, \quad \forall x,\bar x\in \sa{\dom \phi_i}.$$
\mgb{Furthermore,} \sa{we assume that $g_i$ is Lipschitz on $\dom \phi_i$; hence, there exists  $C_{g_i}>0$ such that $\norm{\J g_i(x)}\leq C_{g_i}$ for all $x\in\dom \phi_i$.} 
Let $L_G\triangleq \max_{i\in\cN} \{L_{g_i}\}$ and $C_G \triangleq \max_{i\in\cN} \{C_{g_i}\}$.
\end{assumption}
\begin{assumption}\label{assmp:bounded_dual_domain}
\sa{A primal-dual optimal pair $(x^*,\theta^*)\in\reals^n\times\cK^*$ for \eqref{eq:opt} exists, where $\theta_i^*\in\cK_i^*$ denotes an optimal dual variable corresponding to $-g_i(x)\in\cK_i$ constraint for $i\in\cN$. For $i\in\cN$ such that $g_i(\cdot)$ is not affine, we assume that agent-$i$ knows a bound $B_i$ such that $2\norm{\theta_i^*}\leq B_i$.}
\end{assumption}
\mgb{Assumption \ref{assmp:f} is standard in the analysis of first-order algorithms for distributed optimization, and has been adopted in many works (e.g., \cite{Jakovetic2018,Pu2020,Scaman2017}).  Assumption \ref{assmp:g} has also commonly appeared in the literature \sa{related to constrained optimization problems with functional constraints \cite{aybat2019distributed,hamedani2021decentralized,hamedani2021primal}}. If the $\dom \phi$ is bounded and the Jacobian is continuous on the closure of the domain, it will be satisfied.} \sa{Assumption~\ref{assmp:bounded_dual_domain} requires existence of a primal-dual solution to \eqref{eq:opt}, which is guaranteed to hold under some mild regularity conditions, e.g., whenever \eqref{eq:opt} admits a Slater point. Moreover, \mgb{when $g_i(\cdot)$ is affine}, i.e., when $g_i(x)=A_i x+b_i$ for some $A_i\in\reals^{m_i\times n}$ and $b_i\in\reals^{m_i}$, the knowledge of a dual bound $B_i$ for $-g_i(x)\in\cK_i$ is not required.
It is essential to emphasize that due to conic \mgb{structure of the} constraints, one can still consider a rich class of nonlinear constraints even \mgb{when} $g_i$ is affine, 
one can model convex quadratic inequality constraints, SOCP and SDP constraints. Furthermore, in \mgb{the following remark}, 
we discuss how the agents can compute dual bounds $B_i$ (as defined in Assumption~\ref{assmp:bounded_dual_domain}) for $\theta_i^*$ corresponding to the constraints with nonlinear $g_i$ in the distributed computation setting we assume in this paper.}
\begin{remark}\label{rem:dual-bound}
Let $\cK=\reals^{N}_+$ and $\phi_i(x)=\mathds{1}_{X}(x)$ for $i\in\cN$, where $X=\{x\in\reals^n:\ \norm{x}\leq r\}$. Consider $(P):\ \varphi^*=\min_x\{ \sum_{i\in\cN}f_i(x):\ g_i(x)\leq 0,\ i\in\cN,\ x\in X\}$. In this remark, we discuss how one can compute a Slater point for $(P)$ in a distributed manner.

Given some small $\epsilon\in (0,r)$, let $X^\circ\triangleq \{x\in\reals^n:\ \norm{x}\leq r-\epsilon\}$. Consider the Phase I problem:
    \begin{equation*}
    \begin{aligned}
        ({\rm Phase~I}):\quad (x^\circ,t^\circ)\in &\argmin_{x,t}\  \mathds{1}_{X^\circ}(x)+t\\
        &\ \ \mbox{s.t.}\ \ g_i(x)\leq t:\ \theta_i,\quad \forall~i\in\cN.
    \end{aligned}
    \end{equation*}
    Note that $\bar x= 0_n$ and $\bar t=\bar g+1$ is a Slater point for ({\rm Phase~I}) where $\bar g\triangleq \max_{i\in\cN} g_i(\bar x)$. Moreover, let $\underline{g}_i=\inf\{g_i(x):\ x\in X\}$ for $i\in\cN$, and $\underline{g}\triangleq \max_{i\in\cN}\underline{g}_i$. Let $\cL(x,t,\theta)$ denote the Lagrangian for ({\rm Phase I}) and $\theta^\circ=[\theta_i^\circ]_{i\in\cN}\in\reals^N_+$ be an optimal dual solution for ({\rm Phase I}). Clearly, we have $\underline{g}\leq t^\circ=\inf_{x,t}\cL(x,t,\theta^\circ)\leq \bar t+\sum_{i\in\cN}\theta_i^\circ(g_i(\bar x)-\bar t)$, which implies that
    \begin{align}
        \bar g-\underline{g}+1\geq \sum_{i\in\cN}\theta_i^\circ\big(\bar t-g_i(\bar x)\big)\geq \sum_{i\in\cN}\theta_i^\circ;\notag
    \end{align}
    therefore, $\theta_i^\circ\geq 0$ such that $\theta_i^\circ\leq \bar g-\underline{g}+1$ for all $i\in\cN$. Moreover, by construction $x^\circ$ is a Slater point for the original problem $(P)$. Note that ({\rm Phase I}) satisfies the assumptions for \alg{}, and we can employ \alg{} to compute a Slater point for $(P)$, a special case of \eqref{eq:opt}, in a distributed manner. Next, using the computed Slater point $x^\circ$, the network of agents can compute a bound on $\theta^*$, i.e., an optimal dual solution of $(P)$, by employing a global max operation twice. Indeed, let $\underline{\varphi}$ be a lower bound on $\varphi^*$ and it is known, e.g., when $f_i$ is a loss function for $i\in\cN$, then one can set $\underline{\varphi}=0$. Then, using a similar Lagrangian argument with above, one can show that $\sum_{i\in\cN}\theta_i^*\leq (\sum_{i\in\cN}f_i(x^\circ)- \underline{\varphi}\mgb{)}/(-\max_{i\in\cN} g_i(x^\circ))\leq (\underline{\varphi}-N\max_{i\in\cN}f_i(x^\circ))/\max_{i\in\cN}g_i(x^\circ)$.
\end{remark}

\mgb{We next provide two definitions before presenting our algorithm and our main results.}

\begin{definition}
    \label{def:dual-sets}
    \sa{Under \cref{assmp:bounded_dual_domain}, for each $i\in\cN$, if $g_i(\cdot)$ is not affine, let $\cB_i\triangleq\{\theta_i\in\reals^{m_i}:\ \norm{\theta_i}\leq B_i\}$; otherwise, if $g_i(\cdot)$ is an affine function, let $\cB_i=\reals^{m_i}$. Define $\cB\triangleq\Pi_{i\in\cN}\cB_i$.}
\end{definition}

\begin{definition}
    \sa{For any $\bx\in\dom \phi$, let $\J G(\bx)\triangleq\diag\left([\J g_i(x_i)]_{i\in\cN}\right)\in\reals^{m\times n|\cN|}$.}
\end{definition}

\subsection{\mgb{Main Results}}
In this paper, we propose \alg{}, displayed in Algorithm \ref{alg:apdb}, for solving the constrained, composite convex consensus optimization problem in \eqref{eq:opt}, and propose \algz{}, displayed in Algorithm \ref{alg:apdbz}, for solving the \textit{unconstrained}$^2$\footnote{\sa{$^2$Here ``unconstrained" means $g_i(\cdot)=0$ for all $i\in\cN$. Any node $i\in\cN$ can still have a simple set constraint incorporated in \eqref{eq:opt} through choosing $\phi_i$ as an indicator function.}} version of the composite convex consensus optimization problem. Given the parameters $\alpha_i^k,\beta_i^k$ and $\tilde\alpha_i^{k+1},\tilde\beta_i^{k+1},\tilde\varsigma_i^{k+1}$ as set in \alg{}, the backtracking test function $E_i^k(\cdot,\cdot)$ in \lin{algeq:test} of \alg{} is defined as follows:
\begin{equation}
\label{eq:line-search-node}
    \begin{aligned}
        E_i^k(x,\theta)\triangleq
        & -\Big(\sa{\frac{1}{\tilde\tau_i^k}}-\eta_i^k(\alpha_i^k + \beta_i^k)-\tilde\varsigma_i^{k+1}\Big) \|x-x_i^k\|^2 - \frac{1}{\tilde\sigma_i^k} \|\theta-\theta_i^k\|^2\\
        &\mbox{} + \sa{\frac{\sa{2}}{\tilde\alpha_i^{k+1}}}\left\|\J g_i (x)^\top(\theta - \theta_i^k)\right\|^2 + \sa{\frac{1}{\tilde\beta_i^{k+1}}}\left\|\left(\J g_i(x) - \J g_i(x_i^k)\right)^\top \sa{\theta_i^k} \right\|^2+2\Lambda_i(x),
    \end{aligned}
\end{equation}
where $\Lambda_i(x)\triangleq f_i(x)-f_i(x_i^k)-\fprod{\grad f_i(x_i^k),~x-x_i^k}$, for all $i\in\cN$– in case $L_{g_i}=0$, setting $\beta_i^k=0$ and $\tilde\beta_i^{k+1}=0,$ we adopt $0^2/0=0;$ hence, we set $\left\|\left(\J g_i(x) - \J g_i(x_i^k)\right)^\top \sa{\theta_i^k} \right\|^2/\tilde\beta_i^{k+1}=0.$

\begin{remark}
    \sa{From convexity of $f_i(\cdot)$, one has $f_i(x)-f_i(x_i^k)\leq \fprod{\grad f_i(x), x-x_i^k}$; therefore, one can bound the last term in the test function as follows: $\Lambda_i(x)\leq \fprod{\grad f_i(x)-\grad f_i(x_i^k), x-x_i^k}$. In the definition of $E_i^k(\cdot,\cdot)$ replacing $\Lambda_i(x)$ with this inner product leads to a stronger condition; that said, in
practice, we have found this condition to be numerically more stable.}
\end{remark}

Now we formally state our main results for \alg{} and \algz{}.

\begin{theorem}\label{thm:main1}
\sa{Suppose that \cref{assmp:f,assmp:g,assmp:N,assmp:bounded_dual_domain} hold, and 
    $\delta,c_\alpha,c_\beta,c_\varsigma>0$ are given such that $\delta+c<1$, where $c\triangleq c_\alpha+c_\beta+c_\varsigma$. \sa{Let $(x^*,\theta^*)\in\reals^n\times\cK^*$ denote an arbitrary primal-dual optimal pair satisfying \cref{assmp:bounded_dual_domain} and $\lambda^*\in\reals^{n|\cE|}$ be an optimal dual variable corresponding to the consensus constraint $A\bx=0$ in \eqref{pbm-to-solve}.}
    For all $i\in\cN$, it holds for all step size parameter values $\bar\tau_i,\zeta_i>0$ that
    when initialized from arbitrary $x_i^0\in\dom\phi_i$ and $\theta_i^0=0_{m_i}$,
    the ergodic iterate sequences $\{(\bar{x}_i^k, \bar{\theta}_i^k)\}_{k\geq 0}$ generated by \alg{}, displayed in Algorithm \ref{alg:apdb}, satisfy}
\begin{align*}
\textbf{(i) Suboptimality:} \quad & |\sum_{i\in\cN}\varphi_i(\bar{x}_i^K) - \sa{\varphi^*}| =\cO\Big(\frac{1}{K}\Big),\\ 
\textbf{(ii) Infeasibility:} \quad & \sum_{i\in\cN}  \norm{\theta_i^*}~d_{\sa{-\cK_i}} \Big(g_i(\bar{x}_i^K)\Big)  + \norm{\lambda^*}~\| A \bar{\bx}^K \|  =\cO\Big(\frac{1}{K}\Big), 
\end{align*}
where $(\bar x_i^K,\bar\theta_i^K)=\sum_{k=0}^{K-1}t_k(x_i^k,\theta_i^k)/\sum_{k=0}^{K-1}t_k$ for $i\in\cN$, \sa{and $\{t_k\}_{k\geq 0}\in\reals_{++}$ is defined recursively such that $t_{k+1}=t_k/\eta^{k+1}$ for $k\geq 0$ and $t_0=1$.}
Moreover, there exists $(x^*,\theta^*)$ a primal-dual optimal solution to \eqref{eq:opt} such that the actual primal-dual iterate sequence $\{(x_i^k,\theta_i^k)\}_{k\geq 0}$ converges \mgb{to} $(x^*,\theta_i^*)$ for all $i\in\cN$, i.e., $\lim_{k\to\infty}x_i^k=x^*$ and $\lim_{k\to\infty}\theta_i^k=\theta_i^*$ for $i\in\cN$.
\end{theorem}

\begin{proof}
    \sa{The proof of this result is given in Section~\ref{sec:main-proofs} in three parts. The first part is on establishing the rate result which is stated in more detail with explicit $\cO(1)$ constants -- see Corollary \ref{coro:subopt}. In the second part, we provide an upper bound on the total number of gradient and projection evaluations required for each $i\in\cN$ to collectively compute an \( \epsilon \)-optimal solution -- see Corollary \ref{cor:complexity}. Finally, in Theorem \ref{thm:limit}, we show that the iterate sequence converges to a primal-dual optimal solution of to \eqref{eq:opt}.}
\end{proof}
\begin{theorem}\label{thm:main2}
\sa{Consider the consensus optimization problem $\sa{\varphi^*}\triangleq \min_{\sa{x\in\reals^n}} \sum_{i\in\cN} \sa{\varphi_i(x)}$, where $\varphi_i(x)\triangleq \phi_i(x)+f_i(x)$ for $x\in\dom \phi_i$. Suppose \cref{assmp:f,assmp:g,assmp:N} hold, and 
    $\delta,c_\alpha,c_\varsigma>0$ are given such that $\delta+c<1$, where $c= c_\alpha+\mgc{c_\beta+} c_\varsigma$ \mgc{with $c_\beta=0$}.}
    \sa{For all $i\in\cN$, it holds for all step size parameter value $\bar\tau_i>0$ that when initialized from arbitrary $x_i^0\in\dom\phi_i$, the ergodic iterate sequence $\{\bar{x}_i^k\}_{k\geq 0}$ generated by \algz{}, displayed in Algorithm \ref{alg:apdbz}, satisfy}
\begin{align*}
|\sum_{i\in\cN}\varphi_i(\bar{x}_i^K) - \sa{\varphi^*}| =\cO\Big(\frac{1}{K}\Big),\qquad \| A \bar{\bx}^K \| =\cO\Big(\frac{1}{K}\Big), 
\end{align*}
where $\bar x_i^K=\sum_{k=0}^{K-1}t_k x_i^k/\sum_{k=0}^{K-1}t_k$ for $i\in\cN$, \sa{and $\{t_k\}_{k\geq 0}\in\reals_{++}$ is defined recursively such that} $t_{k+1}=t_k/\eta^{k+1}$ for $k\geq 0$ and $t_0=1$.
Moreover, there exists an optimal solution $x^*$ such that the actual iterate sequence $\{x_i^k\}_{k\geq 0}$ converges \mgb{to} $x^*$ for all $i\in\cN$, i.e., $\lim_{k\to\infty}x_i^k=x^*$ for $i\in\cN$.
\end{theorem}
\begin{proof}
    This result follows from the proof of \cref{thm:main1} for the case $g_i(\cdot)=0$ for $i\in\cN$.
\end{proof}

{\small
\begin{algorithm}[hbt!]
\caption{Distributed APD with Backtracking (\alg)}
\label{alg:apdb}
\begin{algorithmic}[1]
\State{\bf{Inputs:}} constants $\sa{\delta},c_\alpha,c_\beta,c_{\varsigma},\sa{c_\gamma}>0$: $c_\alpha{+}c_\beta{+}c_{\varsigma}<\sa{1-\delta}$,\quad \sa{$c_\gamma\leq 1/(2|\cE|)$}, \quad \sa{$\rho\in(0,1)$}
\State{\bf{Inputs \sa{for each $i\in\cN$}:}}

initial states \sa{$x_i^0\in\dom\phi_i,\theta_i^0\in\cK_i^*$}; \sa{step size parameters $\sa{\bar\tau_i},\zeta_i>0$}; dual bound $B_i$ if $L_{g_i}>0$\label{algeq:init}
\State \sa{$(x_i^{-1},\theta_i^{-1})\gets (x_i^0,\theta_i^0)$,\quad $s_i^0\gets 0\quad \forall~i\in\cN$}
\State \sa{$\tau_i^{-1}\gets\bar\tau_i,\quad \sigma_i^{-1}\gets\zeta_i\bar\tau_i,\quad \tau_i^0\gets\bar\tau_i,\quad \sigma_i^0\gets\zeta_i\bar\tau_i\quad \forall~i\in\cN$}
\State \sa{$\alpha_i^{0}\gets c_\alpha/\bar\tau_i,\quad \beta_i^{0}\gets c_\beta/\bar\tau_i,\quad \varsigma_i^{0}\gets c_{\varsigma}/\bar\tau_i\quad \forall~i\in\cN$} \label{algeq:abv-0} \Comment{$\beta_i^0=0$ if $L_{g_i}=0$}
\State \sa{$r_i^0\gets{\J g}_i(x_i^0)^\top \theta_i^0+\sum_{j\in\cN_i}(s_i^0-s_j^0),\quad r_i^{-1}\gets r_i^0,\quad\forall~i\in\cN$}
\State \sa{$\bar\tau\gets\max_{i\in\cN}\{\bar\tau_i\}$}
\For{$k=0,1,2,\dots$}
\State $\eta^k\gets 1$
  \ForAll{$i\in\cN$}
  \State $\tilde\tau_i^k\gets \sa{\tau_i^{k-1}}$
    \Loop\Comment{\sa{Backtracking Loop for $i\in\cN$}}
    \State $\tilde\sigma_i^k\gets \zeta_i \tilde\tau_i^k,\quad \eta_i^k\gets\tau_i^{k-1}/\tilde\tau_i^k$ \label{algeq:eta-ik}
    \State $\tilde\alpha_i^{k+1}\gets c_\alpha/\tilde\tau_i^k,\quad \tilde\beta_i^{k+1}\gets c_\beta/\tilde\tau_i^k,\quad \sa{\tilde\varsigma_i^{k+1}}\gets c_{\varsigma}/\tilde\tau_i^k$ \Comment{$\tilde\beta_i^{k+1}=0$ if $L_{g_i}=0$}\label{algeq:abv-tilde}
    \State \sa{$\tilde p_i^k\gets r_i^k+\eta_i^k(r_i^k-r_i^{k-1})$}
    \State $\tilde x_i^{k+1} \gets {\rm prox}_{\tilde\tau_i^k \phi_i}\Big(x_i^k - \tilde\tau_i^k\big(\nabla f_i(x_i^k)+\tilde p_i^k\big)\Big)$ \label{algeq:tilde-x-plus}
    \State $\tilde \theta_i^{k+1} \gets \cP_{\cK_i^*\sa{\cap\cB_i}}\!\left(\theta_i^k + \tilde\sigma_i^k\,g_i(\tilde x_i^{k+1})\right)$ \label{algeq:tilde-theta-plus}
    \If{$E_i^k(\tilde x_i^{k+1},\tilde \theta_i^{k+1})\leq \sa{-\frac{\delta}{\tilde\tau_i^k}\norm{\tilde x_i^{k+1}-x_i^k}^2 - \frac{\delta}{\tilde\sigma_i^k}\norm{\tilde \theta_i^{k+1} -\theta_i^k}^2}$} \Comment{See \eqref{eq:line-search-node} for $E_i^k(\cdot,\cdot)$}
    \label{algeq:test}
    \State \textbf{break} 
    \Else
    \State $\tilde\tau_i^k\gets \sa{\rho} \tilde\tau_i^k$ \Comment{Backtracking for primal step size}
    \EndIf
    \EndLoop
  \EndFor
  \State $\eta^k \gets \max_{i\in\cN} \eta_i^k$ \Comment{max-consensus step} \label{alg:max_consensus}
  \State $\gamma^k\gets
  \sa{\frac{c_\gamma}{\bar\tau}(\frac{2}{c_\alpha} + \frac{\eta^k }{c_{\varsigma}})^{-1}}$
  \label{algeq:gamma-update}
  \ForAll{$i\in\cN$}
  \State $\tau_i^k\gets \tau_i^{k-1}/\eta^k,\quad \sigma_i^k\gets\zeta_i\tau_i^k$ \label{algeq:tau-sigma}
  \State $s_i^{k+1}\gets s_i^k + \gamma^k ((1+\eta^k)x_i^{k} -\eta^k x_i^{k-1})$
  \label{algeq:APD-s}
  \State \sa{$p_i^k\gets r_i^k+\eta^k(r_i^k-r_i^{k-1})$} \label{algeq:APD-p}
  \If{$\eta^k{>1}$}
  {\Comment{$\eta^k>1$: At least one node did backtracking}}
    \State $x_i^{k+1} \gets {\rm prox}_{\tau_i^k \phi_i}\Big(x_i^k - \tau_i^k\big(\nabla f_i(x_i^k)+p_i^k\big)\Big)$ \label{algeq:APD-x}
    \State $\theta_i^{k+1} \gets \cP_{\cK_i^*\sa{\cap\cB_i}}\!\left(\theta_i^k + \sigma_i^k\,g_i(x_i^{k+1})\right)$ \label{algeq:APD-theta}
    \Else
    {\Comment{$\eta^k=1$: no node did backtracking}}
    \State $x_i^{k+1} \gets \tilde x_i^{k+1}$ \label{algeq:APD-x-nb}
    \State $\theta_i^{k+1} \gets \tilde \theta_i^{k+1}$ \label{algeq:APD-theta-nb}
    \EndIf
   \State \sa{$r_i^{k+1}\gets{\J g}_i(x_i^{k+1})^\top \theta_i^{k+1}+\sum_{j\in\cN_i}(s_i^{k+1}-s_j^{k+1})$}
   \State $\alpha_i^{k+1}\gets c_\alpha/\tau_i^k,\quad \beta_i^{k+1}\gets c_\beta/\tau_i^k,\quad \varsigma_i^{k+1}\gets c_{\varsigma}/\tau_i^k$  \Comment{$\beta_i^{k+1}=0$ if $L_{g_i}=0$}\label{algeq:abv}
  \EndFor
  \State $k\gets k+1$
\EndFor
\end{algorithmic}
\end{algorithm}}%
{\small
\begin{algorithm}[hbt!]
\caption{Distributed APD with Backtracking (\algz)}
\label{alg:apdbz}
\begin{algorithmic}[1]
\State{\bf{Inputs:}} constants $\sa{\delta},c_\alpha,c_{\varsigma},\sa{c_\gamma}>0$: $c_\alpha{+}c_{\varsigma}<\sa{1-\delta}$,\quad \sa{$c_\gamma\leq 1/(2|\cE|)$}, \quad \sa{$\rho\in(0,1)$}
\State{\bf{Inputs \sa{for each $i\in\cN$}:}}

initial state \sa{$x_i^0\in\dom\phi_i$}; \sa{step size parameter $\sa{\bar\tau_i}>0$} \label{algzeq:init}
\State \sa{$x_i^{-1}\gets x_i^0$,\quad $s_i^0\gets 0,\quad\tau_i^{-1}\gets\bar\tau_i,\quad \tau_i^0\gets\bar\tau_i\quad \forall~i\in\cN$}
\State \sa{$r_i^0\gets\sum_{j\in\cN_i}(s_i^0-s_j^0),\quad r_i^{-1}\gets r_i^0,\quad\forall~i\in\cN$}
\State \sa{$\bar\tau\gets\max_{i\in\cN}\{\bar\tau_i\}$}
\For{$k=0,1,2,\dots$}
\State $\eta^k\gets 1$
  \ForAll{$i\in\cN$}
  \State $\tilde\tau_i^k\gets \sa{\tau_i^{k-1}}$
    \Loop\Comment{\sa{Backtracking Loop for $i\in\cN$}}
    \State $\eta_i^k\gets\tau_i^{k-1}/\tilde\tau_i^k$ \label{algeq:eta-ik-0}
    \State \sa{$\tilde p_i^k\gets r_i^k+\eta_i^k(r_i^k-r_i^{k-1})$}
    \State $\tilde x_i^{k+1} \gets {\rm prox}_{\tilde\tau_i^k \phi_i}\Big(x_i^k - \tilde\tau_i^k\big(\nabla f_i(x_i^k)+\tilde p_i^k\big)\Big)$ \label{algeq:tilde-x-plus-0}
    \If{$f_i(\tilde x_i^{k+1})-f_i(x_i^k)-\fprod{\grad f_i(x_i^k),~\tilde x_i^{k+1}-x_i^k}\leq\frac{1}{2\tilde\tau_i^k}\Big(1-\delta-c_\alpha-c_\varsigma\Big)\norm{\tilde x_i^{k+1}-x_i^k}^2$}
    \label{algeq:test-0}
    \State \textbf{break} 
    \Else
    \State $\tilde\tau_i^k\gets \sa{\rho} \tilde\tau_i^k$ \Comment{Backtracking for primal step size}
    \EndIf
    \EndLoop
  \EndFor
  \State $\eta^k \gets \max_{i\in\cN} \eta_i^k$ \Comment{max-consensus step} \label{alg:max_consensus-0}
  \State $\gamma^k\gets
  \sa{\frac{c_\gamma}{\bar\tau}(\frac{2}{c_\alpha} + \frac{\eta^k }{c_{\varsigma}})^{-1}}$
  \label{algeq:gamma-update-0}
  \ForAll{$i\in\cN$}
  \State $\tau_i^k\gets \tau_i^{k-1}/\eta^k$ \label{algeq:tau-sigma-0}
  \State $s_i^{k+1}\gets s_i^k + \gamma^k ((1+\eta^k)x_i^{k} -\eta^k x_i^{k-1})$
  \label{algeq:APD-s-0}
  \State \sa{$p_i^k\gets r_i^k+\eta^k(r_i^k-r_i^{k-1})$} \label{algeq:APD-p-0}
  \If{$\eta^k{>1}$}
  {\Comment{$\eta^k>1$: At least one node did backtracking}}
    \State $x_i^{k+1} \gets {\rm prox}_{\tau_i^k \phi_i}\Big(x_i^k - \tau_i^k\big(\nabla f_i(x_i^k)+p_i^k\big)\Big)$ \label{algeq:APD-x-0}
    \Else
    {\Comment{$\eta^k=1$: no node did backtracking}}
    \State $x_i^{k+1} \gets \tilde x_i^{k+1}$ \label{algeq:APD-x-nb-0}
    \EndIf
   \State \sa{$r_i^{k+1}\gets\sum_{j\in\cN_i}(s_i^{k+1}-s_j^{k+1})$}
  \EndFor
  \State $k\gets k+1$
\EndFor
\end{algorithmic}
\end{algorithm}}

\section{\mgb{Motivation and Design of the Proposed Algorithms: \alg{} and \algz{}}}
\sa{Under \cref{assmp:bounded_dual_domain}, \mgb{the} constrained convex optimization problem in~\mgb{\eqref{pbm-to-solve}} can be equivalently written as a minimax problem through the use of Lagrangian duality:}
\begin{equation}
    \min_{\bx}\max_{\sa{\theta,\lambda}}\cL(\bx,\theta,\lambda) \triangleq
    \sa{\varphi(\bx)+ \langle \theta, G(\bx) \rangle} + \langle \lambda, A\bx \rangle \sa{-h(\theta)},
    \label{saddle-point}
\end{equation}
\sa{where $h:\reals^m\to\reals\cup\{+\infty\}$ denotes the indicator function of
$\cK^*\cap\cB$, i.e., $h(\theta)=0$ if $\theta\in\cK^*\cap\cB$, and $+\infty$ otherwise, i.e., $h(\theta)=\sum_{i\in\cN}\mathds{1}_{\cK_i^*\cap\cB_i}(\theta_i)$.}

\subsection{\sa{Preliminaries}}
We briefly present a related previous work ~\cite{hamedani2021primal}, where we proposed an accelerated primal-dual (APD) algorithm
for solving convex-concave saddle-point (SP) problems. As it is shown in~\cite{hamedani2021primal}, APD can be viewed as an extension
of the primal-dual algorithm proposed in~\cite{chambolle2016ergodic} \sa{from 
bilinear SP problems to a more general setting} with a
non-bilinear coupling term. Let $\mathcal{X}\subset\mathbb{R}^{n_x}$ and $\mathcal{Y}\subset\mathbb{R}^{n_y}$ be finite-dimensional vector spaces. Here we present a slightly 
\sa{modified} version of APD \sa{proposed in~\cite{hamedani2021decentralized}} to solve the following problem:
\begin{equation}
\label{eq:sp}
{\min_{\bx\in\mathcal{X}}\ \max_{\by\in\mathcal{Y}}\
\mathcal{L}(\bx,\by)\triangleq \big(\psi_x+q_x\big)(\bx)\;+\cH(\bx,\by)\;-\;\big(\psi_y+q_y\big)(\by),}
\end{equation}
where $\psi_x,\psi_y$ are possibly non-smooth convex functions, and $q_x,q_y$ are convex and differentiable on open sets containing $\operatorname{dom}\psi_x$ and $\operatorname{dom}\psi_y$ satisfying a descent property governed by
$L_{q_x}\in\mathbb{R}^{n_x\times n_x}$ and $L_{q_y}\in\mathbb{R}^{n_y\times n_y}$, i.e., \sa{$q_x(\bx')\leq q_x(\bx)+\fprod{\grad q_x(\bx),~\bx'-\bx}+\frac{1}{2}(\bx'-\bx)^\top L_{q_x} (\bx'-\bx)$ for all $\bx',\bx\in\dom\psi_x$,
and similar inequality also holds for $q_y$.} The coupling $\cH:\mathcal{X}\times\mathcal{Y}\to\mathbb{R}$ is a continuously differentiable function that is convex in $\bx$ for any $\by\in\mathcal{Y}$, and concave in $\by$ for any $\bx\in\mathcal{X}$. 
Moreover, for any $\bx\in\mathcal{X}$, $ \nabla_y \cH(\bx,\cdot)$ is $L_{yy}$-Lipschitz \sa{for some $L_{yy}\in\reals_+$} and $\nabla_x \cH(\bx,\cdot)$ is $\reals_+^{p_{xy}\times n_y}\ni L_{xy}$-Lipschitz for some $p_{xy}\ge 1$, \sa{i.e., $\norm{\nabla_x \cH(\bx,\by')-\nabla_x \cH(\bx,\by)}\leq \norm{L_{xy}(\by'-\by)}$ for all $\by,\by'\in\cY$}; for any $\by\in\mathcal{Y}$, $ \nabla_x \cH(\cdot,\by)$ is $L_{xx}$-Lipschitz \sa{for some $L_{xx}\in\reals_+$} and $\nabla_y \cH(\cdot,\by)$ is $\reals_+^{p_{yx}\times n_x} \ni L_{yx}$-Lipschitz for some $p_{yx}\ge 1$.
\sa{Given $Q_x,Q_y\in\mathbb{S}^{n_x}_{++}\times \mathbb{S}^{n_y}_{++}$} and initial iterates $\bx^0\in\mathcal{X}$ and $\by^0\in\mathcal{Y}$, the slightly modified version of APD iterations \mgb{consist}
of the following updates:
\sa{
\begin{subequations}\label{eq:apd}
\begin{align}
&\bp^k \gets\; 2\,\nabla_y \cH(\bx^k,\by^k)\;-\;\nabla_y \cH(\bx^{k-1},\by^{k-1}), \label{eq:pk}\\
& \by^{k+1} \gets\; \arg\min_{\by\in\mathcal{Y}}\ \psi_y(\by)\;+\;\big\langle \nabla q_y(\by^k)-\bp^k,\; \by \big\rangle\;+\;\tfrac12\|\by-\by^k\|_{Q_y}^2, \label{eq:yupdate}\\
&\bx^{k+1} \gets\; \arg\min_{\bx\in\mathcal{X}}\ \psi_x(\bx)\;+\;\big\langle \nabla q_x(\bx^k)+\nabla_x \cH(\bx^k,\by^{k+1}),\; \bx \big\rangle\;+\;\tfrac12\|\bx-\bx^k\|_{Q_x}^2, \label{eq:xupdate}
\end{align}
\end{subequations}
where \sa{$\norm{\bx}_{Q_x}=\sqrt{\bx^\top Q_x\bx}$ for $\bx\in\cX$ and $\norm{\by}_{Q_y}=\sqrt{\by^\top Q_y\by}$ for $\by\in\cY$.} Based on the discussion in~\cite{hamedani2021primal}, if \sa{$Q_x$ and $Q_y$ are chosen such that there exist some $c\geq 1$, 
and $\alpha,\beta\geq 0$ satisfying}
\begin{align}
\label{eq:conds}
Q_x \;-\; L_{q_x} \;\succeq\; L_{xx}\,I_{n_x} \;+\; \frac{1}{\alpha}\,L_{yx}^\top L_{yx}, 
\qquad
Q_y \;-\; L_{q_y} \;\succeq\; c\,(\alpha+\beta)\,I_{n_y} \;+\; \frac{1}{\beta}\,L_{yy}^2\,I_{n_y},
\end{align}
-- see~\cite[Assumption 3]{hamedani2021primal} with $\delta=0$, then according to \cite[Eq. (3.3)]{hamedani2021primal} and \cite[Lemma~3.4]{hamedani2021primal}, it holds for all $K\geq 0$ that
\begin{align}
\label{eq:APD-rate}
0\leq \cL(\bar\bx^K,\by)-\cL(\bx,\bar\by^K)\leq\frac{1}{2K}\left(\norm{\bx-\bx^0}^2_{Q_x}+\norm{\by-\by^0}^2_{Q_y}\right),\quad \forall~\bx\in\cX,\ \by\in\cY,
\end{align}
where $(\bar\bx^K,\bar\by^K)=\frac{1}{K}\sum_{k=1}^K (\bx^k,\by^k)$.}

\sa{Our objective is to design a momentum-based primal-dual method built on APD framework in~\eqref{eq:apd} for the saddle-point formulation in \eqref{saddle-point}, which can be seen as a special case of \eqref{eq:sp} with a very particular structure: $\by=[\theta^\top \lambda^\top]^\top$, $\psi_x(\bx)=\phi(\bx)$, $q_x(\bx)=f(\bx)$, $\cH(\bx,\by)=\fprod{\theta, G(\bx)}+\fprod{\lambda, A\bx}$, $\psi_y(\by)=h(\theta)$ and $q_y(\by)=0$. Therefore, we will focus on a variant of the primal-dual iterations stated in \eqref{eq:apd} with two goals in mind: \textbf{(i)} we prefer using momentum acceleration for the primal updates rather than having momentum term in dual updates as in \eqref{eq:apd}; \textbf{(ii)} \sa{one should be able to adaptively select node-specific step sizes
via adopting a 
backtracking scheme, without relying on some prior knowledge on global Lipschitz constants.}}

\sa{The first goal is mainly motivated by the primal-dual dynamics for the particular saddle point formulation we focus on. Indeed, $\cH(\bx,\by)$ has a curvature in $\bx$, while it is affine in $\by$, and that is why using momentum in $\bx$-updates helps dampening the oscillatory behavior that naturally arises due to primal-descent-dual-ascent-type updates, e.g., see~\cite{schafer2019competitive,zheng2024dissipative} for the oscillatory behavior of gradient-descent-ascent updates; on the other hand, the use of momentum in $\by$-update as in \eqref{eq:apd} may even exacerbate the oscillations as $\grad_\by\cH(\bx,\by)$ only depends on $\bx$, i.e., if $\{\bx^k\}$ oscillates, not only there is no damping mechanism for $\bx$-updates in \eqref{eq:apd}, but oscillations in $\{\bx^k\}$ also causes $\{\by^k\}$ to oscillate even more.}

\sa{The first goal requires a change in the update order of \eqref{eq:apd}, which can be justified assuming \eqref{eq:sp} has a saddle point. Indeed, applying the extended version of APD given in \eqref{eq:apd} on an equivalent problem $\min_{\by\in\mathcal{Y}}\ \max_{\bx\in\mathcal{X}}\ -\mathcal{L}(\bx,\by)$ with $q_y(\cdot)=0$ would lead to the following iterations with the desired momentum acceleration on the primal updates:
\begin{subequations}\label{eq:apd-swap}
\begin{align}
&\bp^k \gets\; (1+\eta^k)\,\nabla_x \cH(\bx^k,\by^k)\;-\eta^k\;\nabla_x \cH(\bx^{k-1},\by^{k-1}), \label{eq:pk-2}\\
& \bx^{k+1} \gets\; \arg\min_{\bx\in\mathcal{X}}\ \psi_x(\bx)\;+\;\big\langle \nabla q_x(\bx^k)+\bp^k,\; \bx \big\rangle\;+\;\tfrac12\|\bx-\bx^k\|_{Q_x^k}^2, \label{eq:xupdate-2}\\
&\by^{k+1} \gets\; \arg\min_{\by\in\mathcal{Y}}\ \psi_y(\by)\;-\;\big\langle 
\nabla_y \cH(\bx^{k+1},\by^{k}),\; \by \big\rangle\;+\;\tfrac12\|\by-\by^k\|_{Q_y^k}^2, \label{eq:yupdate-2}
\end{align}
\end{subequations}
with $Q_x^k=Q_x$, $Q_y^k=Q_y$, and $\eta^k=1$ for all $k\geq 0$. A convergence rate result similar to \eqref{eq:APD-rate} continues to hold for this
variant with $Q_x$ and $Q_y$ chosen slightly different than \eqref{eq:conds} due to primal-dual switch:
\begin{align}
Q_y 
\;\succeq\; L_{yy}\,I_{n_y} \;+\; \frac{1}{\alpha}\,L_{xy}^\top L_{xy}, 
\qquad
Q_x \;-\; L_{q_x} \;\succeq\; c\,(\alpha+\beta)\,I_{n_x} \;+\; \frac{1}{\beta}\,L_{xx}^2\,I_{n_x}, \notag
\end{align}
for some $c\geq 1$ and $\alpha,\beta\geq 0$. Our motivation for giving \eqref{eq:apd-swap} with iteration dependent step-size matrices $Q_x^k, Q_y^k$ and momentum parameter $\eta^k$ is to set the grounds for a setting that will allow us to analyze a variant of APD method employing node-specific local step-size search. In the rest, we use the notation $\bx^k = [x_i^k]_{i\in\cN}$ for $k\geq 0$.}

{Towards the second goal stated above, 
i.e., to ensure that the nodes can 
adaptively select their primal and dual step sizes at each iteration through backtracking, we still need to slightly modify the updates given in \eqref{eq:apd-swap} since the $\lambda$-update in \eqref{eq:apd-swap}, related to the dual variable vector $\lambda$ corresponding to the consensus constraint $A\bx=0$, is \textit{not} suitable for a distributed local step size search mechanism. Indeed, given some $\tau_i^k,\sigma_i^k>0$ for $i\in\cN$ and $\gamma^k>0$ for $k\geq 0$, let
\begin{equation*}
\label{eq:APD-step-sizes}
    Q_x^k=\bD_{\tau^k},\qquad
    Q_y^k=
\begin{pmatrix}
    \bD_{\sigma^k} & 0\\
    0 & \frac{1}{\gamma^k} I_{n|\cE|}
\end{pmatrix},
\end{equation*}
where $\bD_{\tau^k}=\diag\left(\left[\frac{1}{\tau_i^k}\;I_n\right]_{i\in\cN}\right)$, $\bD_{\sigma^k}=\diag\left(\left[\frac{1}{\sigma_i^k}\;I_{m_i}\right]_{i\in\cN}\right)$. For any iteration $k\geq 0$, after a candidate point $(\bx^{k+1},\by^{k+1})$ for the next iteration is computed according to \eqref{eq:apd-swap} using some given step sizes $\{\tau_i^k,\sigma_i^k\}_{i\in\cN}$ and $\gamma^k$, one needs to check a backtracking condition to ensure that $(\bx^{k+1},\by^{k+1})$ provides a sufficient decrease in a suitably defined potential function, where $\by^{k+1}=(\theta^{k+1},\lambda^{k+1})$ --if the condition holds, we accept $(\bx^{k+1},\by^{k+1})$; otherwise, we decrease the given step sizes and compute a new candidate point. In this procedure, one needs to compute a term involving the candidate dual point $\lambda^{k+1}$ every time the condition is checked, e.g., $\norm{\lambda^{k+1}-\lambda^k}^2$.
However, the update rule in~\eqref{eq:apd-swap} implies that $\lambda^{k+1}=\lambda^k+\gamma^k A\bx^{k+1}$, which is a function of $\bx^{k+1}$ for $k\geq 0$; hence, checking a test function involving $\lambda^{k+1}$ is not suitable for distributed computation. Indeed, as an example, we consider the quantity $\norm{\lambda^{k+1}-\lambda^k}^2$ which can be computed as below:}
\begin{equation}\label{eq:lamba_bd}
\begin{aligned}
    \norm{\lambda^{k+1}-\lambda^k}^2 = &(\gamma^k)^2 \langle A\bx^{k+1} , A\bx^{k+1} \rangle
    = (\gamma^k)^2 \langle  \bx^{k+1}, (\Omega \otimes \bI_n) \bx^{k+1}\rangle\\
    =  &(\gamma^k)^2 \sum_{i\in\cN} \langle x_i^{k+1}, \sum_{j\in\cN_i} (x_i^{k+1} - x_j^{k+1}) \rangle.
\end{aligned}
\end{equation}
\sa{According to \eqref{eq:lamba_bd}, at iteration $k\geq 0$, for each agent $i\in\cN$, checking its \textit{local} backtracking condition involves computing the term $\langle x_i^{k+1}, \sum_{j\in\cN_i} (x_i^{k+1} - x_j^{k+1}) \rangle$, which requires the decision vectors $x_j^{k+1}$ of its neighbors $j\in\cN_i$ that satisfy their own backtracking conditions. Therefore, if one adopts the update scheme in \eqref{eq:apd-swap}, it is not trivial or practical to compute $\{x_i^{k+1}\}_{j\in\cN}$ in a distributed manner that satisfy all local backtracking conditions simultaneously. To avoid this issue, we slightly change the update order as stated in \eqref{eq:update_unify}, i.e., at iteration $k\geq 0$, consensus related dual variable $\lambda$ is updated first using the information related to $\bx^k$ and $\bx^{k-1}$, then the $\bx$-update is executed, and finally the dual variable $\theta_i$ related to the local constraint $-g_i(x_i)\in\cK_i$ is updated for all $i\in\cN$. Hence, for the variant of APD in \eqref{eq:update_unify}, the term $\norm{\lambda^{k+1}-\lambda^k}^2$ only relies on the points $\{x_i^{k}\}_{i\in\cN}$ that satisfy the local backtracking conditions at iteration $k-1$. It should be emphasized that as an artifact of the proof technique we adopted, rather than setting $\lambda^{k+1}=\lambda^k+\gamma^k A\bx^k$, we set $\lambda^{k+1}=\lambda^k+\gamma^k A((1+\eta^k)\bx^k-\eta^k \bx^{k-1})$ for some properly chosen $\eta^k\geq 1$, i.e., we also use momentum term in $\lambda$-update as well.$^3$}{\footnote{\sa{$^3$Later in \cref{lem:one-step}, we establish how duality gap changes after one iteration of \alg. This choice of $\lambda$-update leads to a telescoping sum for the duality gap bound in \eqref{eq:Lagrangian-diff} by ensuring that $t_{k+1}Q^{k+1}(\bz) - t_k R^{k+1}(\bz)\leq 0$ holds for all $k\geq 0$ and $\bz\in\cX\times\cY$.}}}

\sa{Given some arbitrary initial points $\bx^0\in\dom\phi$, $\theta^0\in\cK^*$, $\lambda^0\in\reals^{n|\cE|}$, step sizes $\tau_i^k,\sigma_i^k>0$ for $i\in\cN$, $\gamma^k>0$ and the momentum parameter $\eta^k > 0$ for all $k\geq 0$, we set $\bx^{-1}=\bx^0$, $\theta^{-1}=\theta^0$, $\lambda^{-1}=\lambda^0$, and} \mgb{we consider the following iterations for $k\geq 0$:}
\begin{subequations}\label{eq:update_unify}
    \begin{align}
        & \lambda^{k+1} \gets \argmin_\lambda -\left\langle A\Big((1+\eta^k)\bx^{k} - \eta^k \bx^{k-1}\Big), \lambda \right\rangle +\frac{1}{2\gamma^k}\norm{\lambda - \lambda^k}^2\,, \label{eq:lambda-update} \\
        &\bp^k \gets A^\top \Big((1+\eta^k)\lambda^k - \eta^k \lambda^{k-1}\Big) +(1+\eta^k)\J G(\bx^k)^\top \theta^k -\eta^k \J G(\bx^{k-1})^\top \theta^{k-1}\,,\label{eq:p-update}\\
        & \bx^{k+1} \gets \argmin_{\bx}   \sum_{i\in\cN} \Big[\sa{\phi_i(x_i)+} \langle \grad f(x_i^k) + p_i^k, x_i \rangle   + \frac{1}{2\tau_i^k}\norm{x_i - x_i^k}^2\Big]\,,\label{eq:x-update}\\
        &\theta_i^{k+1} \gets \argmin_{\theta_i \in\cK_i^*\sa{\cap\cB_i}} - \langle \theta_i, g_i(x_i^{k+1}) \rangle + \frac{1}{2\sigma_i^k}\norm{\theta_i - \theta_i^k}^2,\quad\forall~i\in\cN.\label{eq:theta-update}
    \end{align}
\end{subequations}
Here, update \eqref{eq:lambda-update} enforces the 
\sa{consensus} constraint \sa{$A\bx=0$} 
via an 
ascent step;
\eqref{eq:p-update} forms a momentum/correction term using past dual/primal states;
\eqref{eq:x-update} performs a proximal gradient-type step on each block $x_i$ whereas
\eqref{eq:theta-update} is a projected ascent \sa{ensuring $\theta_i^{k+1}\in \cK_i^\ast$}. \sa{At this point, we have an abstract algorithm as described in \eqref{eq:update_unify} with momentum acceleration in $\bx$-updates to exploit the curvature of 
$G(\bx)$, and it is in \mgb{a} suitable form to design a backtracking condition for it. Next, we discuss how this abstract update rule can be implemented in a distributed manner using only local communications for $n$-dimensional data vectors.}

\subsection{Distributed implementation of the abstract method}

\sa{Since we set $\bx^0=\bx^{-1}$ and \eqref{eq:lambda-update} implies that $\lambda^{k+1} = \lambda^k + \gamma^k A\Big((1+\eta^k)\bx^{k} - \eta^k \bx^{k-1}\Big)$ for all $k\geq 0$, we get
$$\lambda^{k+1} = 
\sum_{\ell=0}^{k} \gamma^\ell A\Big((1+\eta^\ell)\bx^{\ell} - \eta^\ell \bx^{\ell-1}\Big),\quad\forall~k\geq 0,$$
and we choose $\lambda^0 \triangleq 0$. }%
To implement these steps efficiently using only 
local communications among the neighboring nodes, we eliminate $\{\lambda^k\}$ via an auxiliary consensus state sequence $\{\bs^k\}_{k\geq 0}$ such that
$\bs^0 \triangleq 0$ and {$\bs^{k+1} = \bs^k + \gamma^k(1+\eta^k)\bx^{k} - \gamma^k\eta^k \bx^{k-1}$} for all $k\geq 0$; hence, $\lambda^k = A\bs^k$ for all $k\geq0$.
Using $A^\top A= \Omega \otimes \bI_n$, we obtain that
\begin{equation*}
    \langle A^\top\lambda^k, \bx \rangle = \langle \bx, (\Omega \otimes \bI_n) \bs^k \rangle = \sum_{i\in\cN} \langle x_i, \sum_{j\in\cN_i} (s_i^k - s_j^k) \rangle,\quad\sa{\forall\bx.}
\end{equation*}
\sa{Therefore,} 
\eqref{eq:update_unify} \sa{can be implemented in a distributed manner, i.e., each node $i\in\cN$, starting from arbitrary $x_i^0\in\dom\phi_i$ and $\theta_i^0\in\cK_i^*$, initializes $s_i^0=s_i^{-1}=0$, $x_i^{-1}=x_i^0$ and $\theta_i^{-1}=\theta_i^0$,
\sa{sets $r_i^{-1}=\J g_i(x_i^0)^{\!\top}\theta_i^0$,} and computes the following updates$^4$:\footnote{$^4$Our initialization implies that $p_i^0=r_i^0=\J g_i(x_i^0)^{\!\top}\theta_i^0$ for all $i\in\cN$.}}
\begin{subequations}
\label{eq:local-updates}
    \begin{align}
        s_i^{k+1} \gets & s_i^k + \gamma^k \Big((1+\eta^k)x_i^{k} -\eta^k x_i^{k-1}\Big) \label{eq:s-local}\\
        \sa{r_i^{k}}\gets & \sa{{\J g}_i(x_i^{k})^\top \theta_i^{k}+\sum_{j\in\cN_i}(s_i^{k}-s_j^{k})}\\
        p_i^k \gets & \sa{r_i^k+\eta^k(r_i^k-r_i^{k-1})}\label{eq:p-local}\\
        x_i^{k+1} \gets & {\rm prox}_{\tau_i^k \phi_i} \Big(x_i^k - \tau_i^k (\grad f_i(x_i^k) + p_i^k) \Big) \label{eq:x-local}\\
        \theta_i^{k+1} \gets & \sa{\cP_{\cK_i^*\cap\cB_i}}\Big(\theta_i^k + \sigma_i^k g_i(x_i^{k+1})\Big), \label{eq:theta-local}
    \end{align}
\end{subequations}
for all $k\geq 0$, \mgb{where} the \sa{proximal map,} ${\rm prox}_{\tau \phi_i}(\cdot)$, is defined \sa{for any $\tau>0$ and $x\in\reals^n$ as follows:}
\begin{equation*}
    {\rm prox}_{\tau \phi_i}(x)\triangleq \arg\min_{w\in\reals^n} \Big\{ \tau \phi_i(w) + \frac{1}{2}\norm{w-x}^2\Big\}.
\end{equation*}
Thus each node $i\in\cN$ updates its local variables $(x_i,\theta_i,s_i)$ using only local gradients/Jacobians and communicating only with the neighboring nodes in $\mathcal N_i$.

\mgc{The accelerated primal-dual algorithm with backtracking (\texttt{ADPB}) introduced in \cite{hamedani2021primal} is a version of ADP that supports backtracking; however it was designed for centralized minimax problems.} \sa{Our goal in this paper is to design a \textit{distributed} stepsize selection mechanism for the abstract method given in~\eqref{eq:local-updates} so that we can extend \texttt{ADPB}~\cite{hamedani2021primal} from centralized to the decentralized setting. The \textit{distributed} stepsize selection procedure we propose is based on local backtracking and it does not require a priori knowledge of global Lipschitz constants.}

\subsection{From the abstract method towards \alg{} in Algorithm \ref{alg:apdb}}
\sa{Here
we propose a step-size choice mechanism for the abstract distributed update scheme in \eqref{eq:local-updates}, which would lead to \alg{} displayed in Algorithm~\ref{alg:apdb}. Our goal in each iteration $k\geq 0$ is to employ an Armijo-type local backtracking condition for determining local step-sizes $\tau_i^k,\sigma_i^k>0$ for each $i\in\cN$, and to use a global-max consensus across the network to coordinate the momentum parameter $\eta^k > 0$.}

At each iteration $k\geq 0$, every node $i\in\cN$ calls a backtracking subroutine to compute candidate step sizes and a momentum parameter: $(\tilde{\tau}_i^k,\tilde\sigma_i^k,\eta_i^k)$.
Within the backtracking subroutine, each node $i\in\cN$, employs Armijo-type step-size search to determine $\tilde{\tau}_i^k$ using only local information --for the sake of notational simplicity, in the coming argument we do not explicitly show the iteration index $k\geq 0$. For each node $i\in\cN$, given the current primal-dual iterate $(x_i,\theta_i)$, \sa{and the gradient terms related to the agent-specific constraint function $g_i$ and consensus violation for the current and the previous iterations, i.e.,} $r_i$ and $r_i^{-}$, the trial primal stepsize $\tilde{\tau}_i$ is initialized using the step size of the previous iteration $\tau_i^{-}$, i.e., $\tilde{\tau}_i \gets \tau_i^{-}$, and node-$i$ computes the candidate primal-dual iterate $(\tilde x_i,\tilde\theta_i)$ as follows:
\begin{align}
\tilde p_i &\gets r_i+\frac{\tau_i^{-}}{\tilde\tau_i}(r_i-r_i^{-}), \notag\\
\tilde{x}_i &\gets \operatorname{prox}_{\tilde{\tau}_i\phi_i}\!\Big(x_i-\tilde{\tau}_i\big(\nabla f_i(x_i)+\tilde{p}_i\big)\Big), \notag\\
\tilde{\theta}_i &\gets \cP_{K_i^*\cap \cB_i}\!\big(\theta_i+\tilde{\sigma}_i\,g_i(\tilde{x}_i)\big), \notag
\end{align}
where $\rho\in(0,1)$ is a fixed contraction factor, $\delta\in(0,1)$ is a backtracking parameter, and $\tilde{p}_i$ denotes a local momentum term mimicking the update in~\eqref{eq:p-local}. \sa{We adopted an Armijo-type backtracking condition based on} a local merit function \sa{$E_i$}
as defined in \eqref{eq:line-search-node}, where for given trial primal stepsize $\tilde\tau_i$, the parameters within $E_i$ depend  only on $\tilde\tau_i$ and $\tau_i^{-}$, i.e.,
\begin{equation*}
    \tilde{\alpha}_i^{k+1} = \frac{c_\alpha}{\tilde{\tau}_i^k},\quad \tilde{\beta}_i^{k+1} = \frac{c_\beta}{\tilde{\tau}_i^k},\quad \tilde{\varsigma}_i^{k+1}=\frac{c_{\varsigma}}{\tilde{\tau}_i^k},\quad\tilde{\sigma}_i^k=\zeta_i \tilde{\tau}_i^k, \quad\eta_i^k = \frac{\tau_i^{k-1}}{\tilde{\tau}_i^k},
\end{equation*}
are all derived from $\tilde{\tau}_i^k$ (with fixed $c_\alpha,c_\beta,c_{\varsigma}>0$ and $\zeta_i>0$), while we set $\alpha_i^k=c_\alpha/\tau_i^{k-1}$ and $\beta_i=c_\beta/\tau_i^{k-1}$.
The merit function $E_i$ is a quadratic surrogate tailored to the primal--dual structure, and the trial pair $\big(\tilde{\tau}_i,\tilde{\sigma}_i\big)$ is \emph{accepted} if the condition,
\sa{
\begin{equation}
E_i(\tilde{x}_i,\tilde{\theta}_i) \;\le\; -\;\frac{\delta}{2\,\tilde{\tau}_i}\|\tilde{x}_i-x_i\|^2\;-\;\frac{\delta}{2\,\tilde{\sigma}_i}\|\tilde{\theta}_i-\theta_i\|^2,
\label{eq:armijo}
\end{equation}}%
holds; otherwise, both step sizes are shrunk, \sa{i.e.,} 
\(
\tilde{\tau}_i\gets\rho\,\tilde{\tau}_i,\ \tilde{\sigma}_i\gets\rho\,\tilde{\sigma}_i,
\)
and the loop repeats \sa{until \eqref{eq:armijo} holds eventually}. Upon the acceptance \sa{of the trial pair $\big(\tilde{\tau}_i,\tilde{\sigma}_i\big)$}, the routine returns the current step $\tilde{\tau}_i$ together with \sa{$\eta_i=\frac{\tau^{-}_i}{\tilde{\tau}_i}$.}
Condition~\eqref{eq:armijo} enforces a decrease proportional to the ``step energy'' $\|\tilde{x}_i-x_i\|^2$ and $\|\tilde{\theta}_i-\theta_i\|^2$ in a model that captures the curvature of $f_i$, and the local constraint sensitivity via $\J g_i$.
As $\tilde{\sigma}_i$ and $\tilde{\tau}_i$ are coupled through some $\zeta_i>0$, the primal and dual step sizes always remain properly scaled. \sa{More precisely, once $\tilde\tau_i^k$ is accepted, node-$i$ sets $\tilde\sigma_i^k=\zeta_i\tilde\tau_i^k$ and $\eta_i^k=\frac{\tau_i^{k-1}}{\tilde\tau_i^k}$; and next,}
\alg~sets $\eta^{k}$ via
\mgb{\begin{equation*}
\eta^{k}\;=\;\max_{i\in\mathcal N}\eta_i^k.
\label{eq:eta-policy}\end{equation*}}\mgb{This step requires a \emph{global max consensus}, i.e., the cooperation of the nodes to compute a maximum over the node variables $\eta_i^k$. It can be typically implemented with protocols such as \sa{LoRaWAN} (low power, long range) which supports long-distance low-power wireless communication between the nodes \cite{chen2025parameter,kim2016low}}. Then, \mgb{\alg}~derives
\(
\tau_i^k = {\tau_i^{k-1}}/{\eta^k},\
\alpha_i^{k+1}=c_\alpha/\tau_i^k,\
\beta_i^{k+1}=c_\beta/\tau_i^k,\
\varsigma_i^{k+1}=c_{\varsigma}/\tau_i^k,\
\sigma_i^{k+1}=\zeta_i\tau_i^k,
\) and chooses $\gamma_k$ via
\mgb{\begin{equation*}
\gamma_k \;=\; \frac{c_\gamma}{\bar\tau}(\frac{2}{c_\alpha} + \frac{\eta^k }{c_{\varsigma}})^{-1}\!,
\label{eq:gamma-policy}
\end{equation*}
with $\bar\tau\triangleq\max_{i\in\cN}\{\bar\tau_i\}$ and \sa{for some $c_\gamma>0$ such that} $c_\gamma \leq 1/(2 |\cE|)$, to update the $s_i^k$ variable.} \mgb{With this information at hand, the subsequent updates
consist of updating the node-based $p_i^k$, \sa{$r_i^k$} and $x_i^k$ variables as outlined in Algorithm~\ref{alg:apdb}. Basically, these updates correspond to a fully decentralized implementation of \eqref{eq:local-updates} \sa{--more precisely, Algorithm~\ref{alg:apdb} describes how primal-dual stepsize sequence $\{\tau_i^k,\sigma_i^k\}_{k\geq 0}$ for $i\in\cN$ and the momentum parameter sequence $\{\eta^k\}_{k\geq 0}$ defining the update rule in \eqref{eq:local-updates} can be chosen through employing node-specific backtracking and running the max-consensus protocol across the network once per iteration $k\geq 0$.}}

\section{\sa{Preliminary Technical Results}}
Now, we are ready to give our main \sa{technical} results. We will derive some key inequalities below for 
\alg{} iterates $\{\bx_k,\lambda_k,\theta_k\}_{k\geq 0}$ generated
by
Algorithm~\ref{alg:apdb}. \sa{Let $\by = [\theta^\top, \lambda^\top]^\top$ denote the concatenation of the dual variables and $\Phi:\dom\phi\times\cK^*\times \reals^{n|\cE|}\to\reals$ denote the coupling function:
$\Phi(\bx, \theta, \lambda) = \langle G(\bx), \theta \rangle + \langle A\bx, \lambda \rangle.$}
\sa{Thus, \mgb{we have} $\cL(\bx,\theta,\lambda)=\varphi(\bx)+\Phi(\bx, \theta, \lambda)-h(\theta)$, where \mgb{$\cL$ and $h(\cdot)$ as defined in \eqref{saddle-point}}, and $\varphi(\cdot)$ as defined in \eqref{pbm-to-solve}.}
For $k\geq 0$, define
\begin{equation}
\label{eq:qksk}
\begin{aligned}
&p^k \triangleq \nabla_\bx \Phi(\bx^k, \theta^k, \lambda^{k}) + \eta^k q^k,\quad
{q}^k \triangleq \nabla_\bx \Phi(\bx^k,\theta^k, \lambda^{k}) - \nabla_\bx\Phi(\bx^{k-1},\theta^{k-1}, \lambda^{k-1}),
\\
& q^{k,\by} \triangleq   \nabla_\bx \Phi(\bx^k,\theta^k, \lambda^{k}) - \nabla_\bx\Phi(\bx^{k},\theta^{k-1}, \lambda^{k-1}),\\
&q^{k,\bx} \triangleq   \nabla_\bx \Phi(\bx^{k},\theta^{k-1}, \lambda^{k-1}) - \nabla_\bx\Phi(\bx^{k-1},\theta^{k-1}, \lambda^{k-1});
\end{aligned}
\end{equation}
\sa{hence, $q^k=q^{k,\by}+q^{k,\bx}$. Moreover, for all $i\in\cN$ and $k\geq 0$, we similarly define}
\begin{equation}
    \label{eq:qi-xy}
    \begin{aligned}
    q_i^{k,\by} &\triangleq   \nabla_{x_i} \Phi(\bx^k,\theta^{k}, \lambda^{k}) - \nabla_{x_i}\Phi(\bx^{k},\theta^{k-1}, \lambda^{k-1}),\\
    q_i^{k,\bx} &\triangleq   \nabla_{x_i}\Phi(\bx^{k},\theta^{k-1}, \lambda^{k-1}) - \nabla_{x_i}\Phi(\bx^{k-1},\theta^{k-1}, \lambda^{k-1})
    \end{aligned}
\end{equation}
\mgb{based on the differences of partial derivatives with respect to $x_i$.}
\sa{Given the initial primal-dual point $\bx^0\in\dom \phi$ and $\by^0$ such that $\theta^0\in\cK^*\cap\cB$ and $\lambda^0=0$, we set $\bx^{-1} = \bx^0$ and $\by^{-1} = \by^0$; hence, $q^0=\mathbf{0}$.
Moreover, for $k\geq 0$, the Cauchy-Schwarz inequality} implies that
\begin{equation}\label{eq:CS-qk}
           \langle q^k, \bx^{k+1}-\bx^{k}\rangle \leq  \sum_{i\in\cN}\Big(\frac{1}{2\alpha_i^k}\|q_i^{k,\by}\|^2 + \frac{1}{2\beta_i^k}\|q_i^{k,\bx}\|^2 + \frac{\alpha_i^k+\beta_i^k}{2}\|x_i^{k+1}-x_i^k\|^2\Big)
\end{equation}
for \sa{any} $\alpha_i^k,\beta_i^k>0$. We will use \eqref{eq:CS-qk} later in the proof a few times.

\begin{remark}
For $i\in \cN$ such that $L_{g_i}=0$, for all $k\geq 0$, one trivially has $q_i^{k,\bx}=0;$ hence, setting $\beta_i^k=0,$ we adopt $0^2/0=0$ and set $\|q_i^{k,\bx}\|^2/\beta_i^k=0.$
\end{remark}

\sa{Recall that \texttt{Lines} \ref{algeq:APD-s}-\ref{algeq:APD-theta} in Algorithm \ref{alg:apdb} correspond to \eqref{eq:local-updates}, which can be written in a compact form as in \eqref{eq:update_unify}. Therefore, in the next lemma we provide a crucial result that analyzes the effect of one step update in the form of \eqref{eq:update_unify}. The result holds for any positive $\{\alpha_i^k,\alpha_i^{k+1},\beta_i^k,\beta_i^{k+1},\varsigma_i^k,\varsigma_i^{k+1}\}_{i\in\cN}$ which are indeed free design parameters of \alg. Later in Section \ref{sec:parameters} we discuss how these parameters should be selected.}
\begin{lemma}[One Step Result]
\label{lem:one-step}
Suppose that \cref{assmp:f,assmp:g,assmp:N} hold. 
\mgb{Fix $k \ge 0$ and let
$\bx^k,\bx^{k-1} \in \sa{\dom\phi}$,
$\theta^k,\theta^{k-1} \in \sa{\dom h}$
and $\lambda^{k},\lambda^{k-1} \in \reals^{n|\cE|}$
be arbitrary.
Let $\{\tau_i^k\}_{i\in\cN}$, $\{\sigma_i^k\}_{i\in\cN}$, $\eta^k$ and $\gamma^k$ be arbitrary positive scalars, and suppose
$(\bx^{k+1},\theta^{k+1},\lambda^{k+1})$ is generated from
$(\bx^k,\theta^k,\lambda^k)$ according to \eqref{eq:update_unify}.
For any
$\bx \in \dom \phi$, $\theta \in \dom h$ and $\lambda \in \reals^{n|\cE|}$,
let $\bz \triangleq (\bx,\by)$ with $\by \triangleq (\theta,\lambda)$,
and similarly $\by^{k+1} \triangleq (\theta^{k+1},\lambda^{k+1})$.
Then, for any collection of positive parameters
$\{\alpha_i^k,\alpha_i^{k+1},\beta_i^k,\beta_i^{k+1},\varsigma_i^k$, $ \varsigma_i^{k+1}\}_{i\in\cN}$,
the following inequality holds:}
\begin{equation}
\label{eq:Lagrangian-diff}
    \mathcal{L} (\bx^{k+1},\by) - \mathcal{L}(\bx,\by^{k+1}) \leq Q^k(\bz) - R^{k+1}(\bz) + P^{k},
\end{equation}
where
{\small
\begin{subequations}\label{eq:part-sum}
\begin{align}
        Q^k(\bz) &\triangleq \sum_{i\in\cN}\Big(\frac{1}{2\tau_i^k} \|x_i-x_i^k\|^2  + \frac{1}{2\sigma_i^k} \|\theta_i-\theta_i^{k}\|^2 + \frac{\eta^k}{2\alpha_i^k}\|q_i^{k,\by}\|^2 + \frac{\eta^k}{2\beta_i^k}\|q_i^{k,\bx}\|^2\Big)  \label{eq:Qk-def}\\
        &\quad + \frac{1}{2\gamma^k}\|\lambda- \lambda^k\|^2 + \eta^k \langle q^k, \bx -\bx^{k} \rangle + \eta^k \langle A\bx^{k} - A\bx^{k-1} , \lambda^{k} - \lambda \rangle + \sum_{i\in\cN}\frac{\eta^k \varsigma_i^k}{2}\norm{x_i^k - x_i^{k-1}}^2  \nonumber \\
        R^{k+1}(\bz) &\triangleq  \sum_{i\in\cN} \Big(\frac{1}{2\tau_i^k} \|x_i-x_i^{k+1}\|^2 + \frac{1}{2\sigma_i^k} \|\theta_i-\theta_i^{k+1}\|^2 +\frac{1}{2\alpha_i^{k+1}}\|q_i^{k+1,\by}\|^2 + \frac{1}{2\beta_i^{k+1}}\|q_i^{k+1,\bx}\|^2\Big) \label{eq:Rk-def}\\
        &\quad+ \frac{1}{2\gamma^k}\|\lambda- \lambda^{k+1}\|^2+ \langle q^{k+1}, \bx -\bx^{k+1} \rangle + \langle A\bx^{k+1} - A\bx^k , \lambda^{k+1} - \lambda \rangle + \sum_{i\in\cN}{\frac{\varsigma_i^{k+1}}{2}}\norm{x_i^{k+1} - x_i^{k}}^2 \nonumber\\
        P^{k} &\triangleq  \sum_{i\in\cN} P_i^k -\frac{1}{2}\Big(\frac{1}{\gamma^k} {- \sum_{i\in\cN}\sa{\frac{\eta^k}{\varsigma_{i}^k}} d_{i} }\Big)\|\lambda^{k+1}- \lambda^k\|^2 \label{eq:Pk-def}\\
    P_i^{k} &\triangleq -\frac{1}{2}\sa{\Big(\frac{1}{\tau_i^k}
    -\eta^k(\alpha_i^k + \beta_i^k) {-\varsigma_i^{k+1}}\Big)} \|x_i^{k+1}-x_i^k\|^2 - \frac{1}{2\sigma_i^k } \|\theta_i^{k+1}-\theta_i^k\|^2 \label{eq:Pik-def} \\
    &\qquad + \frac{1}{2\alpha_i^{k+1}}\|q_i^{k+1,\by}\|^2 + \frac{1}{2\beta_i^{k+1}}\|q_i^{k+1,\bx}\|^2+ \sa{f_i(x_i^{k+1})-f_i(x_i^k)-\fprod{\grad f_i(x_i^k),~x_i^{k+1}-x_i^k}},\quad \forall~i\in\cN,\nonumber
\end{align}
\end{subequations}}%
{for $q_i^{k,\bx}$, $q_i^{k,\by}$, $q_i^{k+1,\bx}$, and $q_i^{k+1,\by}$ defined as in~\eqref{eq:qi-xy} for all $i\in\cN$, and $q^k$ defined as in~\eqref{eq:qksk}.}
\end{lemma}

\begin{proof}
\sa{Fix arbitrary $\bx\in\dom \phi$, $\theta\in\dom h$ and $\lambda$.
Using \cite[Lemma 7.1]{hamedani2021primal} for the $\bx-$, $\theta-$ and $\lambda-$ subproblems in \eqref{eq:update_unify}, we get}

\begin{subequations}
\begin{equation} \label{IX}
\begin{aligned}
  \MoveEqLeft
  \sa{\phi(\bx^{k+1})-\phi(\bx)}
  +\langle \grad f(\bx^k) +\bp^k , \bx^{k+1} - \bx \rangle \\
       \leq &  \sum_{i\in\cN}\frac{1}{2\tau_i^k} (  \|x_i- x_i^k\|^2 - \|x_i- x_i^{k+1}\|^2 - \|x_i^{k+1}- x_i^k\|^2
       ),
    \end{aligned}
\end{equation}
\begin{equation}\label{IY1}
    \begin{aligned}
    \MoveEqLeft h(\theta^{k+1}) \sa{-h(\theta) - \langle G(\bx^{k+1}),
       \theta^{k+1} -\theta \rangle }
       \\
       \leq & \sum_{i\in\cN}\frac{1}{2\sigma_i^k} (  \|\theta_i- \theta_i^k\|^2 - \|\theta_i- \theta_i^{k+1}\|^2 -
       \|\theta_i^{k+1}- \theta_i^k\|^2
       ),
    \end{aligned}
\end{equation}
\begin{equation}\label{eq:lambda_bd}
    \begin{aligned}
        \MoveEqLeft \frac{1}{\gamma^k}\langle \lambda- \lambda^{k+1} , \lambda^{k+1} - \lambda^k \rangle = \langle \lambda -\lambda^{k+1}, A((1+\eta^k)\bx^{k} - \eta^k \bx^{k-1}) \rangle\\
        \leq& \frac{1}{2\gamma^k}(\norm{\lambda - \lambda^k}^2 - \norm{\lambda - \lambda^{k+1}}^2 - \norm{\lambda^{k+1} - \lambda^k}^2).
    \end{aligned}
\end{equation}
\end{subequations}
\sa{The convexity of $f(\cdot)$ implies that}
\begin{equation}\label{eq:bd_f}
    \begin{aligned}
        \langle \grad f(\bx^k), \bx^{k+1}-\bx \rangle
        &= \langle \grad f(\bx^k), \bx^{k+1}-\bx^{k} \rangle + \langle \grad f(\bx^k), \bx^k-\bx \rangle \\
         &\geq f(\bx^k) +\langle \grad f(\bx^k), \bx^{k+1}-\bx^{k} \rangle - f(\bx)
    \end{aligned}
\end{equation}
\sa{holds for all $\bx\in\dom \phi$ and $k\geq 0$.}
\sa{Note that $G(\bx^{k+1})=\grad_\theta \Phi(\bx^{k+1},\theta^k,\lambda^{k+1})$; hence,} the inner product in \eqref{IY1} can be
\sa{written equivalently using the linearity of $\Phi(\bx^{k+1}, \cdot)$:}
\begin{equation}
\begin{aligned}\label{eq:y_concave}
   \MoveEqLeft {\langle G(\bx^{k+1}),
       \theta^{k+1} -\theta \rangle} =  \langle \nabla_\theta  \Phi(\bx^{k+1}, \theta^k,\lambda^{k+1}), \theta^{k+1} - \theta \rangle\\
     = &
    \langle \nabla_\theta \Phi(\bx^{k+1}, \theta^k,\lambda^{k+1}), \theta^k - \theta \rangle
    +
    \langle \nabla_\theta \Phi(\bx^{k+1}, \theta^k,\lambda^{k+1}), \theta^{k+1} - \theta^k \rangle
    \\
    \sa{=} &
    \Phi(\bx^{k+1}, \theta^k,\lambda^{k+1}) - \Phi(\bx^{k+1}, \theta,\lambda^{k+1})  
    + \langle \nabla_\theta \Phi(\bx^{k+1}, \theta^k,\lambda^{k+1}), \theta^{k+1} - \theta^k \rangle.
\end{aligned}
\end{equation}
By combining \eqref{eq:y_concave} and \eqref{IY1}, we   have
\begin{equation}\label{IY2}
    \begin{aligned}
    \MoveEqLeft h(\theta^{k+1}) -h(\theta)
       \leq \sum_{i\in\cN}\frac{1}{2\sigma_i^k} (  \|\theta_i- \theta_i^k\|^2 - \|\theta_i- \theta_i^{k+1}\|^2 -
       \|\theta_i^{k+1}- \theta_i^k\|^2
       )\\
       &+\Phi(\bx^{k+1}, \theta^k,\lambda^{k+1}) - \Phi(\bx^{k+1}, \theta,\lambda^{k+1})  
    + \langle \nabla_\theta \Phi(\bx^{k+1}, \theta^k,\lambda^{k+1}), \theta^{k+1} - \theta^k \rangle.
    \end{aligned}
\end{equation}
\mgb{On the other hand,} by \sa{combining Equations \eqref{IX} and \eqref{eq:bd_f}, and adding both $f(\bx^{k+1})$ and $-\Phi(\bx, \theta^{k+1}, \lambda^{k+1})$ to both sides, for $k\geq 0$,} we get
\begin{equation}\label{IX2}
\begin{aligned}
    \MoveEqLeft \phi(\bx^{k+1})+f(\bx^{k+1}) - \phi(\bx)- f(\bx) - \Phi(\bx,\theta^{k+1},\lambda^{k+1}) \\
    &\leq  - \Phi(\bx,\theta^{k+1},\lambda^{k+1})  +\langle\bp^k , \bx-\bx^{k+1} \rangle + \sa{\Lambda^k},
    \\
    & \quad   +  \sum_{i\in\cN}\frac{1}{{2}\tau_i^k} \left[ \| x_i - x_i^k \| ^ 2 - \| x_i - x_i^{ k + 1 }  \| ^ 2 -
       \|x_i^{k+1}- x_i^k\|^2\right],
\end{aligned}
\end{equation}
\sa{where $\Lambda^k\triangleq f(\bx^{k+1})-f(\bx^k) -\langle \grad f(\bx^k), \bx^{k+1}-\bx^{k} \rangle$.} Then, \sa{summing \eqref{IY2} and \eqref{IX2} leads to}
{\allowdisplaybreaks
\begin{align}
\MoveEqLeft \mathcal{L}(\bx^{k+1}, \theta, \lambda)  - \mathcal{L}(\bx, \theta^{k+1},\lambda^{k+1}) \nonumber\\
   = &
    \phi(\bx^{k+1})+f(\bx^{k+1}) + \Phi(\bx^{k+1},\theta, \lambda) - h(\theta) - \phi(\bx) - f(\bx) - \Phi(\bx,\theta^{k+1},\lambda^{k+1}) +h(\theta^{k+1}) \nonumber \\
   \leq & \Phi(\bx^{k+1},\theta^{k+1}, \lambda^{k+1}) -\Phi(\bx,\theta^{k+1},\lambda^{k+1}) +\langle\bp^k , \bx-\bx^{k+1} \rangle  \nonumber \\
   & + \Phi(\bx^{k+1}, \theta, \lambda) - \Phi(\bx^{k+1}, \theta, \lambda^{k+1})
   + 
   \sa{\Lambda^k} \nonumber \\
   &  +  \sum_{i\in\cN}\frac{1}{{2}\sigma_i^k} \left[  \|\theta_i-\theta_i^k\|^2- \|\theta_i -  \theta_i^{k+1}\|^2 -
      \|\theta_i^{k+1}- \theta_i^k\|^2 \right] \nonumber\\
   & +  \sum_{i\in\cN}\frac{1}{{2}\tau_i^k} \left[ \| x_i - x_i^k \| ^ 2 - \| x_i - x_i^{ k + 1 }  \| ^ 2 -
       \|x_i^{k+1}- x_i^k\|^2\right],\quad\forall~k\geq 0, \label{eq:one-step-aux1}
\end{align}}%
where \sa{we used {\small $\Phi(\bx^{k+1}, \theta^k,\lambda^{k+1}) - \Phi(\bx^{k+1},\theta^{k+1}, \lambda^{k+1}) 
    + \langle \nabla_\theta \Phi(\bx^{k+1}, \theta^k,\lambda^{k+1}), \theta^{k+1} - \theta^k \rangle=0$} due to linearity of $\Phi(\bx^{k+1},\cdot,\lambda^{k+1})$.}
\sa{Moreover, using the convexity of $\Phi(\cdot,\by^{k+1})$, we get}
\begin{equation*}
    \begin{aligned}
         \MoveEqLeft \Phi(\bx^{k+1}, \theta^{k+1},\lambda^{k+1}) -  \Phi(\bx, \theta^{k+1},\lambda^{k+1}) + \langle \bp^k, \bx -\bx^{k+1}  \rangle
         \\
         \leq & \langle \nabla_\bx\Phi(\bx^{k+1},\theta^{k+1}, \lambda^{k+1}), \bx^{k+1} - \bx \rangle + \langle \nabla_\bx \Phi(\bx^k, \theta^{k}, \lambda^{k}) + \eta^{k} q^k,\bx -\bx^{k+1}  \rangle
         \\
         = & {-\langle q^{k+1}, \bx -\bx^{k+1}\rangle + \eta^k \langle q^k, \bx -\bx^{k} \rangle +\eta^k \langle q^k, \bx^k -\bx^{k+1}\rangle.}
    \end{aligned}
\end{equation*}
Next, we bound the term $\Phi(\bx^{k+1}, \theta, \lambda) - \Phi(\bx^{k+1}, \theta, \lambda^{k+1})$ as follows:
\begin{equation*}
    \begin{aligned}
        \MoveEqLeft \Phi(\bx^{k+1}, \theta, \lambda) - \Phi(\bx^{k+1}, \theta, \lambda^{k+1})\\
        =& \langle A\bx^{k+1}, \lambda - \lambda^{k+1} \rangle\\
        = & \langle A\bx^{k+1}, \lambda - \lambda^{k+1} \rangle- \langle A\bx^k,\lambda - \lambda^{k+1} \rangle +\eta^k \langle A\bx^{k}-A\bx^{k-1},  \lambda^{k+1}-\lambda\rangle \\
        & +  \langle A((1+\eta^k)\bx^k-\eta^k\bx^{k-1}),\lambda - \lambda^{k+1} \rangle \\
        \leq & - \langle A\bx^{k+1} - A\bx^k , \lambda^{k+1} - \lambda \rangle + \eta^k \langle A\bx^{k} - A\bx^{k-1} , \lambda^{k} - \lambda \rangle\\
        & + \eta^k \langle A\bx^{k} - A\bx^{k-1} , \lambda^{k+1} - \lambda^k \rangle + \frac{1}{2\gamma^k}(\norm{\lambda - \lambda^k}^2 - \norm{\lambda - \lambda^{k+1}}^2 - \norm{\lambda^{k+1} - \lambda^k}^2),
    \end{aligned}
\end{equation*}
where the inequality follows from \eqref{eq:lambda_bd}. {Thus,} the above inequalities with \eqref{eq:one-step-aux1} implies
\begin{equation*}
\begin{aligned}
\MoveEqLeft \mathcal{L}( \bx^{k+1}, \theta, \lambda)  - \mathcal{L}(\bx, \theta^{k+1},\lambda^{k+1}) \\
    \leq&
    -\langle q^{k+1}, \bx -\bx^{k+1} \rangle + \eta^k \langle q^k, \bx -\bx^{k} \rangle +\eta^k \langle q^k, \bx^k -\bx^{k+1}\rangle +
    \sa{\Lambda^k}\\
    & - \langle A\bx^{k+1} - A\bx^k , \lambda^{k+1} - \lambda \rangle + \eta^k \langle A\bx^{k} - A\bx^{k-1} , \lambda^{k} - \lambda \rangle + \eta^k \langle A\bx^{k} - A\bx^{k-1} , \lambda^{k+1} - \lambda^k \rangle\\
    & +  \sum_{i\in\cN}\frac{1}{{2}\sigma_i^k} \left[  \|\theta_i-\theta_i^k\|^2- \|\theta_i -  \theta_i^{k+1}\|^2 -
      \|\theta_i^{k+1}- \theta_i^k\|^2 \right]
   \\
   & +  \sum_{i\in\cN}\frac{1}{{2}\tau_i^k} \left[ \| x_i - x_i^k \| ^ 2 - \| x_i - x_i^{ k + 1 }  \| ^ 2 -
       \|x_i^{k+1}- x_i^k\|^2\right]\\
       & + \frac{1}{2\gamma^k} (  \|\lambda- \lambda^k\|^2 - \|\lambda- \lambda^{k+1}\|^2 -
       \|\lambda^{k+1}- \lambda^k\|^2).
\end{aligned}
\end{equation*}
\sa{Thus, given some arbitrary $\alpha_i^k,\beta_i^k>0$ for $i\in\cN$ and $k\geq 0$, \eqref{eq:CS-qk} implies that}
{\allowdisplaybreaks
\begin{align*}
        \MoveEqLeft
        \mathcal{L}( \bx^{k+1}, \theta, \lambda)  - \mathcal{L}(\bx, \theta^{k+1},\lambda^{k+1})\\
        \leq& \Big[\sum_{i\in\cN}\Big(\frac{1}{2\tau_i^k} \|x_i-x_i^k\|^2  + \frac{1}{2\sigma_i^k} \|\theta_i-\theta_i^{k}\|^2\Big) + \frac{1}{2\gamma^k}\|\lambda- \lambda^k\|^2 \\
        &+ \eta^k \langle q^k, \bx -\bx^{k} \rangle + \eta^k \langle A\bx^{k} - A\bx^{k-1} , \lambda^{k} - \lambda \rangle \Big]\\
        & - \Big[ \sum_{i\in\cN} \Big(\frac{1}{2\tau_i^k} \|x_i-x_i^{k+1}\|^2 + \frac{1}{2\sigma_i^k} \|\theta_i-\theta_i^{k+1}\|^2 \Big) + \frac{1}{2\gamma^k}\|\lambda- \lambda^{k+1}\|^2 \\
        &+ \langle q^{k+1}, \bx -\bx^{k+1} \rangle + \langle A\bx^{k+1} - A\bx^k , \lambda^{k+1} - \lambda \rangle \Big]\\
        &  
        -\Big[\sum_{i\in\cN}\Big(\frac{1}{2\tau_i^k} \|x_i^{k+1}-x_i^k\|^2 + \frac{1}{2\sigma_i^k } \|\theta_i^{k+1}-\theta_i^k\|^2\Big) + \frac{1}{2\gamma^k}\|\lambda^{k+1}- \lambda^k\|^2\Big]\\
        &+\eta^k \langle q^k, \bx^k -\bx^{k+1}\rangle +\eta^k \langle A\bx^{k} - A\bx^{k-1} , \lambda^{k+1} - \lambda^k \rangle
        + 
        \sa{\Lambda^k}\\
        \leq &\Big[\sum_{i\in\cN}\Big(\frac{1}{2\tau_i^k} \|x_i-x_i^k\|^2  + \frac{1}{2\sigma_i^k} \|\theta_i-\theta_i^{k}\|^2 + \frac{\eta^k}{2\alpha_i^k}\|q_i^{k,\by}\|^2 + \frac{\eta^k}{2\beta_i^k}\|q_i^{k,\bx}\|^2\Big) + \frac{1}{2\gamma^k}\|\lambda- \lambda^k\|^2 \\
        &+ \eta^k \langle q^k, \bx -\bx^{k} \rangle+ \eta^k \langle A\bx^{k} - A\bx^{k-1} , \lambda^{k} - \lambda \rangle + \sum_{i\in\cN} \sa{\frac{\eta^k \varsigma_i^k}{2}} \norm{x_i^k - x_i^{k-1}}^2 \Big]\\
        & - \Big[ \sum_{i\in\cN} \Big(\frac{1}{2\tau_i^k} \|x_i-x_i^{k+1}\|^2 + \frac{1}{2\sigma_i^k} \|\theta_i-\theta_i^{k+1}\|^2 +\frac{1}{2\alpha_i^{k+1}}\|q_i^{k+1,\by}\|^2 + \frac{1}{2\beta_i^{k+1}}\|q_i^{k+1,\bx}\|^2\Big)\\
        &\quad+ \frac{1}{2\gamma^k}\|\lambda- \lambda^{k+1}\|^2+ \langle q^{k+1}, \bx -\bx^{k+1} \rangle + \langle A\bx^{k+1} - A\bx^k , \lambda^{k+1} - \lambda \rangle \\
        &\quad+ \sum_{i\in\cN}\sa{\frac{\varsigma_i^{k+1}}{2}}\norm{x_i^{k+1} - x_i^{k}}^2\Big]\\
        &  
        -\Big[\sum_{i\in\cN}\Big(\frac{1}{2}\Big(\frac{1}{\tau_i^k}
        -\eta^k(\alpha_i^k + \beta_i^k)\Big) \|x_i^{k+1}-x_i^k\|^2 + \frac{1}{2\sigma_i^k } \|\theta_i^{k+1}-\theta_i^k\|^2\Big) + \frac{1}{2\gamma^k}\|\lambda^{k+1}- \lambda^k\|^2\Big]\\
        & \sa{+\Lambda^k}  + \sum_{i\in\cN} \Big(\frac{1}{2\alpha_i^{k+1}}\|q_i^{k+1,\by}\|^2 + \frac{1}{2\beta_i^{k+1}}\|q_i^{k+1,\bx}\|^2
        +\sa{\frac{\varsigma_i^{k+1}}{2}}\norm{x_i^{k+1} - x_i^{k}}^2 \\
        &\quad + \sa{\frac{\eta^k}{2\varsigma_i^k}d_i} 
        \norm{\lambda^{k+1} - \lambda^k}^2\Big),
\end{align*}}%
\sa{where in the last inequality we invoke Young's inequality the third time, i.e.,
\begin{equation}\label{eq:inner_prod_A}
    \begin{aligned}
        \MoveEqLeft\eta^k \langle A\bx^{k} - A\bx^{k-1} , \lambda^{k+1} - \lambda^k \rangle\\
        =& \eta^k \langle \bx^{k} - \bx^{k-1},~A^\top(\lambda^{k+1} - \lambda^k) \rangle\\
        \leq &\sum_{i\in\cN}\Big( \frac{\eta^k \varsigma_i^k}{2}\norm{x_i^k - x_i^{k-1}}^2 + \frac{\eta^k}{2\varsigma_i^k} \norm{A_iA_i^\top}\norm{\lambda^{k+1} - \lambda^k}^2\Big)\\
        \leq &\sum_{i\in\cN}\Big( \frac{\eta^k \varsigma_i^k}{2}\norm{x_i^k - x_i^{k-1}}^2 +  \frac{\eta^k}{2\varsigma_i^k} d_{i}\norm{\lambda^{k+1} - \lambda^k}^2\Big)
    \end{aligned}
\end{equation}
holds for any $\varsigma_i^k>0$ --in the first inequality $A_i=\Omega_i\otimes \mathbf{I}_n$ with $\Omega_i\in\reals^{|\cN|}$ denoting the $i$-th column of the Laplacian $\Omega$ for $i\in\cN$; hence, the final inequality follows from the fact that the spectral norm $\norm{A_iA_i^\top}=\norm{\Omega_i}^2=d_i$ for $i\in\cN$.}
\sa{Thus, using the definitions of $Q^k(\bz)$, $R^{k+1}(\bz)$ and $P^k$ given in \eqref{eq:part-sum},} we obtain the desired inequality in \eqref{eq:Lagrangian-diff}.
\end{proof}

\begin{lemma}
\label{lem:rough-apd-bound}
\sa{Given some arbitrary initial points $\bx^0\in\dom\phi$, $\theta^0\in\cK^*$, $\lambda^0\in\reals^{n|\cE|}$, step sizes $\gamma^k>0$, $\tau_i^k,\sigma_i^k>0$ for $i\in\cN$, and the momentum parameter $\eta^k > 0$ for all $k\geq 0$,   set $\bx^{-1}=\bx^0$, $\theta^{-1}=\theta^0$, $\lambda^{-1}=\lambda^0$, and consider the iterations as in \eqref{eq:update_unify}. Suppose that there exists $\{t_k\}_{k\geq 0}\subset\reals_{++}$ such that the algorithm parameters satisfy the following inequalities:}
\begin{equation}\label{eq:tk-rule}
{\footnotesize
    \max_{i\in\cN}\Big\{\frac{\tau_i^k}{\tau_i^{k+1}}\Big\}\leq  \frac{t_k}{t_{k+1}},\quad
    \max_{i\in\cN}\Big\{\frac{\sigma_i^k}{\sigma_i^{k+1}}\Big\}\leq  \frac{t_k}{t_{k+1}}, \quad \frac{\gamma^k}{\gamma^{k+1}}\leq \frac{t_k}{t_{k+1}},\quad \eta^{k+1}=\frac{t_{k}}{t_{k+1}} \quad\forall~k\geq 0.}
\end{equation}
\sa{Then, for any $\bz=(\bx,\by)$ such that $\bx\in\dom \phi$, $\theta\in\cK^*\cap\cB$, it holds that
\begin{equation}\label{eq:rough-apd-bound}
             T_K \left(\mathcal{L}(\bar{\bx}^{K}, \by)  - \mathcal{L}(\bx, \bar{\by}^{K}) \right) \leq t_0 Q^0(\bz) - t_{K-1} R^K(\bz)+\sum_{k=0}^{K-1}t_kP^k,\quad\forall~K\geq 1,
\end{equation}
where $T_K=\sum_{k=0}^{K-1}t_k$, and $(\bar\bx^K,\bar\by^K)\triangleq \frac{1}{T_K}\sum_{k=0}^{K-1}t_k(\bx^k,\by^k)$.}
\end{lemma}
\begin{proof}
\sa{Fix an arbitrary $\bz$ as given in the statement of the lemma.
}
\sa{Note that $\cL(\cdot,\by)-\cL(\bx,\cdot)$ is a convex function;} thus, if we multiply $t_k$ for both sides of 
\eqref{eq:Lagrangian-diff} and sum the resulting inequality from $k = 0$ to $K -1$, then using Jensen’s lemma,
we get
\begin{equation}\label{INEQ: raw gap diff}
    \begin{aligned}
             T_K \left(\mathcal{L}(\bar{\bx}^{K}, \by)  - \mathcal{L}(\bx, \bar{\by}^{K}) \right)
            &\leq
             \sum_{k=0}^{K-1}t_k \Big(Q^k(\bz)-R^{k+1}(\bz)+\sa{P^{k}}
          \Big).\\
    \end{aligned}
\end{equation}
\sa{Next, we argue that under the parameter rule in~\eqref{eq:tk-rule}, for $k=1$ to $K-2$, we have $t_{k+1}Q^{k+1}(\bz) - t_k R^{k+1}(\bz)\leq 0$; indeed, it is easy to verify 
this inequality after writing it equivalently as follows, using the definitions of $Q^k(\bz)$ and $R^{k+1}(\bz)$ given in \eqref{eq:Qk-def} and \eqref{eq:Rk-def}, respectively:}
\begin{equation*}
    \begin{aligned}
        & \sum_{i\in\cN}\Big(\frac{t_{k+1}}{2\tau_i^{k+1}} \|x_i-x_i^{k+1}\|^2  + \frac{t_{k+1}}{2\sigma_i^{k+1}} \|\theta_i-\theta_i^{k+1}\|^2 + \frac{t_{k}}{2\alpha_i^{k+1}}\|q_i^{k+1,\by}\|^2 + \frac{t_{k}}{2\beta_i^{k+1}}\|q_i^{k+1,\bx}\|^2\Big)\\
        &\quad+ \frac{t_{k+1}}{2\gamma^{k+1}}\|\lambda- \lambda^{k+1}\|^2
        + t_{k} \langle q_{k+1}, \bx -\bx^{k+1} \rangle + t_{k} \langle A\bx^{k+1} - A\bx^{k} , \lambda^{k+1} - \lambda \rangle \\
        & \quad + \sum_{i\in\cN}\frac{t_{k} \varsigma_i^{k+1}}{2}\norm{x_i^{k+1} - x_i^{k}}^2 \\
        & \leq  \sum_{i\in\cN} \Big(\frac{t_k}{2\tau_i^k} \|x_i-x_i^{k+1}\|^2 + \frac{t_k}{2\sigma_i^k} \|\theta_i-\theta_i^{k+1}\|^2 +\frac{t_k}{2\alpha_i^{k+1}}\|q_i^{k+1,\by}\|^2 + \frac{t_k}{2\beta_i^{k+1}}\|q_i^{k+1,\bx}\|^2\Big)\\
        &\quad+ \frac{t_k}{2\gamma^k}\|\lambda- \lambda^{k+1}\|^2+ t_k\langle q^{k+1}, \bx -\bx^{k+1} \rangle + t_k\langle A\bx^{k+1} - A\bx^k , \lambda^{k+1} - \lambda \rangle \\
        &\quad +\sum_{i\in\cN}\frac{t_k\varsigma_i^{k+1}}{2}\norm{x_i^{k+1} - x_i^{k}}^2\Big],
    \end{aligned}
\end{equation*}
\sa{where we used the condition $t_k=t_{k+1}\eta^{k+1}$ for $k\geq 0$. Therefore, the desired inequality follows from \eqref{INEQ: raw gap diff}.}
\end{proof}
\sa{\cref{lem:rough-apd-bound} indicates that bounding $\sum_{k=0}^{K-1}t_kP^k$ is essential to derive rate results for \alg. Next, we show that $P^k$ can be bounded above by some quantity that can be written as a node-specific consecutive iterate differences whenever $\gamma^k$ is sufficiently small. This observation will play an important role in arguing for each $i\in\cN$ that \textbf{(i)} there is a non-zero lower bound $\hat\tau_i>0$ such that $\tau_i^k\geq \hat\tau_i$ for $k\geq 0$, and that \textbf{(ii)} backtracking condition for agent-$i$ should \mgc{be satisfied after a} finite number of contractions for all iterations $k\geq 0$.}
\begin{lemma}
\label{lem:bar-P-ik}
For arbitrary $k\geq 0$, $P^k$ defined in \eqref{eq:Pk-def} can be bounded from above as follows:
\begin{equation}\label{eq:E_ik}
\begin{aligned}
    P^k &\leq \sum_{i\in\cN}\bar P_i^k - \frac{1}{2}\Big[\frac{1}{\gamma^k} - \sum_{i\in\cN}d_{i}\Big(\frac{\eta^k}{\varsigma_{i}^k}+\frac{\sa{2}}{\alpha_i^{k+1}}\Big)\Big]\norm{\lambda^{k+1}-\lambda^k}^2,\\
    \sa{\bar P_i^k}
    &\triangleq  \sa{\Lambda_i^k}+\frac{1}{\alpha_i^{k+1}} \sa{\norm{{\J g}_i(x_i^{k+1})^\top (\theta_i^{k+1}-\theta_i^k)}}^2+\frac{1}{2\beta_i^{k+1}}\norm{\sa{\big(\J g_i(x_i^{k+1}) - \J g_i(x_i^k)\big)^\top \theta_i^k}}^2\\
    &\quad -\frac{1}{2} \Big(\frac{1}{\tau_i^k} 
        -\eta^k(\alpha_{i}^k + \beta_{i}^k){-\varsigma_i^{k+1}}
        \Big)\|x_i^{k+1}-x_i^k\|^2 - \frac{1}{2\sigma_i^k} \|\theta_i^{k+1}-\theta_i^k\|^2,
\end{aligned}
\end{equation}
and \sa{$\Lambda_i^k\triangleq \sa{f_i(x_i^{k+1})-f_i(x_i^k)-\fprod{\grad f_i(x_i^k),~x_i^{k+1}-x_i^k}}$ for $i\in\cN$.} Thus, it holds that \sa{$P^k\leq \sum_{i\in\cN}\bar P_i^k$} if $\gamma^k \leq \Big(\sum_{i\in\cN}d_i(\frac{\sa{2}}{\alpha_{i}^{k+1}}+\frac{\eta^k}{\varsigma_{i}^k})\Big)^{-1}$.
\end{lemma}
\begin{proof}
\sa{Give  arbitrary $\bx',\bx''\in\dom \phi$ and $\theta',\theta''\in\cK^*\cap\cB$. Recall that $\Phi(\bx, \theta, \lambda) = \langle G(\bx), \theta \rangle + \langle A\bx, \lambda \rangle$ for any $\bx\in\dom\phi$, $\theta\in\cK^*\cap\cB$ and $\lambda\in\reals^{n|\cE|}$. Thus, for any $i\in\cN$ and $\lambda$, we have
\begin{subequations}
    \label{eq:partial-gradients}
\begin{align}
    \left\| \nabla_{x_i} \Phi(\bx'', \theta',\lambda) - \nabla_{x_i} \Phi(\bx', \theta',\lambda) \right\| &= \norm{({\J g}_i(x''_i)-{\J g}_i(x'_i))^\top \theta'_i},\label{eq:xx}\\
    \left\| \nabla_{x_i} \Phi(\bx'', \theta'',\lambda) - \nabla_{x_i} \Phi(\bx'', \theta',\lambda) \right\| &= \norm{{\J g}_i(x''_i)^\top (\theta''_i-\theta'_i)}. \label{eq:xt}
\end{align}
\end{subequations}}%
\sa{We consider the terms $\frac{1}{\alpha_i^{k+1}}\|q_i^{k+1,\by}\|^2$ and $\frac{1}{\beta_i^{k+1}}\|q_i^{k+1,\bx}\|^2$ that appear in the definition of $P^k$ given in \eqref{eq:Pk-def}. For any $i\in\cN$, using \eqref{eq:xt}, we get}
\begin{equation}\label{eq:qyi-bound}
    \begin{aligned}
        \MoveEqLeft \frac{1}{2\alpha_i^{k+1}}\norm{q_i^{k+1,\by}}^2
        \triangleq
        \frac{1}{2\alpha_i^{k+1}}\norm{\grad_{x_i} \Phi(\bx^{k+1},\theta^{k+1},\lambda^{k+1}) - \grad_{x_i} \Phi(\bx^{k+1},\theta^k,\lambda^k)}^2\\
        \leq & \frac{1}{\alpha_i^{k+1}}\norm{\grad_{x_i} \Phi(\bx^{k+1},\theta^{k+1},\lambda^{k+1}) -\grad_{x_i} \Phi(\bx^{k+1},\theta^k,\lambda^{k+1})}^2 \\
        & \mbox{} + \frac{1}{\alpha_i^{k+1}}\norm{\grad_{x_i} \Phi(\bx^{k+1},\theta^k,\lambda^{k+1}) - \grad_{x_i} \Phi(\bx^{k+1},\theta^k,\lambda^k)}^2\\
        \leq & \frac{1}{\alpha_i^{k+1}} \sa{\norm{{\J g}_i(x_i^{k+1})^\top (\theta_i^{k+1}-\theta_i^k)}^2} +\frac{
        \sa{d_i}}{\alpha_i^{k+1}}\norm{\lambda^{k+1} - \lambda^k}^2,
    \end{aligned}
\end{equation}
\sa{where in the first inequality we use $\norm{a+b}^2 \leq 2\norm{a}^2 + 2\norm{b}^2$ after adding and subtracting the term $\grad_{x_i} \Phi(\bx^{k+1},\theta^k,\lambda^{k+1})$, and the last inequality follows from \eqref{eq:xt}, $A^\top(\lambda^{k+1}-\lambda^k)=\grad_{\bx} \Phi(\bx^{k+1},\theta^k,\lambda^{k+1}) - \grad_{\bx} \Phi(\bx^{k+1},\theta^k,\lambda^k)$ and the fact that the spectral norm $\norm{A_iA_i^\top}=\norm{\Omega_i}^2=d_i$ for $i\in\cN$.} Similarly, \sa{using \eqref{eq:xx},}
we can obtain that
\begin{equation}
\label{eq:qxi-bound}
    \begin{aligned}
    \frac{1}{2\beta_i^{k+1}}\norm{q_i^{k+1,\bx}}^2
        & \triangleq
        \frac{1}{2\beta_i^{k+1}}\norm{\nabla_{x_i}\Phi(\bx^{k+1},\theta^{k}, \lambda^{k}) - \nabla_{x_i}\Phi(\bx^{k},\theta^{k}, \lambda^{k})}^2\\
        & = \frac{1}{2\beta_i^{k+1}}\norm{\sa{\big(\J g_i(x_i^{k+1}) - \J g_i(x_i^k)\big)^\top \theta_i^k}}^2.
    \end{aligned}
\end{equation}
\sa{Therefore, using the bounds in \eqref{eq:qyi-bound} and \eqref{eq:qxi-bound} within the definition of $P^k$ in \eqref{eq:Pk-def}, we get}
    \begin{align*}
        P^{k} &=  -\frac{1}{2}\sum_{i\in\cN}\Big[\Big(\frac{1}{\tau_i^k} 
        -\eta^k(\alpha_{i}^k + \beta_{i}^k){-\varsigma_i^{k+1}}\Big) \|x_i^{k+1}-x_i^k\|^2 + \frac{1}{\sigma_i^k } \|\theta_i^{k+1}-\theta_i^k\|^2\Big] \\
        &\quad - \frac{1}{2}\Big(\frac{1}{\gamma^k}{- \sum_{i\in\cN}\frac{\eta^k}{\varsigma_{i}^k}d_{i}}\Big)\|\lambda^{k+1}- \lambda^k\|^2
         + \sum_{i\in\cN} \Big(\sa{\Lambda_i^k}+\frac{1}{2\alpha_i^{k+1}}\|q_i^{k+1,\by}\|^2 + \frac{1}{2\beta_i^{k+1}}\|q_i^{k+1,\bx}\|^2 \Big) \\
        &\leq - \frac{1}{2}\sum_{i\in\cN} \Big[\Big(\frac{1}{\tau_i^k} 
        -\eta^k(\alpha_{i}^k + \beta_{i}^k){-\varsigma_i^{k+1}}
        \Big)\|x_i^{k+1}-x_i^k\|^2 + \frac{1}{\sigma_i^k} \|\theta_i^{k+1}-\theta_i^k\|^2\Big]\\
        &\quad +\sum_{i\in\cN}\Big[\sa{\Lambda_i^k}+\frac{1}{\alpha_i^{k+1}} \sa{\norm{{\J g}_i(x_i^{k+1})^\top (\theta_i^{k+1}-\theta_i^k)}}^2+\frac{1}{2\beta_i^{k+1}}\norm{\sa{\big(\J g_i(x_i^{k+1}) - \J g_i(x_i^k)\big)^\top \theta_i^k}}^2\Big]\\
        &\quad - \frac{1}{2}\Big[\frac{1}{\gamma^k} - \sum_{i\in\cN}d_{i}\Big(\frac{\eta^k}{\varsigma_{i}^k}+\frac{\sa{2}}{\alpha_i^{k+1}}\Big)\Big]\norm{\lambda^{k+1}-\lambda^k}^2;
    \end{align*}
hence, 
$P^k\leq \sum_{i\in\cN}\bar P_i^k$ holds whenever $\gamma^k \leq \Big(\sum_{i\in\cN}d_i(\frac{\sa{2}}{\alpha_{i}^{k+1}}+\frac{\eta^k}{\varsigma_{i}^k})\Big)^{-1}$.
\end{proof}

\section{Parameter Choices}
\label{sec:parameters}
\sa{In the previous section, we considered a meta algorithm, i.e., given some arbitrary initial points $\bx^0\in\dom\phi$, $\theta^0\in\cK^*$, $\lambda^0\in\reals^{n|\cE|}$, step sizes $\gamma^k>0$, $\tau_i^k,\sigma_i^k>0$ for $i\in\cN$, and the momentum parameter $\eta^k > 0$ for all $k\geq 0$, we set $\bx^{-1}=\bx^0$, $\theta^{-1}=\theta^0$, $\lambda^{-1}=\lambda^0$, and consider the iterations as in \eqref{eq:update_unify}. For this generic framework, we established a duality gap result given in \eqref{eq:rough-apd-bound} assuming that there exists $\{t_k\}_{k\geq 0}\subset\reals_{++}$ such that the algorithm parameters satisfy \eqref{eq:tk-rule}; furthermore, we were able to bound the error term $P^k$ assuming that $\gamma^k \leq \Big(\sum_{i\in\cN}d_i(\frac{\sa{2}}{\alpha_{i}^{k+1}}+\frac{\eta^k}{\varsigma_{i}^k})\Big)^{-1}$.
In this section, we will show that \alg{} displayed in Algorithm \ref{alg:apdb} satisfies all these conditions stated above. Towards this goal, we first argue that \alg{} is well defined by showing that the agent-specific local backtracking conditions can be satisfied in finite backtracking iterations for all $i\in\cN$}.

\sa{For all $i\in\cN$, let $\Tilde{\tau}_i^k,\Tilde{\sigma}_i^k, \eta_i^k$ denote the candidate stepsizes and momentum term agent-$i$ will be testing by checking \sa{the stepsize-search condition} $E_i^k(\tilde x_i^{k+1},\tilde \theta_i^{k+1})\leq -\frac{\delta}{2\tilde\tau_i^k}\norm{\tilde x_i^{k+1}-x_i^k}^2 - \frac{\delta}{2\tilde\sigma_i^k}\norm{\tilde \theta_i^{k+1} -\theta_i^k}^2$ in \lin{algeq:test} of \alg{} at the current iteration $k\geq 0$ by setting ${\alpha}_i^{k}$, ${\beta}_i^{k}$, $\Tilde{\alpha}_i^{k+1}$, $\Tilde{\beta}_i^{k+1}$, and $\Tilde{\varsigma}_i^{k+1}$ in a particular manner as stated in \alg, where $E_i^k(\cdot,\cdot)$ is defined in \eqref{eq:line-search-node}. Note that at the time test condition is checked in \lin{algeq:test}, according to \lin{algeq:eta-ik} and \lin{algeq:abv-tilde} of \alg, given some \textit{candidate} primal step size $\tilde\tau_i^k\in (0,\tau_i^{k-1}]$, we set
\begin{equation}
\label{eq:test-parameters-1}
    {\eta_i^k} = {\tau_i^{k-1}}/\tilde{\tau}_i^k,\quad \tilde{\sigma}_i^k = \zeta_i \tilde{\tau}_i^k,\quad \tilde{\alpha}_i^{k+1}={c_\alpha}/{\tilde{\tau}_i^k},\quad \tilde{\beta}_i^{k+1} ={c_\beta}/{\tilde{\tau}_i^k},\quad \tilde{\varsigma}_i^{k+1}={c_{\varsigma}}/{\tilde{\tau}_i^k};
\end{equation}
and the test point $(\tilde x_i^{k+1},\tilde\theta_i^{k+1})$ is computed according to \lin{algeq:tilde-x-plus} and \lin{algeq:tilde-theta-plus} of \alg.
Moreover, according to \lin{algeq:abv}, we also have
\begin{equation}
\label{eq:test-parameters-2}
    {\alpha}_i^{k}={c_\alpha}/{{\tau}_i^{k-1}},\qquad {\beta}_i^{k} ={c_\beta}/{{\tau}_i^{k-1}},\qquad {\varsigma}_i^{k} ={c_\varsigma}/{{\tau}_i^{k-1}}.
\end{equation}
Therefore, evaluating $E_i^k(\cdot,\cdot)$, defined in \eqref{eq:line-search-node}, at the test point $(\tilde x_i^{k+1},\tilde\theta_i^{k+1})$ using the parameters as stated in \eqref{eq:test-parameters-1} and \eqref{eq:test-parameters-2}, the test condition for agent-$i$ can be equivalently written as follows:}
\begin{equation}\label{eq:node_bc}
    \begin{aligned}
        E_i^k(\Tilde{x}_i^{k+1},\Tilde{\theta}_i^{k+1})=
        & \sa{2\tilde\Lambda_i^k} -\frac{1}{\Tilde{\tau}_i^k}\Big(\sa{1-(c_\alpha+c_\beta+c_\varsigma)}\Big) \|\Tilde{x}_i^{k+1}-x_i^k\|^2 - \frac{1}{\Tilde{\sigma}_i^k} \|\Tilde{\theta}_i^{k+1}-\theta_i^k\|^2\\
        &\mbox{} + \sa{2\frac{\tilde\tau_i^k}{c_\alpha}}\|\J g_i (\Tilde{x}_i^{k+1})^\top(\Tilde{\theta}_i^{k+1} - \theta_i^k) \|^2 + \sa{\frac{\tilde\tau_i^k}{c_\beta}}\|\left(\J g_i(\Tilde{x}_i^{k+1}) - \J g_i(x_i^k)\right)^\top \sa{\theta_i^k} \|^2 \\
        &\leq -\frac{\delta}{\Tilde{\tau}_i^k}\norm{\Tilde{x}_i^{k+1}-x_i^k}^2 - \frac{\delta}{\Tilde{\sigma}_i^k}\norm{\Tilde{\theta}_i^{k+1} -\theta_i^k}^2,
    \end{aligned}
\end{equation}
where \sa{$\tilde\Lambda_i^k\triangleq f_i(\tilde x_i^{k+1})-f_i(x_i^k)-\fprod{\grad f_i(x_i^k),~\tilde x_i^{k+1}-x_i^k}$.} \sa{Note that this condition can be easily checked by agent-$i$ locally without requiring any local communication.}

\sa{Next, we provide a sufficient condition on the stepsizes for \eqref{eq:node_bc} to hold. First, we provide some useful bounds on partial gradient differences. Consider some arbitrary $\bx',\bx''\in\dom \phi$ and $\theta',\theta''\in\cK^*\cap\cB$. Let $\bar L_{x_ix_i}\triangleq L_{g_i}B_i$ for $i\in\cN$ such that $L_{g_i}>0$; thus, for any 
$\lambda\in\reals^{n|\cE|}$, \cref{assmp:g} implies that
\begin{equation}
\label{eq:Lxx}
    \norm{({\J g}_i(x''_i)-{\J g}_i(x'_i))^\top \theta'_i} \le  L_{g_i} \|\theta'_i\|  \|x''_i - x'_i\|\leq \bar L_{x_ix_i}\|x''_i - x'_i\|,
\end{equation}
 which follows from $\theta'_i\in\cB_i$, i.e., $\norm{\theta'_i}\leq B_i$. Similarly, \cref{assmp:g} also implies
 \begin{equation}
 \label{eq:Lxy}
    \norm{{\J g}_i(x''_i)^\top (\theta''_i-\theta'_i)} \le  C_{g_i} \|\theta''_i-\theta'_i\|.
\end{equation}}%
\mgc{In the next result, we provide sufficient conditions for the backtracking line search to terminate.}
\begin{theorem}
\label{lem:sufficient-cond}
    \sa{Consider node $i\in\cN$. Given \alg{} parameters $\delta,c_\alpha,c_\beta,c_\varsigma>0$, suppose $\delta+c<1$ where $c\triangleq c_\alpha+c_\beta+c_\varsigma$. Moreover, given a candidate primal step size $\tilde\tau_i^k\in (0,\tau_i^{k-1}]$, suppose that candidate dual step size $\tilde\sigma_i^k$ and momentum parameter $\eta_i^k$ together with test function parameters $\tilde\alpha_i^{k+1},\tilde\beta_i^{k+1},\tilde\varsigma_i^{k+1}$ and $\alpha_i^k,\beta_i^k,\varsigma_i^k$ are set as in \eqref{eq:test-parameters-1} and \eqref{eq:test-parameters-2}.}

    \sa{Then, for all $k\geq 0$, the test condition in
    \emph{\lin{algeq:test} of \alg{}} holds whenever at the time of checking it, $\tilde\tau_i^k$ satisfies the following two conditions, with the convention $L_{g_i}B_i=0$ for $L_{g_i}=0,$}
    \begin{equation}
    \label{eq:sufficient-cond}
        \frac{\sa{1}-\delta}{\tilde\tau_i^k}\geq \frac{c}{\tilde\tau_i^k}+ \sa{L_{f_i}}+\frac{L_{g_i}^2B_i^2}{c_\beta} \tilde\tau_i^k,
        \qquad
        \frac{1-\delta}{\tilde{\tau}_i^k}\geq \frac{\sa{2}\zeta_i C^2_{g_i}}{c_\alpha}\tilde\tau_i^k;
    \end{equation}
    moreover, \eqref{eq:sufficient-cond} is guaranteed to hold whenever $\tilde\tau_i^k\in(0,\hat\tau_i]$, where
    \begin{equation}
    \label{eq:hat-tau}
    {\displaystyle
        \hat\tau_i\triangleq\min\left\{\sa{\frac{-L_{f_i}+\sqrt{L_{f_i}^2+4(1-(\delta+c))L_{g_i}^2B_i^2/c_\beta}}{2L_{g_i}^2B_i^2/c_\beta}},\ \frac{1}{C_{g_i}}\sqrt{\frac{{c_\alpha(1-\delta)}}{\sa{2}\zeta_i}}\right\}}
    \end{equation}
for $i\in \cN$ such that $L_{g_i}>0$, and       $\hat{\tau}_i\triangleq\min\{\frac{1-(\delta+c)}{2L_{f_i}}, \frac{1}{C_{g_i}}\sqrt{\frac{c_\alpha(1-\delta)}{2\zeta_i}}\}$ for $i\in \cN$ such that $L_{g_i}=0$.
   Finally, one also has$^5$\footnote{$^5$With the convention that $1/0=\infty$ when $L_{g_i}B_i=0$ for $\in\cN$ such that $L_{g_i}=0$.}
$$
\hat\tau_i\geq \min\Big\{\frac{1-(\delta+c)}{2L_{f_i}},~\frac{1}{L_{g_i}B_i}\sqrt{\frac{c_\beta(1-(\delta+c))}{2}},~\frac{1}{C_{g_i}}\sqrt{\frac{{c_\alpha(1-\delta)}}{\sa{2}\zeta_i}}\Big\},\quad\forall i\in \cN.
$$
\end{theorem}

\begin{proof}
   \sa{First, note that $\tilde\Lambda_i^k\leq \frac{L_{f_i}}{2}\norm{\tilde x_i^{k+1}-x_i^k}^2$ as $\grad f_i$ is Lipschitz with constant $L_{f_i}$ over $\dom \phi_i$. Morover, since $g_i$ and $\J{g_i}$ are Lipschitz on $\dom\phi_i$ with constants $C_{g_i}$ and $L_{g_i}$, respectively, and $\norm{\theta_i^k}\leq B_i$, using \eqref{eq:Lxx} and \eqref{eq:Lxy}, we get}
   \begin{equation*}
   \begin{aligned}
       \MoveEqLeft \sa{2\frac{\tilde\tau_i^k}{c_\alpha}}\|\J g_i (\Tilde{x}_i^{k+1})^\top(\Tilde{\theta}_i^{k+1} - \theta_i^k) \|^2 + \sa{\frac{\tilde\tau_i^k}{c_\beta}}\|\left(\J g_i(\Tilde{x}_i^{k+1}) - \J g_i(x_i^k)\right)^\top \sa{\theta_i^k} \|^2\\
       &\leq \sa{\frac{\tilde\tau_i^k}{c_\alpha}}2C_{g_i}^2\|\Tilde{\theta}_i^{k+1} - \theta_i^k \|^2 + \sa{\frac{\tilde\tau_i^k}{c_\beta}} L_{g_i}^2 B_i^2\norm{\Tilde{x}_i^{k+1}-x_i^k}^2.
    \end{aligned}
   \end{equation*}
   \sa{Since $\tilde\sigma_i^k=\zeta_i\tilde\tau_i^k$, we can conclude that \eqref{eq:node_bc} holds if \eqref{eq:sufficient-cond} holds.}

   \sa{Clearly, the roots of the two quadratic inequalities in \eqref{eq:sufficient-cond} immediately imply that any $\tilde\tau_i^k\in(0,\hat\tau_i]$ satisfies \eqref{eq:sufficient-cond} for $\hat\tau_i$ given as in~\eqref{eq:hat-tau}. Finally, the lower bound on $\hat\tau_i$ is obtained by solving for the roots of another quadratic system, which provides us with a stronger condition, i.e., $\frac{1-(\delta+c)}{2\tilde\tau_i^k}\geq \max\Big\{L_{f_i},~\frac{L_{g_i}^2B_i^2}{c_\beta} \tilde\tau_i^k\Big\}$ and $\frac{1-\delta}{\tilde{\tau}_i^k}\geq \frac{\sa{2}\zeta_i C^2_{g_i}}{c_\alpha}\tilde\tau_i^k$.}
\end{proof}

\begin{definition}
\label{def:admissible-parameters}
    \sa{Given a candidate primal step size $\tilde\tau_i^k\in(0,\tau_i^{k-1}]$, let candidate dual step size $\tilde\sigma_i^k$ and momentum parameter $\eta_i^k$ together with test function parameters $\tilde\alpha_i^{k+1},\tilde\beta_i^{k+1},\tilde\varsigma_i^{k+1}$ and $\alpha_i^k,\beta_i^k,\varsigma_i^k$ are set as in \eqref{eq:test-parameters-1} and \eqref{eq:test-parameters-2}. The \alg{} parameters $(\tilde\tau_i^k,\tilde\sigma_i^k,\eta_i^k)$ are called \textit{admissible} if the test condition in {\lin{algeq:test} of \alg{}} holds with these parameters.}
\end{definition}

\sa{In~\cref{lem:sufficient-cond} we established that for each $k\geq 0$, the test condition in \lin{algeq:test} of \alg{} holds after a finite number of backtracking iterations.
For $k\geq 0$, let the \textit{candidate} primal step size is $\tilde\tau_i^k$ when the condition in \lin{algeq:test} holds for $i\in\cN$; hence, $\eta_i^k=\frac{\tau_i^{k-1}}{\tilde\tau_i^k}$, which implies that the number of times node-$i$ employs backtracking (i.e., shrinks its primal step size) within the $k$-th iteration is equal to $n_i^k\triangleq \log_{1/\rho}\eta_i^k$. Thus, $\eta^k = \max_{i\in\cN} \{\eta_i^k\}=\max_{i\in\cN}\{\tau_i^{k-1}/\tilde\tau_i^k\}$ denotes the largest amount of contraction of primal step size among all the nodes.}
\begin{corollary}
    \label{cor:contraction-bound}
    \sa{Under the premise of \cref{lem:sufficient-cond}, for all $k\geq 0$, the number of times node-$i$ employs backtracking (i.e., shrinks its primal step size) within the $k$-th iteration of \alg{}, i.e., $n_i^k\triangleq \log_{1/\rho}\eta_i^k$, is bounded above by $\bar n_i\triangleq \log_{1/\rho}\lceil\bar\tau_i/\hat\tau_i\rceil+1$ for $i\in\cN$.}
\end{corollary}

\sa{The \textit{accepted} parameters $\tau_i^k,\sigma_i^k, \eta^k, \alpha_i^{k+1}, \beta_i^{k+1}, \varsigma_i^{k+1}$ are set after the max-consensus step
$\eta^k = \max_{i\in\cN} \{\eta_i^k\}$, where $\tau_i^k = {\tau_i^{k-1}}/{\eta^k}$, $\sigma_i^k = \zeta_i \tau_i^k$, and $\alpha_i^{k+1}=c_\alpha/\tau_i^k$, $\beta_i^{k+1}=c_\beta/\tau_i^k$, $\varsigma_i^{k+1}=c_\varsigma/\tau_i^k$. Then, each node-$i$ computes $x_i^{k+1}$ and $\theta_i^{k+1}$ according to \lin{algeq:APD-x} and \lin{algeq:APD-theta} of~\alg, respectively, for $i\in\cN$. It is essential to note that $\eta^k=1$ implies no backtracking is employed, i.e., $\eta_i^k=1$ for all $i\in\cN$; hence, $\tau_i^k=\tilde\tau_i^k=\tau_i^{k-1}$ for all $i\in\cN$, which also implies that $\tilde x_i^{k+1}={\rm prox}_{\tau_i^k \phi_i}\Big(x_i^k - \tau_i^k\big(\nabla f_i(x_i^k)+p_i^k\big)\Big)$ and $\tilde\theta_i^{k+1}=\cP_{\cK_i^*\sa{\cap\cB_i}}\!\left(\theta_i^k + \sigma_i^k\,g_i(x_i^{k+1})\right)$ -- that is why we set $x_i^{k+1}=\tilde x_i^{k+1}$ and $\theta_i^{k+1}=\tilde \theta_i^{k+1}$ in \lin{algeq:APD-x-nb} and \lin{algeq:APD-theta-nb} of~\alg{} for the scenario $\eta^k=1$.}

\sa{Below we collect all the parameter conditions stated so far within the hypothesis of \cref{lem:rough-apd-bound,lem:bar-P-ik,lem:sufficient-cond} together with \eqref{eq:test-parameters-1} and \eqref{eq:test-parameters-2}. All of these required conditions are formally stated below in \textbf{Condition}~\ref{cond:step}.}
\begin{condition}\label{cond:step}
\sa{For the given 
parameters $\delta,c_\alpha,c_\beta,c_\varsigma>0$, the primal-dual step-size sequence $\{(\tau_i^k,\sigma_i^k)\}_{k\geq 0}\subset\reals_{++}\times\reals_{++}$ for $i\in\cN$ and the momentum parameter sequence $\{(\gamma^k,\eta^k)\}_{k\geq 0}\subset\reals_{++}\times\reals_{++}$ is acceptable if there exists $\{t_k\}_{k\geq 0}\subset\reals_{++}$ such that \eqref{eq:tk-rule} holds together with $\gamma^k \leq \Big(\sum_{i\in\cN}d_i(\frac{\sa{2}}{\alpha_{i}^{k+1}}+\frac{\eta^k}{\varsigma_{i}^k})\Big)^{-1}$ where the auxiliary parameter sequences $\{(\alpha_i^k,\beta_i^k,\varsigma_i^k)\}_{k\geq 0}$ are defined as in \eqref{eq:test-parameters-2} such that $\delta+c<1$ for $c\triangleq c_\alpha+c_\beta+c_\varsigma$, and $(\alpha_i^0,\beta_i^0,\varsigma_i^0)$ is initialized as $\alpha_i^0=c_\alpha/\bar\tau_i$, $\beta_i^0=c_\beta/\bar\tau_i$, $\varsigma_i^0=c_\varsigma/\bar\tau_i$
for some $\bar\tau_i>0$ such that $\tau_i^k\leq\bar\tau_i$ for $k\geq 0$ for all $i\in\cN$.}
\end{condition}

Next, we show that the particular parameter choice in \alg{} satisfy  \textbf{Condition}~\ref{cond:step}.
\begin{theorem}
\label{thm:parameters}
    \sa{Suppose \cref{assmp:f,assmp:g,assmp:N} hold, and 
    $\delta,c_\alpha,c_\beta,c_\varsigma>0$ are given that satisfy $\delta+c<1$, where $c\triangleq c_\alpha+c_\beta+c_\varsigma$. For each $i\in\cN$, let $x_i^0\in\dom\phi_i,\theta_i^0\in\cK_i^*$ be the initial primal-dual variables, 
    and let $\bar\tau_i,\zeta_i>0$ and $\rho\in(0,1)$ denote the step size parameters of node $i$ and the contraction coefficient. Then, the primal-dual step-size sequence $\{(\tau_i^k,\sigma_i^k)\}_{k\geq 0}\subset\reals_{++}\times\reals_{++}$ for $i\in\cN$ and the momentum parameter sequence $\{(\gamma^k,\eta^k)\}_{k\geq 0}\subset\reals_{++}\times\reals_{++}$ generated according to \alg{} are acceptable with respect to~\textbf{Condition}~\ref{cond:step}.}
\end{theorem}
\begin{proof}
    \sa{For all $k\geq 0$, according to \cref{lem:sufficient-cond}, the test condition in {\lin{algeq:test} of \alg{}} holds within finite backtracking iterations, let $(\tilde\tau_i^k,\tilde\sigma_i^k,\eta_i^k)$ be an \textit{admissible} set of local parameters in terms of Definition~\ref{def:admissible-parameters}. According to \alg{}, we have $\eta_i^k=\tau_i^{k-1}/\tilde\tau_i^k\geq 1$ since $\tilde\tau_i^k\in(\rho_i\hat\tau_i,~\tau_i^{k-1}]$ for all $i\in\cN$; therefore, $\eta^k=\max_{i\in\cN}\eta_i^k\geq 1$ as well. Moreover, according to \lin{algeq:gamma-update} and \lin{algeq:tau-sigma} of \alg{}, we have $\gamma^k=\frac{c_\gamma}{\bar\tau}\Big(\frac{2}{c_\alpha}+\frac{\eta^k}{c_\varsigma}\Big)^{-1}$ for some $c_\gamma\in \Big(0,\frac{1}{2|\cE|}\Big)$ while the primal-dual step sizes are set to $\tau_i^k=\tau_i^{k-1}/\eta^k$ and $\sigma_i^k=\zeta_i\tau_i^k$, where $\bar\tau=\max_{i\in\cN}\bar\tau_i$. Hence, for any $i\in\cN$, it holds that
    \begin{equation}
    \label{eq:tau-monotone}
        \tau_i^k=\tau_i^{k-1}/\eta^k\leq \tau_i^{k-1}/\eta_i^k\leq\tau_i^{k-1}.
    \end{equation}
    The auxiliary parameter update in \lin{algeq:abv} of~\alg{} shows that $(\alpha_i^{k+1},\beta_i^{k+1},\varsigma_i^{k+1})$ are defined based on \eqref{eq:test-parameters-2} such that $\delta+c<1$ for $c\triangleq c_\alpha+c_\beta+c_\varsigma$, and $(\alpha_i^0,\beta_i^0,\varsigma_i^0)$ is indeed initialized as $\alpha_i^0=c_\alpha/\bar\tau_i$, $\beta_i^0=c_\beta/\bar\tau_i$, $\varsigma_i^0=c_\varsigma/\bar\tau_i$
for some $\bar\tau_i>0$ such that $\tau_i^k\leq\bar\tau_i$ for $k\geq 0$ for all $i\in\cN$ --the last condition holds as it has been shown that $\{\tau_i^k\}_{k\geq 0}$ is a nonincreasing sequence such that $\tau_i^k\leq\bar\tau_i$ for all $k\geq 0$ by construction for all $i\in\cN$.}

\sa{Next, we check the condition on $\{\gamma^k\}$ to be acceptable by verifying that
\begin{equation}
\label{eq:gamma-check}
\begin{aligned}
    \sum_{i\in\cN}d_i\Big(\frac{\sa{2}}{\alpha_{i}^{k+1}}+\frac{\eta^k}{\varsigma_{i}^k}\Big)
    &=\sum_{i\in\cN}d_i\Big(\frac{2\tau_i^k}{c_\alpha}+\frac{\eta^k\tau_i^{k-1}}{c_\varsigma}\Big)\leq \sum_{i\in\cN}d_i\tau_i^{k-1}\Big(\frac{2}{c_\alpha}+\frac{\eta^k}{c_\varsigma}\Big)\\
    &\leq \bar\tau\sum_{i\in\cN}d_i\Big(\frac{2}{c_\alpha}+\frac{\eta^k}{c_\varsigma}\Big)\leq \frac{\bar\tau}{c_\gamma}\Big(\frac{2}{c_\alpha}+\frac{\eta^k}{c_\varsigma}\Big);
\end{aligned}
\end{equation}
hence, we can conclude that $\gamma^k \leq \Big(\sum_{i\in\cN}d_i(\frac{\sa{2}}{\alpha_{i}^{k+1}}+\frac{\eta^k}{\varsigma_{i}^k})\Big)^{-1}$ as in \textbf{Condition}~\ref{cond:step}, where the first inequality follows from \eqref{eq:tau-monotone}, the second inequality uses $\tau_i^{k}\leq \bar\tau_i$ for $k\geq -1$ and $\bar\tau=\max_{i\in\cN}\bar\tau_i$, finally in the last inequality we use $\sum_{i\in\cN}d_i=2|\cE|$.}

\sa{To verify that \textbf{Condition}~\ref{cond:step} holds for the sequences $\{(\tau_i^k,\sigma_i^k)\}_{k\geq 0}$ for $i\in\cN$ and $\{(\gamma^k,\eta^k)\}_{k\geq 0}$ generated according to \alg{}, one needs to show that there exists $\{t_k\}_{k\geq 0}\subset\reals_{++}$ such that \eqref{eq:tk-rule} holds. At this point, we have shown that $\{\eta^k\}_{k\geq 0}$ exists that satisfy $\eta^k\geq 1$. Thus, we recursively define $\{t_k\}_{k\geq 0}$ as follows: $t_0=1$ and $t_{k+1}=t_k/\eta^{k+1}$, which implies that $t_{k+1}\leq t_k\leq 1$ for $k\geq 0$. Moreover, \eqref{eq:tau-monotone} and our choice of $\sigma_i^k=\zeta_i\tau_i^k$ imply that $\max_{i\in\cN}\Big\{\frac{\tau_i^{k-1}}{\tau_i^{k}}\Big\}=\eta^k=\frac{t_{k-1}}{t_{k}}$ and
    $\max_{i\in\cN}\Big\{\frac{\sigma_i^{k-1}}{\sigma_i^{k}}\Big\}=\eta^k=\frac{t_{k-1}}{t_{k}}$ both hold by construction. For $k\geq 0$, the only thing that remains to be shown at this point is $\frac{\gamma^k}{\gamma^{k+1}}\leq \frac{t_k}{t_{k+1}}=\eta^{k+1}$; indeed,
    \begin{equation}
    \begin{aligned}
        \frac{\gamma^k}{\gamma^{k+1}} & = \frac{\frac{c_\gamma}{\bar\tau}(\frac{2}{c_\alpha} + \frac{\eta^k }{c_{\varsigma}})^{-1}}{\frac{c_\gamma}{\bar\tau}(\frac{2}{c_\alpha} + \frac{\eta^{k+1} }{c_{\varsigma}})^{-1}}= \frac{2 c_{\varsigma} + \eta^{k+1}c_\alpha}{2 c_{\varsigma} + \eta^k c_\alpha} \leq \eta^{k+1}
    \end{aligned}
\end{equation}
holds if and only if $2c_{\varsigma}(\eta^{k+1} - 1) + c_\alpha\eta^{k+1}(\eta^k - 1)\geq 0$, which can be easily verified as $\eta^k \geq 1$ and $\eta^{k+1}\geq 1$.}
\end{proof}

\section{Convergence guarantees for \alg{} and the Proof of Main Results}
\label{sec:main-proofs}
\sa{Now we are ready to establish convergence guarantees for \alg. We first analyze how we can bound the duality gap for \alg{} primal-dual iterate sequence through carefully controlling the difference between the admissible parameters $(\tilde\tau_i^k,\tilde\sigma_i^k,\eta_i^k)$ according to Definition~\ref{def:admissible-parameters} and the updated parameters after computing max-consensus  among the nodes $(\tau_i^k,\sigma_i^k,\eta^k)$.}

\sa{We first define an index set that plays a crucial role in the analysis of \alg{}. Recall that for all $k\geq 0$, $\eta^k=\max_{i\in\cN}\frac{\tau_i^{k-1}}{\tilde\tau_i^k}$; thus, $\eta^k>1$ implies that in the $k$-th \alg{} iteration at least one node has employed backtracking, i.e., its candidate primal step size has contracted.}
\begin{definition}
    \label{def:I-set}
    \sa{Let $\cI\triangleq\{k\in\integers_+:\ \eta^k>1\}$ be the set of iterations with contraction.}
\end{definition}

\begin{lemma}
\label{lem:I-cardinality}
    \sa{Under the premise of \cref{thm:parameters}, let $\{\eta^k\}_{k\geq 0}$ be the momentum sequence generated by \alg{} displayed in \cref{alg:apdb}. Then, $\{t_k\}_{k\geq 0}$ such that $t_0=1$ and $t_{k+1}=t_k/\eta^{k+1}$ for $k\geq 0$ satisfies \textbf{Condition}~\ref{cond:step} such that $t_k\leq t_0=1$ and $t_k\geq \hat t\triangleq \rho \min_{i\in\cN}\hat\tau_i/\bar\tau_i$, which implies that $T_K=\sum_{k=0}^{K-1}t_k\geq \hat t K$ for all $K\geq 1$, i.e., $T_K=\Omega(K)$. Moreover, $|\cI|\leq \Big\lfloor\log_{1/\rho}\Big(\max_{i\in\cN}\frac{\bar\tau_i}{\hat\tau_i}\Big)\Big\rfloor+1$.}
\end{lemma}

\begin{proof}
\sa{Since $t_0=1$, from the recursive definition of $\{t_{k}\}$, it follows that $t_k=\Big(\Pi_{\ell=1}^k\eta^\ell\Big)^{-1}$ for $k\geq 1$. Moreover, for $i\in\cN$ and $k\geq 1$, since $\tau_i^k=\tau_i^{k-1}/\eta^k$, we have $\Pi_{\ell=1}^k\eta^\ell=\frac{\tau_i^0}{\tau_i^k}$; hence, $t_k=\frac{\tau_i^k}{\tau_i^0}$. Due to $\tau_i^k=\tau_i^{k-1}/\eta^k$ for $i\in\cN$, whenever at least one node shrinks its step size, all the other nodes must shrink theirs as well. Therefore, for all $k\geq 0$ such that $\eta^k>1$, we have $\log_{1/\rho}\eta^k\geq 1$, i.e., every node must shrink their step size from $\tau_i^{k-1}$ to $\tau_i^k$ by using at least one contraction by $\rho\in(0,1)$. On the other hand, since $\{\tau_i^k\}_{k\geq 0}$ is nonincreasing sequences for all $i\in\cN$, it follows from \cref{lem:sufficient-cond} that whenever $\tau_i^{k^*-1}\leq\hat\tau_i$ for all $i\in\cN$ for some $k^*\geq 0$, the test condition in {\lin{algeq:test} of \alg{}} is guaranteed to hold for all $k\geq k^*$, i.e., $\eta^k=1$ for all $k\geq k^*$. Therefore, we can conclude that $|\cI|\leq \Big\lfloor\log_{1/\rho}\Big(\max_{i\in\cN}\frac{\bar\tau_i}{\hat\tau_i}\Big)\Big\rfloor+1$; thus, $\tau_i^k\geq \bar\tau_i \rho^{|\cI|}\geq \rho\bar\tau_i\min_{j\in\cN}\frac{\hat\tau_j}{\bar\tau_j}$ for all $i\in\cN$. Using our earlier observation $t_k=\frac{\tau_i^k}{\tau_i^0}$ for all $i\in\cN$ and $k\geq 1$, we obtain the following lower bound: $t_k\geq \rho \min_{i\in\cN}\frac{\hat\tau_i}{\bar\tau_i}$, which follows from $\tau_i^0\leq\bar\tau_i$ for $i\in\cN$.}
\end{proof}

\begin{lemma}\label{thm:line-search-diff}
\sa{Under the premise of \cref{thm:parameters}, given some arbitrary initial points $(x_i^0,\theta_i^0)\in\dom\phi_i\times \cK_i^*$ and the step size parameters $\bar\tau_i,\zeta_i>0$ for $i\in\cN$, and the contraction coefficient $\rho\in(0,1)$, let $\{x_i^k,\theta_i^k\}_{k\geq 0}$ denote the \alg{} iterate sequence for $i\in\cN$. Then, it holds that
\begin{equation}
\label{eq:error-bound-k}
    \sum_{i\in\cN} \bar P_i^k\leq -\frac{\delta}{2}\sum_{i\in\cN}\Big(\frac{1}{\tau_i^k}\norm{x_i^{k+1}-x_i^k}^2 + \frac{1}{\sigma_i^k}\norm{\theta_i^{k+1} -\theta_i^k}^2\Big)+\mathbf{1}_{\cI}(k)\sum_{i\in\cN}\Xi_i^k,
\end{equation}
where $\cI$ is given in Definition~\ref{def:I-set}, $\bar P_i^k$ is defined in~\eqref{eq:E_ik}, and $\mathbf{1}_{\cI}(k)=1$ if $k\in\cI$, equal to $0$ otherwise; moreover, for $i\in\cN$, for the case $L_{g_i}>0$, $\Xi_i^k\triangleq 2\Big(L_{f_i}+\frac{\bar\tau_i}{c_\beta}B_i^2L_{g_i}^2\Big)D_i^2 + 4\frac{\bar\tau_i}{c_\alpha}C_{g_i}^2B_i^2$, and for the case $L_{g_i}=0$, $\Xi_i^k=2 L_{f_i}D_i^2+\frac{\bar{\tau}_i}{c_\alpha}C_{g_i}^2\norm{{\theta}_i^{k+1} -\theta_i^k}^2$ for $k\geq 0$.}
\end{lemma}
\begin{proof}
    Let $\cI\subset\integers_+$ be the index set as given in Definition~\ref{def:I-set}. First, consider a fixed iteration $k\in\integers_+$ such that $k\not\in\cI$. Since $\eta^k\geq 1$ for all $k\geq 0$, $k\in\cI^c$ implies that $\eta^k=1$; hence, according to \lin{algeq:APD-x-nb} and \lin{algeq:APD-theta-nb} of \alg, we get $\tilde x_i^{k+1}=x_i^{k+1}$ and $\tilde \theta_i^{k+1}=\theta_i^{k+1}$ for all $i\in\cN$, and we have $\tilde\tau_i^k=\tau_i^k=\tau_i^{k-1}$, which also implies $\tilde\sigma_i^k=\sigma_i^k=\sigma_i^{k-1}$ together with $\tilde\alpha_i^{k+1}=\alpha_i^{k+1}$, $\tilde\beta_i^{k+1}=\beta_i^{k+1}$ and $\tilde\varsigma_i^{k+1}=\varsigma_i^{k+1}$. Therefore, it is easy to verify that for all $i\in\cN$, it holds that
    \begin{equation}
    	\label{eq:bar_P-Ic}
    	\bar P_i^k=\frac{1}{2}E_i^k(x_i^{k+1},\theta_i^{k+1})\leq -\frac{\delta}{2\tau_i^k}\norm{x_i^{k+1}-x_i^k}^2 - \frac{\delta}{2\sigma_i^k}\norm{\theta_i^{k+1} -\theta_i^k}^2,\quad\forall~k\in\cI^c,
    \end{equation}
    which follows from \lin{algeq:test} of \alg{} and \eqref{eq:node_bc}.
    
    Next, consider a fixed iteration $k\in\integers_+$ such that $k\in\cI$. For each $i\in\cN$, using $\Lambda_i^k\leq\frac{L_{f_i}}{2}\norm{x_i^{k+1}-x_i^k}^2$, $\alpha_i^{k+1}=c_\alpha/\tau_i^k$, $\beta_i^{k+1}=c_\beta/\tau_i^k$, $\varsigma_i^{k+1}=c_\varsigma/\tau_i^k$ and $\tau_i^k=\tau_i^{k-1}/\eta^k$, we can bound $\bar P_i^k$ defined in~\eqref{eq:E_ik} as follows:
    \begin{equation*}
    	\begin{aligned}
    		\bar P_i^k\leq 
    		& \frac{L_{f_i}}{2} \|{x}_i^{k+1}-x_i^k\|^2 -\frac{1}{2{\tau}_i^k}\Big(\sa{1-(c_\alpha+c_\beta+c_\varsigma)}\Big) \|{x}_i^{k+1}-x_i^k\|^2 - \frac{1}{2{\sigma}_i^k} \|{\theta}_i^{k+1}-\theta_i^k\|^2\\
    		&\mbox{} + \sa{\frac{\tau_i^k}{c_\alpha}}\|\J g_i ({x}_i^{k+1})^\top({\theta}_i^{k+1} - \theta_i^k) \|^2 + \sa{\frac{\tau_i^k}{2c_\beta}}\|\left(\J g_i({x}_i^{k+1}) - \J g_i(x_i^k)\right)^\top \sa{\theta_i^k} \|^2 \\
    		&\leq \frac{1}{2}\Big(L_{f_i}+\frac{\tau_i^k}{c_\beta}B_i^2L_{g_i}^2-\frac{\delta}{{\tau}_i^k}\Big)\norm{{x}_i^{k+1}-x_i^k}^2 + \Big(\frac{\tau_i^k}{c_\alpha}C_{g_i}^2- \frac{1}{2{\sigma}_i^k}\Big)\norm{{\theta}_i^{k+1} -\theta_i^k}^2,
    	\end{aligned}
    \end{equation*}
    where we used $\eta^k(\alpha_i^k+\beta_i^k)=(c_\alpha+c_\beta)/\tau_i^k$ for the first inequality, and $\delta\leq 1-(c_\alpha+c_\beta+c_\varsigma)$ in the second inequality. For all $i\in\cN$, recall that for the case $L_{g_i}>0$, we have $\norm{\theta_i^k}\leq B_i$; hence, $\norm{\theta_i^{k+1}-\theta_i^k}^2\leq 4B_i^2$ for all $k\geq 0$. Therefore, for all $i\in\cN$, since $\tau_i^k\leq\bar\tau_i$, using $\norm{x_i^{k+1}-x_i^k}^2\leq 4D_i^2$ together with $\delta\in(0,1)$, we get
    	$\bar P_i^k\leq -\frac{\delta}{2\tau_i^k}\norm{x_i^{k+1}-x_i^k}^2 - \frac{\delta}{2\sigma_i^k}\norm{\theta_i^{k+1} -\theta_i^k}^2+\Xi^k_i$,
    	for all $k\in\cI$.
\end{proof}
\begin{theorem}
	\label{thm:bound}
		Under the premise of \cref{thm:parameters}, given some arbitrary initial points $(x_i^0,\theta_i^0)\in\dom\phi_i\times \cK_i^*$ and the step size parameters $\bar\tau_i,\zeta_i>0$ for $i\in\cN$, and the contraction coefficient $\rho\in(0,1)$, it holds for $i\in\cN$ that the \alg{} dual iterate sequence $\{\theta_i^k\}_{k\geq 0}$ is bounded; hence, for $i\in\cN$ such that $L_{g_i}=0$, there exists $B_i>0$ such that $\norm{\theta_i^k}\leq B_i$ for all $k\geq 0$.
\end{theorem}
\begin{proof}
	Let $(\hat\bx,\hat\theta,\hat\lambda)$ be an arbitrary saddle point of $\cL$. Using Young's inequality in a similar manner to \eqref{eq:CS-qk}, it follows from the definition of $Q^k(\hat\bz)$ in \eqref{eq:Qk-def} that
		\begin{equation}
			\label{eq:a-lower-bound}
			\begin{aligned}
				t_kQ^k(\hat\bz) 
				&\geq \sum_{i\in\cN}\Big[\frac{t_k}{2}\Big(\frac{1}{\tau_i^k}-\eta^k(\alpha_i^k+\beta_i^k)\Big)\norm{x_i^k-\hat x_i}^2+\frac{t_k}{2\sigma_i^k}\norm{\theta_i^k-\hat\theta_i}^2\Big]\\
				&\quad +\frac{t_k}{2}\Big(\frac{1}{\gamma^k}-\sum_{i\in\cN}d_i\frac{\eta^k}{\varsigma_i^k}\Big)\norm{\lambda^k-\hat\lambda}^2\\
				&\geq \frac{\hat t}{2}\sum_{i\in\cN}\Big(\frac{\delta}{\bar\tau_i}\norm{x_i^k-\hat x_i}^2+\frac{1}{\zeta_i\bar\tau_i}\norm{\theta_i^k-\hat\theta_i}^2\Big)+\frac{{2}\hat t\rho}{c_\alpha}{|\cE|}~\min_{i\in\cN}\{\hat\tau_i\}~\norm{\lambda^k-\hat\lambda}^2>0,
			\end{aligned}
		\end{equation}
		where in the second inequality we used $t_k\geq \hat t$, $\tau_i^k\leq \bar\tau_i$, $\sigma_i^k=\zeta_i\tau_i^k$, and $1-c_\alpha-c_\beta\geq \delta$ together with the fact that 
		\begin{equation*}
			\frac{1}{\gamma^k}-\sum_{i\in\cN}d_i\frac{\eta^k}{\varsigma_i^k}\geq \sum_{i\in\cN}d_i\frac{2}{\alpha_i^{k+1}}=\frac{2}{c_\alpha}\sum_{i\in\cN}{d_i}\tau_i^k\geq {\frac{4\rho}{c_\alpha}|\cE|}\min_{i\in\cN}\hat\tau_i>0,
		\end{equation*}
		which holds due to Condition~\ref{cond:step}, {$|\cE|=\sum_{i\in\cN}d_i/2$} and $\tau_i^k\geq \rho \min_{i\in\cN}\hat\tau_i$ for all $k\geq 0$.
	
	Moreover, since Condition~\ref{cond:step} holds, \eqref{eq:tk-rule} immediately implies $t_kR^{k+1}(\hat\bz)\geq t_{k+1}Q^{k+1}(\hat\bz)$. Therefore, after multiplying both sides of \eqref{eq:Lagrangian-diff} by $t_k$, and using the fact that $\cL(\bx^{k+1},\hat\by)-\cL(\hat\bx,\by^{k+1})\geq 0$, \cref{lem:bar-P-ik} implies that
		\begin{align*}
			0\leq t_k Q^k(\hat\bz)-t_{k+1}Q^{k+1}(\hat\bz)+t_k\sum_{i\in\cN}\bar P_i^k- \frac{t_k}{2}\Big[\frac{1}{\gamma^k} - \sum_{i\in\cN}d_{i}\Big(\frac{\eta^k}{c_\varsigma}\tau_i^{k-1}+\frac{\sa{2}}{c_\alpha}\tau_i^{k}\Big)\Big]\norm{\lambda^{k+1}-\lambda^k}^2.
		\end{align*}
		Recall that $\gamma^k=\frac{c_\gamma}{\bar\tau}\Big(\frac{2}{c_\alpha}+\frac{\eta^k}{c_\varsigma}\Big)^{-1}$ for some $c_\gamma\in \Big(0,\frac{1}{2|\cE|}\Big)$; hence, using $\eta^k\geq 1$, $\tau_i^k\leq\tau_i^{k-1}\leq\bar\tau_i\leq\bar\tau$ and $\sum_{i\in\cN}d_i=2|\cE|$, we get $\frac{1}{\gamma^k} - \sum_{i\in\cN}d_{i}\Big(\frac{\eta^k}{c_\varsigma}\tau_i^{k-1}+\frac{\sa{2}}{c_\alpha}\tau_i^{k}\Big)\geq \bar\tau\Big(\frac{1}{c_\gamma}-\sum_{i\in\cN}d_i\Big)\Big(\frac{2}{c_\alpha}-\frac{1}{c_\varsigma}\Big)>0$. Finally, we define $\delta'\triangleq \hat t\cdot \bar\tau\Big(\frac{1}{c_\gamma}-\sum_{i\in\cN}d_i\Big)\Big(\frac{2}{c_\alpha}-\frac{1}{c_\varsigma}\Big)>0$. Thus, using \eqref{eq:error-bound-k} and the fact that $0<\hat t\leq t_k\leq t_0=1$, we obtain a crucial inequality:
		\begin{equation}
			\label{eq:key-inequality-convergence}
			\begin{aligned}
				0\leq &t_k Q^k(\hat\bz)-t_{k+1}Q^{k+1}(\hat\bz)+\mathbf{1}_{\cI}(k)\sum_{i\in\cN}\Xi^k_i-\frac{\delta'}{2}\norm{\lambda^{k+1}-\lambda^k}^2\\
				&\quad -\frac{\delta}{2}\sum_{i\in\cN}\Big(\frac{1}{\tau_i^k}\norm{x_i^{k+1}-x_i^k}^2 + \frac{1}{\sigma_i^k}\norm{\theta_i^{k+1} -\theta_i^k}^2\Big),\quad\forall~k\geq 0.
			\end{aligned}
		\end{equation}
		Next, let $a_k\triangleq t_k Q^k(\hat\bz)$, $b_k\triangleq \frac{\delta}{2}\sum_{i\in\cN}\Big(\frac{1}{\tau_i^k}\norm{x_i^{k+1}-x_i^k}^2 + \frac{1}{\sigma_i^k}\norm{\theta_i^{k+1} -\theta_i^k}^2\Big)+\frac{\delta'}{2}\norm{\lambda^{k+1}-\lambda^k}^2$, and $c_k\triangleq \mathbf{1}_{\cI}(k)\sum_{i\in\cN}\Xi^k_i$ for all $k\geq 0$. Recall that $\Xi_i^k=2 L_{f_i}D_i^2+\frac{\bar{\tau}_i}{c_\alpha}C_{g_i}^2\|\theta_i^{k+1}-\theta_i^k\|^2$ for $i\in\cN$ such that $L_{g_i}=0$; on the other hand, $\Xi_i^k=2 L_{f_i}D_i^2+4\frac{\bar{\tau}_i}{c_\alpha}C_{g_i}^2B_i^2$ for $i\in\cN$ such that $L_{g_i}>0$. Clearly, $a_k,b_k,c_k\geq 0$ such that $a_{k+1}\leq a_k-b_k+c_k$ for all $k\geq 0$. Recall that $|\cI|$ is bounded from above according to \cref{lem:I-cardinality}; therefore, we can conclude that $\sum_{k=0}^{+\infty}c_k<\infty$. This implies that $\lim_{k\to\infty}a_k\geq 0$ exists and $\sum_{k=0}^{+\infty}b_k<+\infty$. 
		Since $\{a_k\}$ is convergent, it must be bounded; hence, \eqref{eq:a-lower-bound} implies that $\{(x_i^k,\theta_i^k)\}_{k\geq 0}$ for all $i\in\cN$ and $\{\lambda^k\}_{k\geq 0}$ are bounded sequences-- note that even if $\{x_i^k\}_{k\geq 0}$ must be bounded by Assumption~\ref{assmp:f}, for the case $g_i(\cdot)$ is affine we set $\cB_i=\reals^{m_i}$, and this step of the proof is essential to establish $\{\theta_i^k\}_{k\geq 0}$ stays bounded even for the setting $g_i(\cdot)$ is affine and we set $\cB_i=\reals^{m_i}$.
\end{proof}
\begin{lemma}
		Under the premise of \cref{thm:line-search-diff}, it holds that
		\begin{equation}
			\label{eq:error-control}
			\begin{aligned}
				\sum_{k=0}^{\infty}t_k\Big(\sum_{i\in\cN} \bar P_i^k\Big)
				\leq -\frac{\delta}{2}\sum_{i\in\cN}\sum_{k=0}^{\infty}\Big(\frac{1}{\tau_i^k}\norm{x_i^{k+1}-x_i^k}^2 + \frac{1}{\sigma_i^k}\norm{\theta_i^{k+1} -\theta_i^k}^2\Big)+\Xi,
			\end{aligned}
		\end{equation}
		where $\Xi\triangleq \Big(\Big\lfloor\log_{1/\rho}\Big(\max_{i\in\cN}\frac{\bar\tau_i}{\hat\tau_i}\Big)\Big\rfloor+1\Big) \sum_{i\in\cN}\Xi_i$ such that for $i\in\cN$ with $L_{g_i}>0$, $\Xi_i\triangleq 2\Big(L_{f_i}+\frac{\bar\tau_i}{c_\beta}B_i^2L_{g_i}^2\Big)D_i^2 + 4\frac{\bar\tau_i}{c_\alpha}C_{g_i}^2B_i^2$ using $B_i$ given in Assumption~\ref{assmp:bounded_dual_domain}, and for $i\in\cN$ with $L_{g_i}=0$, $\Xi_i=2 L_{f_i}D_i^2+4\frac{\bar{\tau}_i}{c_\alpha}C_{g_i}^2B_i^2$ for $B_i>0$ given in \cref{thm:bound}.
	\end{lemma}
	\begin{proof}
		The desired results immediately follows from the bound on the cardinality of $\cI$ given in \cref{lem:I-cardinality} and the fact that $t_k\leq t_0=1$ for all $k\geq 0$.
\end{proof}
\begin{theorem}
\sa{Under the premise of \cref{thm:line-search-diff}. Let $\{x_i^k,\theta_i^k\}_{k\geq 0}$ be the \alg{} iterate sequence for $i\in\cN$, and $\{s_i^k\}_{k\geq 0}$ denote the \alg{} auxiliary sequence with $s_i^0=0$ for $i\in\cN$, and define $\{\lambda^k\}_{k\geq 0}$ such that $\lambda^k=A \bs^k$ with $\bs^k=[s_i^k]_{i\in\cN}$ for all $k\geq 0$. For any
$\bx \in \dom \phi$, $\theta\in\cK^*\cap\cB$ and $\lambda \in \reals^{n|\cE|}$,
let $\bz \triangleq (\bx,\by)$ with $\by \triangleq (\theta,\lambda)$.
Then, for any $K\geq 1$, it holds that
\begin{equation}\label{eq:tele_sum_final}
    \begin{aligned}
        \MoveEqLeft \mathcal{L}(\bar{\bx}^K, \by) - \mathcal{L}(\bx, \bar\by^K) + \frac{u_K}{2T_K} \|\lambda - \lambda^K\|^2 \\
         & + \frac{t_{K}}{2T_K} \sum_{i\in\mathcal{N}} \left( \frac{1 - (c_\alpha +c_\beta)}{\tau_i^{K}} \|x_i - x_i^K\|^2 + \frac{1}{\sigma_i^{K}} \|\theta_i - \theta_i^K\|^2 \right)\leq \frac{1}{T_K}\bar\Gamma(\bx,\by),\\
         \MoveEqLeft\bar\Gamma(\bx,\by) \triangleq \sum_{i\in\mathcal{N}} \left( \frac{1}{2{\tau}_i^0} \|x_i - x_i^0\|^2 + \frac{1}{2{\sigma}_i^0} \|\theta_i - \theta_i^0\|^2\right) + \frac{1}{2\gamma^0} \|\lambda\|^2 + \Xi,
    \end{aligned}
\end{equation}
where $u_K\triangleq t_{K-1}\Big(\frac{1}{\gamma^{K-1}} - \sum_{i\in \cN}\frac{{d_i}}{\varsigma_i^{K-1}} \Big)>0$, $T_K=\sum_{k=0}^{K-1}t_k$, and $(\bar\bx^K,\bar\by^K)\triangleq \frac{1}{T_K}\sum_{k=0}^{K-1}t_k(\bx^k,\by^k)$.}
\end{theorem}
\begin{proof}
Fix an arbitrary $\bz=(\bx,\by)$ such that $\bx\in\dom \phi$ and $\theta\in\cK^*\cap\cB$. Since the hypothesis of \cref{lem:rough-apd-bound} holds for \alg{} parameters and $t_0=1$, it follows that
\begin{equation*} 
             T_K \left(\mathcal{L}(\bar{\bx}^{K}, \by)  - \mathcal{L}(\bx, \bar{\by}^{K}) \right) \leq Q^0(\bz) - t_{K-1} R^K(\bz)+\sum_{k=0}^{K-1}t_kP^k,
\end{equation*}
holds for any $K\geq 1$, where $T_K=\sum_{k=0}^{K-1}t_k$ and $(\bar\bx^K,\bar\by^K)\triangleq \frac{1}{T_K}\sum_{k=0}^{K-1}t_k(\bx^k,\by^k)$. Therefore, \cref{lem:bar-P-ik} and \cref{thm:line-search-diff} together imply that
\begin{equation}\label{eq:rough-apd-bound-2}
\begin{aligned}
    \mathcal{L}(\bar{\bx}^{K}, \by)  - \mathcal{L}(\bx, \bar{\by}^{K}) + \frac{t_{K-1}}{T_K} R^K(\bz) \leq (Q^0(\bz)+\Xi)/T_K,
\end{aligned}
\end{equation}
where $\Xi$ is defined in~\eqref{eq:error-control}. Next, using the definitions of $Q^0(\bz)$ in \eqref{eq:Qk-def} and $R^K(\bz)$ in \eqref{eq:Rk-def} together with the fact that $q^0=0$, $\bx^{-1}=\bx^0$, $\theta^{-1}=\theta^0$, $\lambda^0=0$ and $t_0=1$, \eqref{eq:rough-apd-bound-2} implies that
{\small
\begin{equation}\label{eq:tele_sum_original}
    {\displaystyle
    \begin{aligned}
        \sa{\Gamma^K(\bz)} &\triangleq \mathcal{L}(\bar{\bx}^K, \by) - \mathcal{L}(\bx, \bar\by^K) + \frac{1}{2\gamma^{K-1}} \frac{t_{K-1}}{T_K}\|\lambda - \lambda^K\|^2 + \sum_{i\in\mathcal{N}} \frac{\varsigma_i^K}{2}\frac{t_{K-1}}{T_K} \|x_i^K - x_i^{K-1}\|^2  \\
        &\quad + \frac{t_{K-1}}{2T_K} \sum_{i\in\mathcal{N}} \left( \frac{1}{\tau_i^{K-1}}  \|x_i - x_i^K\|^2 + \frac{1}{\sigma_i^{K-1}} \|\theta_i - \theta_i^K\|^2 + \frac{1}{\alpha_i^K} \|q_i^{K,\by}\|^2 + \frac{1}{\beta_i^K} \|q_i^{K,\bx}\|^2 \right) \\
& \quad + \frac{t_{K-1}}{T_K} \langle q^K, \bx - \bx^K \rangle + \frac{t_{K-1}}{T_K} \langle A\bx^K - A\bx^{K-1}, \lambda^K - \lambda \rangle \\
& \leq \frac{1}{2T_K}\sum_{i\in\mathcal{N}} \left( \frac{1}{\tau_i^0} \|x_i - x_i^0\|^2 + \frac{1}{\sigma_i^0} \|\theta_i - \theta_i^0\|^2\right) + \frac{1}{2\gamma^0 T_K} \|\lambda\|^2 + \frac{\Xi}{T_K}.
    \end{aligned}}%
\end{equation}}%
Using the same arguments we employed while deriving \eqref{eq:CS-qk} and \eqref{eq:inner_prod_A}, we can further lower bound the left hand side of \eqref{eq:tele_sum_original} as below:
\begin{equation*}\label{eq:telesum_lowbd}
    \begin{aligned}
 \sa{\Gamma^K(\bz)} &\geq \mathcal{L}(\bar{\bx}^K, \by) - \mathcal{L}(\bx, \bar\by^K) + \frac{1}{2\gamma^{K-1}} \frac{t_{K-1}}{T_K}\|\lambda - \lambda^K\|^2 +  \frac{t_{K-1}}{2T_K} \sum_{i\in\mathcal{N}} \varsigma_i^K\|x_i^K - x_i^{K-1}\|^2\\
         & \quad +\frac{t_{K-1}}{2T_K} \sum_{i\in\mathcal{N}} \left( \frac{1}{\tau_i^{K-1}}  \|x_i - x_i^K\|^2 + \frac{1}{\sigma_i^{K-1}} \|\theta_i - \theta_i^K\|^2 + \frac{1}{\alpha_i^K} \|q_i^{K,\by}\|^2 + \frac{1}{\beta_i^K} \|q_i^{K,\bx}\|^2 \right) \\
 & \quad -\frac{t_{K-1}}{2T_K} \sum_{i\in\cN} \Big( \frac{1}{\alpha_i^K}\norm{q_i^{K,\by}}^2 + \frac{1}{\beta_i^K} \norm{q_i^{K,\bx}}^2 + (\alpha_i^K +\beta_i^K) \norm{x_i-x_i^k}^2\Big)\\
 & \quad -\frac{t_{K-1}}{2T_K} \sum_{i\in\cN} \Big(\varsigma_i^{K-1}\norm{x_i^K- x_i^{K-1}}^2 + \frac{{d_i}}{\varsigma_i^{K-1}} \norm{\lambda-\lambda^K}^2\Big)\\
 &\geq \mathcal{L}(\bar{\bx}^K, \by) - \mathcal{L}(\bx, \bar\by^K) + \frac{t_{K-1}}{2T_K} \Big(\frac{1}{\gamma^{K-1}} - \sum_{i\in \cN}\frac{{d_i}}{\varsigma_i^{K-1}} \Big)\|\lambda - \lambda^K\|^2 \\
         & \quad +\frac{t_{K-1}}{2T_K} \sum_{i\in\mathcal{N}} \left( \frac{1 - (c_\alpha +c_\beta)}{\tau_i^{K-1}} \|x_i - x_i^K\|^2 + \frac{1}{\sigma_i^{K-1}} \|\theta_i - \theta_i^K\|^2 \right),
    \end{aligned}
\end{equation*}
where in the first inequality we used Young's inequality twice, and the second inequality is due to $\varsigma_i^K \geq \varsigma_i^{K-1}$, \sa{which follows from $\tau_i^{K}\leq \tau_i^{K-1}$ for all $i\in\cN$}. \sa{Furthermore, since \alg{} parameter satisfy \textbf{Condition}~\ref{cond:step}, we have} $\gamma^{K-1} \leq \Big(\sum_{i\in\cN}d_i(\frac{{2}}{\alpha_{i}^{K}}+\frac{\eta^{K-1}}{\varsigma_{i}^{K-1}})\Big)^{-1}$, \sa{which together with $\eta^{K-1}\geq 1$ implies that} $\frac{1}{\gamma^{K-1}} - \sum_{i\in \cN}\frac{{d_i}}{\varsigma_i^{K-1}} >0$, and $\frac{1 - (c_\alpha +c_\beta)}{2\tau_i^{K-1}}>0$ follows from $1-c>\delta$; hence, we can conclude using \eqref{eq:tk-rule} to argue that $t_{K-1}/\tau_i^{K-1}\geq t_{K}/\tau_i^{K}$ and $t_{K-1}/\sigma_i^{K-1}\geq t_{K}/\sigma_i^{K}$ for $i\in\cN$.
\end{proof}

\begin{corollary}\label{coro:subopt}
\sa{Let $(x^*,\theta^*)\in\reals^n\times\cK^*$ denote an arbitrary primal-dual optimal pair for \eqref{eq:opt} as defined in \cref{assmp:bounded_dual_domain} such that $\norm{\theta_i^*}\leq B_i/2$ for $i\in\cN$ with $L_{g_i}>0$, and $\lambda^*\in\reals^{n|\cE|}$ be an optimal dual variable corresponding to the consensus constraint $A\bx=0$ in \eqref{pbm-to-solve}. Under the premise of \cref{thm:line-search-diff}, initialized from arbitrary $x_i^0\in\dom\phi_i$ and $\theta_i^0=0$ for all $i\in\cN$, the ergodic iterate sequences of \alg{}, i.e., $\{(\bar{x}_i^k, \bar{\theta}_i^k)\}_{k\geq 0}$ for all $i\in\cN$, satisfy}
\begin{align}
\textbf{(i) Suboptimality:} \quad & |\sum_{i\in\cN}\varphi_i(\bar{x}_i^K) - \sa{\varphi^*}| \leq \frac{\sa{\Gamma_0}}{T_K}=\sa{\cO(1/K)}, \label{eq:primal-subopt}\\
\textbf{(ii) Infeasibility:} \quad & \sum_{i\in\cN}  \norm{\theta_i^*}~d_{\sa{-\cK_i}} \Big(g_i(\bar{x}_i^K)\Big)  + \norm{\lambda^*}~\| A \bar{\bx}^K \|  \leq \frac{\Gamma_0}{T_K}=\sa{\cO(1/K)}, \label{eq:primal-infeasibility}
\end{align}
where $\Gamma_0\triangleq\sum_{i\in\mathcal{N}} \left(\frac{1}{2{\tau}_i^0} \|x^* - x_i^0\|^2 + \frac{2}{{\sigma}_i^0} \sa{\| \theta_i^*\|^2}\right) + \frac{2}{\gamma^0} \sa{\| \lambda^* \|^2} + \Xi$, $\bx^*\triangleq \mathbf{1}_N\otimes x^*$ and $\by^*=(\theta^*,\lambda^*)$.
\end{corollary}
\begin{proof}
\mgc{Consider} $\bx^*\mgc{=} \mgc{\mathbf{1}_N}\otimes x^*$
and note that $(\bx^*,\theta^*,\lambda^*)\in\dom\phi\times (\cK^*\cap\cB)\times \reals^{n|\cE|}$ is a saddle point for $\min_{\bx\in\cX}\max_{\theta,\lambda}\cL(\bx,\theta,\lambda)$. Therefore, it follows from \eqref{eq:tele_sum_final} that
\begin{equation*}
    \mathcal{L}(\bar{\bx}^K, \theta^*, \lambda^*) - \mathcal{L}(\bx^*, \bar{\theta}^K, \bar{\lambda}^K)
\leq \sa{\bar\Gamma(\bx^*,\by^*)}/T_K.
\end{equation*}
Define $\sa{\tilde{\bw}} \mgc{\triangleq}[\tilde{w}_i]_{i \in \mathcal{N}}$ such that
$\tilde{w}_i \triangleq g_i (\bar{x}_i^K)  \in \mathbb{R}^{m_i}$ for all $i \in \mathcal{N}.$
Since $\mathcal{K}_i$ is a closed convex cone, it induces a decomposition on $\mathbb{R}^{m_i}$:
\begin{equation}
\label{eq:w1w2}
    \tilde{w}_i = \tilde{w}_{i}^{(1)} + \tilde{w}_{i}^{(2)}\quad\mbox{s.t.}\quad \tilde{w}_{i}^{(1)} \triangleq \cP_{\sa{-\mathcal{K}_i}}(\tilde{w}_i),\quad \tilde{w}_{i}^{(2)} \triangleq \cP_{\mathcal{K}_i^*}(\tilde{w}_i),
\end{equation}
\sa{which also satisfy $\fprod{\tilde{w}_{i}^{(1)},~\tilde{w}_{i}^{(2)}}=0$ for $i\in\cN$. Thus, for all $i\in\cN$,} $
    \| \tilde{w}_{i}^{(2)} \| = \| \cP_{\sa{-\mathcal{K}_i}}(\tilde{w}_i) - \tilde{w}_i \| \triangleq d_{\sa{-\mathcal{K}_i}}(\tilde{w}_i).$
Next, we define $\tilde{\theta} = [\tilde{\theta}_i]_{i \in \mathcal{N}}\sa{\in\cK^*\cap\cB}$ such that
\begin{equation*}
    \tilde{\theta}_i \triangleq \sa{2\| \theta_i^* \|} \cdot \frac{\tilde{w}_{i}^{(2)}}{\| \tilde{w}_{i}^{(2)} \|} \in\sa{\mathcal{K}_i^*\cap\cB_i},\quad\forall~i\in\cN.
\end{equation*}
Therefore, for each $i\in\cN$, we get
\begin{equation}
\label{eq:tilde-theta}
    \langle g_i (\bar{x}_i^K) , \tilde{\theta}_i \rangle = \sa{2\| \theta_i^* \|} \cdot \fprod{\tilde{w}_i^{(1)} + \tilde{w}_i^{(2)},~ \frac{1}{\| \tilde{w}_i^{(2)} \|}~\tilde{w}_i^{(2)}} = \sa{2\| \theta_i^* \|}~d_{\sa{-\mathcal{K}_i}}(\tilde{w}_i),
\end{equation}
where the second equality follows from the orthogonality $\tilde{w}_i^{(1)} \perp \tilde{w}_i^{(2)}$.

Similarly, we define $\tilde{\lambda} \triangleq \sa{2\| \lambda^* \|} \cdot \frac{A \bar{\bx}^K}{\| A \bar{\bx}^K \|}$; hence, from the definition, we get
\begin{equation}
\label{eq:tilde-lambda}
    \langle A \bar{\bx}^K, \tilde{\lambda} \rangle = \sa{2\| \lambda^* \|}\cdot \| A \bar{\bx}^K \|.
\end{equation}
\sa{Note that $\mathcal{L}(\bx^*, \bar\theta^K, \bar\lambda^K)\leq \mathcal{L}(\bx^*, \theta^*, \lambda^*)=\varphi^*$ since $(\bx^*, \theta^*, \lambda^*)$ is a saddle point of $\cL$, and that $\tilde\theta\in\cK^*\cap\cB$. Thus \eqref{eq:tilde-theta} and \eqref{eq:tilde-lambda} imply}
\begin{equation}
\label{eq:rough-suboptimality-bound}
\begin{aligned}
    \MoveEqLeft\varphi(\bar{\bx}^k) - \varphi(\bx^*) + \sa{2\| \lambda^* \|} \cdot \| A \bar{\bx}^K \| + \sum_{i \in \mathcal{N}} \sa{2\| \theta_i^* \|}~d_{-\mathcal{K}_i}(g_i (\bar{x}_i^K) )\\
    &=\mathcal{L}(\bar{\bx}^K, \tilde{\theta}, \tilde{\lambda}) - \mathcal{L}(\bx^*, \theta^*, \lambda^*)\leq \sa{\mathcal{L}(\bar{\bx}^K, \tilde{\theta}, \tilde{\lambda})-\mathcal{L}(\bx^*, \bar\theta^K, \bar\lambda^K)}.
\end{aligned}
\end{equation}
\sa{Moreover, since $\bx^*\in\dom\phi$ and $\tilde\theta\in\cK^*\cap\cB$,  \eqref{eq:tele_sum_final} also implies that
\begin{equation}
\label{eq:saddle-tilde}
    \mathcal{L}(\bar{\bx}^K, \tilde\theta, \tilde\lambda) - \mathcal{L}(\bx^*, \bar{\theta}^K, \bar{\lambda}^K)
\leq \sa{\bar\Gamma(\bx^*,\tilde\by)}/T_K,
\end{equation}
where $\tilde\by=(\tilde\theta,\tilde\lambda)$, and \eqref{eq:tele_sum_final} implies that
\begin{equation}
\label{eq:Gamma0}
\begin{aligned}
    \bar\Gamma(\bx^*,\tilde\by)
    &= \sum_{i\in\mathcal{N}} \left(\frac{1}{2{\tau}_i^0} \|x^* - x_i^0\|^2 + \frac{1}{2{\sigma}_i^0} \|\tilde\theta_i\|^2\right) + \frac{1}{2\gamma^0} \|\tilde\lambda\|^2 + \Xi=\Gamma_0.
\end{aligned}
\end{equation}
Thus, combining \eqref{eq:rough-suboptimality-bound}, \eqref{eq:saddle-tilde}, and \eqref{eq:Gamma0}, we obtain
\begin{equation}
\label{eq:suboptimality-upper-bound}
\varphi(\bar{\bx}^k) - \varphi(\bx^*) + \sa{2\| \lambda^* \|} \cdot \| A \bar{\bx}^K \| + \sum_{i \in \mathcal{N}} \sa{2\| \theta_i^* \|}~d_{-\mathcal{K}_i}(g_i (\bar{x}_i^K))\leq \Gamma_0/T_K.
\end{equation}}%
\sa{Moreover, using the saddle point property of $(\bx^*,\theta^*,\lambda^*)$ one more time, we see that
$$\mathcal{L}(\bar{\bx}^K, {\theta}^*, {\lambda}^*) - \mathcal{L}(\bx^*, \theta^*, \lambda^*) \geq 0.$$
It follows that }
\begin{equation}\label{eq:gap_fuc_pos}
    {\varphi(\bar{\bx}^k) - \varphi(\bx^*)} +  \langle \lambda^*,  A \bar{\bx}^K \rangle + \sum_{i \in \mathcal{N}}  \langle \theta_i^*,\sa{g_i (\bar{x}_i^K)} \rangle \geq 0.
\end{equation}
\sa{Recall that $g_i(\bar x_i^K)=\tilde w_i$ such that \eqref{eq:w1w2} holds for $i\in\cN$. Hence, we immediately obtain}
\begin{equation*}
    \langle \theta_i^* , g_i (\bar{x}_i^K) \rangle = \langle \theta_i^*,~\tilde{w}_i - \tilde{w}_i^{(1)} + \tilde{w}_i^{(1)} \rangle \leq \langle \theta_i^*,~\tilde{w}_i - \tilde{w}_i^{(1)} \rangle \leq \norm{\theta_i^*} d_{\sa{-\cK_i}} (g_i(\bar{x}_i^K)).
\end{equation*}
Thus  \eqref{eq:gap_fuc_pos} implies
\begin{equation}\label{eq:rough-infeasibility}
    \varphi(\bar{\bx}^k) - \varphi(\bx^*) +  \| \lambda^\star \| \cdot \| A \bar{\bx}^K \| + \sum_{i \in \mathcal{N}}  \| \theta_i^* \|~d_{\sa{-\mathcal{K}_i}}(g_i (\bar{x}_i^K) ) \geq 0.
\end{equation}
Combining \eqref{eq:suboptimality-upper-bound} and \eqref{eq:rough-infeasibility} immediately yields the desired result.
\end{proof}

\begin{corollary}
\label{cor:complexity}
\sa{Under the premise of Corollary \ref{coro:subopt}, let $\hat t\triangleq \rho \min_{i\in\cN}\hat\tau_i/\bar\tau_i$. Then, for any given $\epsilon>0$, the \alg{} ergodic iterate sequence $\{\bar x_i^k\}_{k\geq 0}$ for $i\in\cN$ satisfies}
\begin{equation}
\label{eq:eps-opt}
    |\sum_{i\in\cN}\varphi_i(\bar{x}_i^K) - \varphi^*| \leq \epsilon, \quad
\sum_{i\in\cN}  d_{-\cK_i} (g_i(\bar{x}_i^K))  \norm{\theta_i^*} + \| A \bar{\bx}^K \| \norm{\lambda^*}  \leq \epsilon,\quad\forall~K\geq K_\epsilon\triangleq \frac{\Gamma_0}{\sa{\hat t}\epsilon}.
\end{equation}
\sa{Moreover, for each agent \( i \in \mathcal{N} \), the total number of backtracking steps is upper bounded by $\log_{1/\rho}(1/\hat t)$. Then, to compute an \( \epsilon \)-optimal solution as in \eqref{eq:eps-opt}, the total number of gradient and projection evaluations required for each $i\in\cN$ is upper bounded by}
\begin{equation}
    \cO\left( \sa{\log_{1/\rho}\Big(\frac{1}{\hat t}\Big)\cdot\frac{\Gamma_0}{\hat t}}\cdot\frac{1}{\epsilon} \right).
\end{equation}
\end{corollary}
\begin{proof}
Due to Corollary \ref{coro:subopt}, in order to achieve the complexity, we have that $\frac{\Gamma_0}{T_K} \leq \epsilon$. \sa{Furthermore, \mgc{according} to \cref{lem:I-cardinality}, we have $t_k\geq \hat t\triangleq \rho \min_{i\in\cN}\hat\tau_i/\bar\tau_i$; therefore, $T_K\geq \hat t K$. Therefore, for all $K\geq K_\epsilon$, we observe that $\frac{\Gamma_0}{T_K} \leq \epsilon$ is guaranteed to hold. Finally, combining the iteration complexity with Corollary \ref{cor:contraction-bound} gives us the total gradient complexity for each node.}
\end{proof}

\begin{theorem}
	\label{thm:limit}
		Under the premise of \cref{thm:parameters}, given some arbitrary initial points $(x_i^0,\theta_i^0)\in\dom\phi_i\times \cK_i^*$ and the step size parameters $\bar\tau_i,\zeta_i>0$ for $i\in\cN$, and the contraction coefficient $\rho\in(0,1)$, let $\{x_i^k,\theta_i^k\}_{k\geq 0}$ denote the \alg{} iterate sequence for $i\in\cN$. Then, there exists $(x^*,\theta^*)$ a primal-dual optimal solution to \eqref{eq:opt} such that 
		$\{(x_i^k,\theta_i^k)\}_{k\geq 0}$ converges {to} $(x^*,\theta_i^*)$ for all $i\in\cN$, i.e., $\lim_{k\to\infty}x_i^k=x^*$ and $\lim_{k\to\infty}\theta_i^k=\theta_i^*$ for $i\in\cN$.
\end{theorem}
\begin{proof}
	Let $(\hat\bx,\hat\theta,\hat\lambda)$ be an arbitrary saddle point of $\cL$. In the proof of \cref{thm:bound} we show that $\{a_k\}$ converges, where $a_k\triangleq t_k Q^k(\hat\bz)$ for $k\geq 0$; hence, it follows from \eqref{eq:a-lower-bound} that $\{(x_i^k,\theta_i^k)\}_{k\geq 0}$ for all $i\in\cN$ and $\{\lambda^k\}_{k\geq 0}$ are bounded sequences. Let $\{k_n\}_{n\geq 0}\subset\integers_+$ such that $\lim_{n\to\infty}(x_i^{k_n},\theta_i^{k_n})=(x^*_i,\theta^*_i)$ for $i\in\cN$ and $\lambda^{k_n}\to\lambda^*$, where $\bx^*=[x_i^*]_{i\in\cN}$ and $\by^*=(\theta^*,\lambda^*)$ be an arbitrary limit point --here,  $\theta^*=[\theta_i^*]_{i\in\cN}$. Let $\by^k=(\theta^k,\lambda^k)$ for $k\geq 0$. In the proof of \cref{thm:bound}, we also showed $\sum_{k\geq 0}b_k<\infty$ where $b_k\triangleq \frac{\delta}{2}\sum_{i\in\cN}\Big(\frac{1}{\tau_i^k}\norm{x_i^{k+1}-x_i^k}^2 + \frac{1}{\sigma_i^k}\norm{\theta_i^{k+1} -\theta_i^k}^2\Big)+\frac{\delta'}{2}\norm{\lambda^{k+1}-\lambda^k}^2$ for $k\geq 0$. Since $\tau_i^k\leq\bar\tau_i>0$ and $\sigma_i^k\leq\zeta_i \bar\tau_i>0$ for all $k\geq 0$ and $i\in\cN$, $\sum_{k\geq 0}b_k<\infty$ implies that $\sum_{k\geq 0}\norm{\bx^{k+1}-\bx^k}^2+\norm{\by^{k+1}-\by^k}^2<\infty$. Thus, for any given $\epsilon>0$, there exists $N_1\in\integers_+$ such that for $n\geq N_1$, one has $\norm{\bx^{k_n+1}-\bx^{k_n}}^2+\norm{\by^{k_n+1}-\by^{k_n}}^2\leq \frac{\epsilon}{4}$ and $\norm{\bx^{k_n-1}-\bx^{k_n}}^2+\norm{\by^{k_n-1}-\by^{k_n}}^2\leq \frac{\epsilon}{4}$. Moreover, there also exists $N_2\in\integers_+$ such that for $n\geq N_2$, one has $\norm{\bx^{k_n}-\bx^*}^2+\norm{\by^{k_n}-\theta^*}^2\leq \frac{\epsilon}{4}$. Hence, for all $n\geq \max\{N_1,N_2\}$, it must hold that $\norm{\bx^{k_n+1}-\bx^*}^2+\norm{\by^{k_n+1}-\by^*}^2\leq \epsilon$ and $\norm{\bx^{k_n-1}-\bx^*}^2+\norm{\by^{k_n-1}-\by^*}^2\leq \epsilon$, i.e.,
		\begin{equation}
			\label{eq:pm-limit}
			\lim_{n\to\infty}(x_i^{k_n},y_i^{k_n})=\lim_{n\to\infty}(x_i^{k_n+1},y_i^{k_n+1})=\lim_{n\to\infty}(x_i^{k_n-1},y_i^{k_n-1})=(x^*_i,y^*_i),\quad \forall~i\in\cN.
	\end{equation}
	
	Next, we argue that $(\bx^*,\theta^*,\lambda^*)$ is indeed a saddle point of $\cL$. First, recall that the sequence $\{\lambda^k\}_{k\geq 0}$ defined as $\lambda^k=A\bs^k$ for $\bs^k=[s_i^k]_{i\in\cN}$ satisfies
		{\small
			\begin{align}
				&(\lambda^{k_n+1}-\lambda^{k_n})/\gamma^{k_n}=A\big((1+\eta^{k_n})\bx^{k_n}-\eta^{k_n}\bx^{k_n-1}\big)\label{eq:lambda-subseq}\\
				&\bp^{k_n} = A^\top \Big((1+\eta^{k_n})\lambda^{k_n} - \eta^{k_n} \lambda^{k_n-1}\Big) +(1+\eta^{k_n})\J G(\bx^{k_n})^\top \theta^{k_n} -\eta^{k_n} \J G(\bx^{k_n-1})^\top \theta^{k_n-1},\label{eq:p-subseq}
		\end{align}}%
		which follows from $\bs$-update $\bs^{k+1} = \bs^k + \gamma^k(1+\eta^k)\bx^{k} - \gamma^k\eta^k \bx^{k-1}$ in \eqref{eq:s-local} and \eqref{eq:p-local}. Since $\gamma^k=\frac{c_\gamma}{\bar\tau}\Big(\frac{2}{c_\alpha}+\frac{\eta^k}{c_\varsigma}\Big)^{-1}$ for some $c_\gamma\in \Big(0,\frac{1}{2|\cE|}\Big)$ and $1\leq \eta^k=\max_{i\in\cN}\eta_i^k=\max_{i\in\cN}\frac{\tau_i^{k-1}}{\tilde\tau_i^k}\leq \max_{i\in\cN}\frac{\bar\tau_i}{\rho\hat\tau_i}\leq 1/\hat t$; hence, $\frac{c_\gamma}{\bar\tau}\Big(\frac{2}{c_\alpha}+\frac{1/\hat t}{c_\varsigma}\Big)^{-1}\leq \gamma^k\leq \frac{c_\gamma}{\bar\tau}\Big(\frac{2}{c_\alpha}+\frac{1}{c_\varsigma}\Big)^{-1}$ for all $k\geq 0$. Therefore, $(\lambda^{k_n+1}-\lambda^{k_n})/\gamma^{k_n}\to 0$ and $\eta^{k_n}(\bx^{k_n}-\bx^{k_n-1})\to 0$ as $n\to\infty$. Thus, \eqref{eq:lambda-subseq} implies that 
		\begin{equation}
			\label{eq:consensus-limit}
			A\bx^*=0;
		\end{equation}
		hence, there exists some $x^*\in\reals^n$ such that $\bx^*=\mathbf{1}_N\otimes x^*$. Furthermore, \eqref{eq:p-subseq} implies that $\bp^*\triangleq \lim_{n\to\infty}\bp^{k_n}=A^\top\lambda^*+\J G(\bx^*)^\top\theta^*$.
	
	Finally, consider the first-order optimality conditions for \eqref{eq:x-local} and \eqref{eq:theta-local} along the subsequence $\{k_n\}_{n\geq 0}\subset\integers_+$:
		\begin{align}
			&\frac{1}{\tau_i^{k_n}}(x_i^{k_n}-x_i^{k_n+1})\in\grad f_i(x_i^{k_n})+p_i^{k_n}+\partial\phi_i(x_i^{k_n+1})\label{eq:FO-x}\\
			&\fprod{\theta_i^{k_n+1}-\theta_i^{k_n}-\sigma_i^{k_n}g_i(x_i^{k_n+1}),~\theta_i-\theta_i^{k_n+1}}\geq 0,~\quad\forall~\theta_i\in\cK_i^*\cap\cB_i;\label{eq:FO-theta}
		\end{align}
		hence, because $\bar\tau_i\geq \tau_i^k\geq \rho \min_{i\in\cN}\hat\tau_i>0$ and $\sigma_i^k=\zeta_i\tau_i^k$ for some $\zeta_i>0$ for all $k\geq 0$, taking the limit on both sides of \eqref{eq:FO-x} and \eqref{eq:FO-theta}, and using \cite[Theorem 24.4]{rockafellar1997convex}, we get
		\begin{equation}
			\label{eq:opt-cond-1}
			0\in\grad f_i(x^*)+p^*_i+\partial\phi_i(x^*),\quad \fprod{-g_i(x^*),~\theta_i-\theta^*_i}\geq 0,\quad\forall~\theta_i\in\cK_i^*\cap\cB_i;
		\end{equation}
		hence, 
		\eqref{eq:consensus-limit} implies $(\theta^*,\lambda^*)\in\argmax\{\cL(\bx^*,\theta,\lambda):\ \theta\in\cK^*\cap\cB,\lambda\}$ and 
		\begin{equation}
			\label{eq:eq:opt-cond-2}
			0\in\grad f(\bx^*)+\J G(\bx^*)^\top\theta^*+A^\top\lambda^*+\partial\phi(\bx^*),
		\end{equation}
		which shows that $\bz^*\triangleq (\bx^*,\theta^*,\lambda^*)\in\dom \phi\times\dom h\times\reals^{n|\cE|}$ is a saddle point of $\cL$ as well. Since \eqref{eq:key-inequality-convergence} holds for any saddle point $\hat\bz$ of $\cL$, we can replace $\hat\bz$ with $\bz^*$ and the same arguments we used in the proof of \cref{thm:bound} continue to hold. Accordingly, we have $\bar a_{k+1}\leq \bar a_k-b_k+c_k$ for $k\geq 0$ where $\{\bar a_k\}_{k\geq 0}$ is such that $\bar a_k=t_k Q^k(\bz^*)\geq 0$ for $k\geq 0$. Since $\sum_{k\geq 0}c_k<\infty$, we have $\bar a\triangleq \lim_{k\to\infty} \bar a_k$ exists. Moreover, the subsequence $\{k_n\}_{n\geq 0}$ defined above is independent of $\bz^*$ or $\hat\bz$; hence, it follows from the definition of $Q^k(\bz^*)$ in \eqref{eq:Qk-def} and \eqref{eq:pm-limit} that $\lim_{n\to\infty} a_{k_n}=t_{k_n}Q^{k_n}(\bz^*)=0$. Since for a convergent sequence, every subsequence converges to the same limit point, we must have $\bar a=\lim_{k\to\infty}\bar a_k=0$. Thus, the inequality in \eqref{eq:a-lower-bound} with $\hat\bz$ replaced with $\bz^*$ and taking the limit as $k\to \infty$ implies that $\lim_{k\to\infty}\bx^k=\bx^*$. Since we have already observed that $\bx^*=\mathbf{1}_N\otimes x^*$ for some $x^*\in\reals^n$, we have $\lim_{k\to\infty}x_i^k=x^*$ for all $i\in\cN$. Furthermore, $t_k\geq\hat t>0$ for all $k\geq 0$ and $T_K\to\infty$ as $K\to\infty$ implies that the weighted sequence $\lim_{K\to\infty}\bar x_i^K=\lim_{K\to\infty}\bar x_i^K=x^*$; thus, $A\bx^*=0$ implies that $\lim_{K\to\infty}\norm{A\bar\bx^K}=0$. Finally, taking the limit on both sides of \eqref{eq:primal-subopt} and \eqref{eq:primal-infeasibility} as $K\to \infty$ implies that $\varphi^*=\sum_{i\in\cN}\varphi_i(x^*)$ and $d_{-\cK_i}\big(g_i(x^*)\big)=0$, i.e., $-g_i(x^*)\in\cK_i$, for all $i\in\cN$, which shows that $x^*$ is an optimal solution to \eqref{eq:opt}. On the other hand, setting $\theta_i=0$ within \eqref{eq:opt-cond-1} implies that $\fprod{-g_i(x^*),~\theta_i^*}\leq 0$ for $i\in\cN$; moreover, $-g_i(x^*)\in\cK_i$ and $\theta_i^*\in\cK^*$ imply that $\fprod{-g_i(x^*),~\theta_i^*}\geq 0$; thus, the complementary slackness conditions $\fprod{-g_i(x^*),~\theta_i^*}=0$ hold for all $i\in\cN$, which together with \eqref{eq:eq:opt-cond-2}, $\theta_i^*\in\cK_i^*$ and $-g_i(x^*)\in\cK_i$ for all $i\in\cN$, and $A\bx^*=0$ together imply that $\theta^*=[\theta_i^*]_{i\in\cN}$ is an optimal dual solution.
\end{proof}

\section{Experiments}\label{sec:exp}
\mgb{In this section, we provide numerical experiments to illustrate the performance of \alg{} and \algz{}. The former algorithm can handle \sa{node-specific functional constraints defined by $g_i(\cdot)$}; whereas the latter one is tailored for problems \sa{that are either unconstrained or with simple constraint sets onto which projections are cheap to compute. We test \alg{} over \mgc{$\ell_1$-norm-regularized} quadratically constrained quadratic programming (QCQP) problems and on primal support vector machine (SVM) training problems while \algz{} is tested over \mgc{$\ell_1$-norm-regularized} unconstrained quadratic programming problems.} For all experiments, we test our algorithms on a random small-world network $\cG = (\cN, \cE)$, i.e., we choose $|\cN|$
edges creating a random cycle over nodes, then the remaining $|\cE| -|\cN|$ edges are selected uniformly at random. In our experiments, we set $|\cE|=24$ edges and $|\cN|=12$ nodes. Our experiments are run on a {MacBook Air equipped with an Apple M2 CPU and 8 GB of unified memory}. Our code is available at \url{https://github.com/Qiushui-Xu/D-APDB}.}


\subsection{\mgc{QCQP problems with \mgc{$\ell_1$-norm regularization}.}}

We consider a QCQP problem of the following form:
\begin{subequations}\label{eq:qcqp}
\begin{align}
    \varphi^* \triangleq &\min_{x \in\cX}
    \norm{x}_1+\sum_{i\in\cN}\frac{1}{2}  x^\top Q_i x\\ 
    & \mbox{s.t.}\ g_i(x)\triangleq \frac{1}{2}(x-\bar{x}_i)^\top A_i  (x - \bar{x}_i) \leq 1, \quad i\in\cN,
\end{align}
\end{subequations}
i.e., \sa{for $i\in\cN$, $f_i(x) = \frac{1}{2}  x^\top Q_i x$ for some randomly generated \mgb{$Q_i\succeq 0$}, $\cX=[-10,10]^n$ and $\phi_i(x)= \frac{1}{N}\norm{x}_1$; hence, the local objective $\varphi_i=\phi_i+f_i$ is convex. For every agent $i\in\cN$, $Q_i\in\mathbb{S}^n$ is generated randomly as follows: $Q_i = V_i \Gamma_i V_i^\top$, where $V_i$ is a random orthonormal matrix satisfying $V_i V_i^\top = I$, and $\Gamma_i = \mathrm{diag}(\gamma_{i,1}, \gamma_{i,2}, \dots, \gamma_{i,n})$ is a diagonal matrix whose first element is $\gamma_{i,1} = 5i$, whose last two elements are $\gamma_{i,n-1} = \gamma_{i,n} = 0$, whose third-to-last element is $\gamma_{i,n-2} = 1$, and whose remaining elements are sampled from the uniform distribution $\mathcal{U}[1, 5i]$ supported on the interval [1,5i].}
\mgb{By construction, the elements of $\Gamma_i$, which coincide with the eigenvalues of $Q_i$, are sorted in a decreasing order.} \mgb{With this setup, $Q_i$ has zero eigenvalues} and therefore $f_i(x)$ is convex \mgb{(but not strongly convex)}. For generating matrices $A_i$, \mgb{we follow a similar approach} and set $A_i = U_i R_i U_i^\top$ where $U_i$ is a random orthonormal matrix such that $U_i U_i^\top = I$ and $R_i=\diag(r_{i,1}, r_{i,2},\dots, r_{i,n})$ is a diagonal matrix whose first element \mgb{$r_{i,1}=\frac{1}{4}$, the last element is $r_{i,n}=\frac{1}{16}$, and all the other elements are sampled from the uniform distribution $\mathcal{U}[\frac{1}{16}, \frac{1}{4}]$ and as before} all the elements are sorted in a decreasing order. \mgb{The constraints result in ellipsoid sets with a center at} $\bar{x}_i$. \mgb{We set its $j$-th coordinate to} $\bar{x}_i^j = 2 + \xi_i^j$, where $\xi_i^j$ is uniformly sampled from $[-\frac{1}{2\sqrt{n}}, \frac{1}{2\sqrt{n}}]$. In our experiments, we set $n=20$ and generate the initial point $x_0$ such that its entries are i.i.d. with uniform distribution $\mathcal{U}[-10,10]$.  

\mgb{For this class of randomly generated QCQP problems, we}
compare our \texttt{D-APDB} with \texttt{D-APD} 
in terms of log relative suboptimality, relative consensus error and infeasibility of the average iterate sequence $\{\bar x^k\}_{k\geq 0}\subset\reals^n$, i.e., \mgb{in Fig.~\ref{fig:qcqp-exp},} we plot
$\log(|\varphi(\bar{x}^k) -\varphi^*|/|\varphi^*| +1)$ in (A), $\sum_{i\in\cN}\norm{x_i^k -\bar{x}^k}^2/(N\norm{\bar{x}^k}^2)$ in (B) and  $\max_{i\in\cN} \norm{(g_i(\bar{x}^k))_+}/\max_{i\in\cN} \norm{(g_i(\bar{x}^0))_+}$ in (C), \sa{where $\bar x^k=\sum_{i\in\cN}x_i^k/N\in\reals^n$ for $k\geq 0$}. \mgc{The benchmark algorithm \texttt{D-APD} is essentially our proposed method \texttt{D-APDB} without backtracking; that is, it uses a constant stepsize $\bar{\tau}_i$ and sets $\eta_i^k = 1$ for all $k\geq 0$, rather than performing backtracking on these parameters$^5$.\footnote{\mgc{$^5$We note that \texttt{D-APD} is closely related to the method proposed in \cite{hamedani2021decentralized} for decentralized constrained strongly convex problems when the strong convexity parameter is set to zero.}} Comparing against this benchmark allows us to assess the benefits of incorporating backtracking.}

For the benchmark algorithm \texttt{D-APD} \mgc{which is based on constant stepsize,} we set the initial stepsizes $\bar\tau_i=\mgc{\hat\tau_i}$ \mgc{where $\hat\tau_i$ is as in \eqref{eq:hat-tau}}, $\zeta_i = 1$, $c_\alpha,c_\beta, c_\varsigma = 0.1$ and $\eta^k = 1$ for $k\geq 0$. \mgb{This stepsize requires knowledge of the Lipschitz constants $L_{f_i}$ and constants $L_{g_i}$; therefore, we \sa{compute these constants using the randomly generated}   matrices $A_i, Q_i$$^6$}\footnote{\mgb{$^6$Note $\grad f_i(x)=Q_i x$; hence, we can take $L_{f_i}=\norm{Q_i}_2$ and $L_{g_i}$ can be computed similarly as $L_{g_i}=\norm{A_i}_2$.}}.
\mgb{For \alg{}, because it can potentially support larger stepsizes that adapt  to the local curvature, we set the initial stepsizes to be larger than \mgc{$\hat\tau_i$}}. \mgc{More specifically, we set the initial stepsize of \alg{} (in \lin{algeq:init} of \alg{}) as $\bar{\tau}_i= \kappa \hat\tau_i$ 
with $\kappa = 20$}, 
\mgc{and we set}
the shrink factor $\rho = 0.9$, $\zeta_i = 1$, $c_\alpha,c_\beta, c_\varsigma = 0.1$, and $\delta=0.1$. 
All the other parameters are set according to our theoretical results. The results in Fig. \ref{fig:qcqp-exp} shows that our method outperforms \texttt{D-APD} on average over 20 simulations, each corresponding to a randomly generated problem instance, where the shaded areas demonstrate the standard deviation of the performance over the runs.

\begin{figure}[htbp]
  \centering
  \begin{subfigure}[t]{0.48\textwidth}
    \centering
    \includegraphics[width=\linewidth]{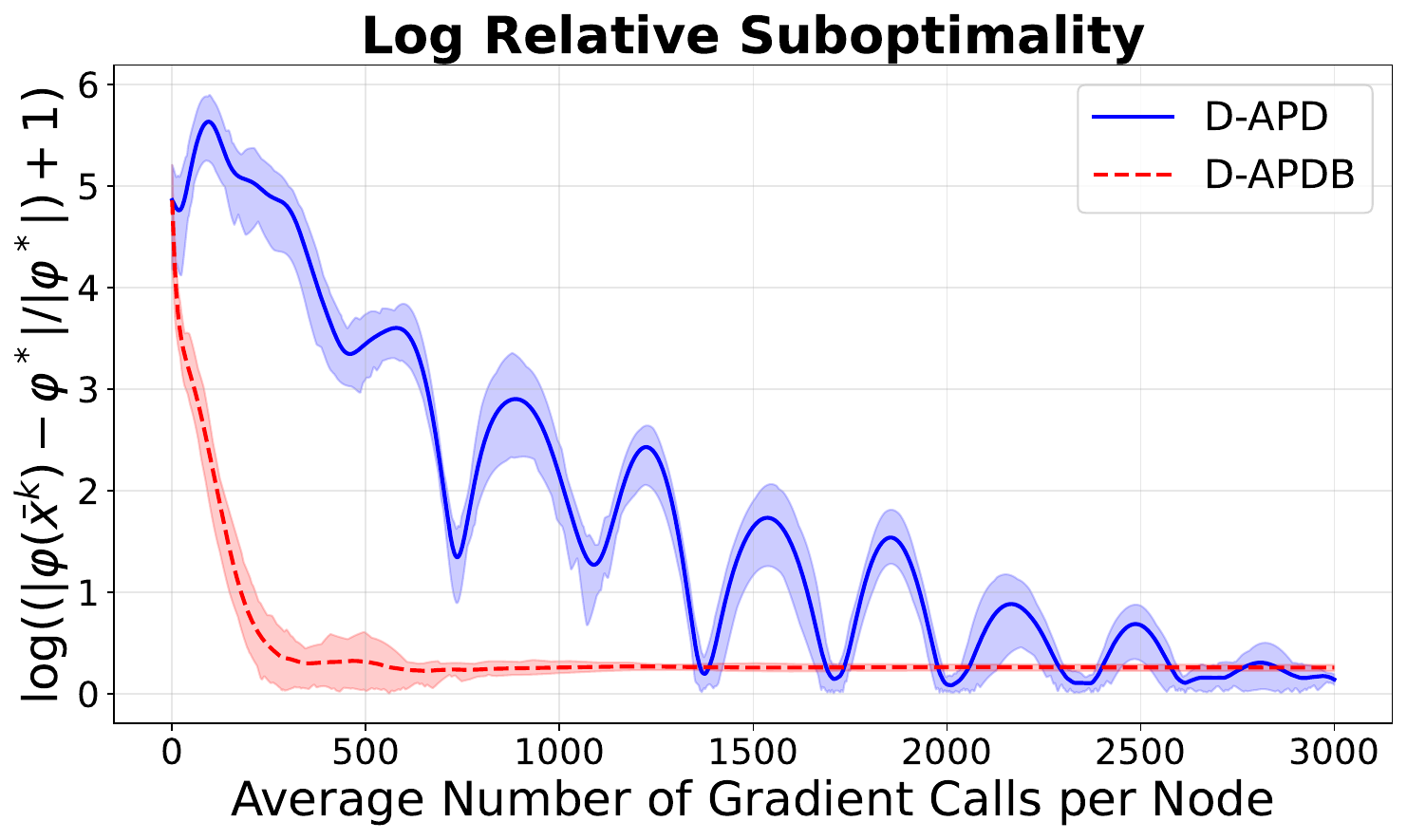}
    \captionsetup{font=scriptsize}
    \caption{Log Relative Suboptimality}
    \label{fig:qcqp-subopt}
  \end{subfigure}
  \hfill
  \begin{subfigure}[t]{0.48\textwidth}
    \centering
    \includegraphics[width=\linewidth]{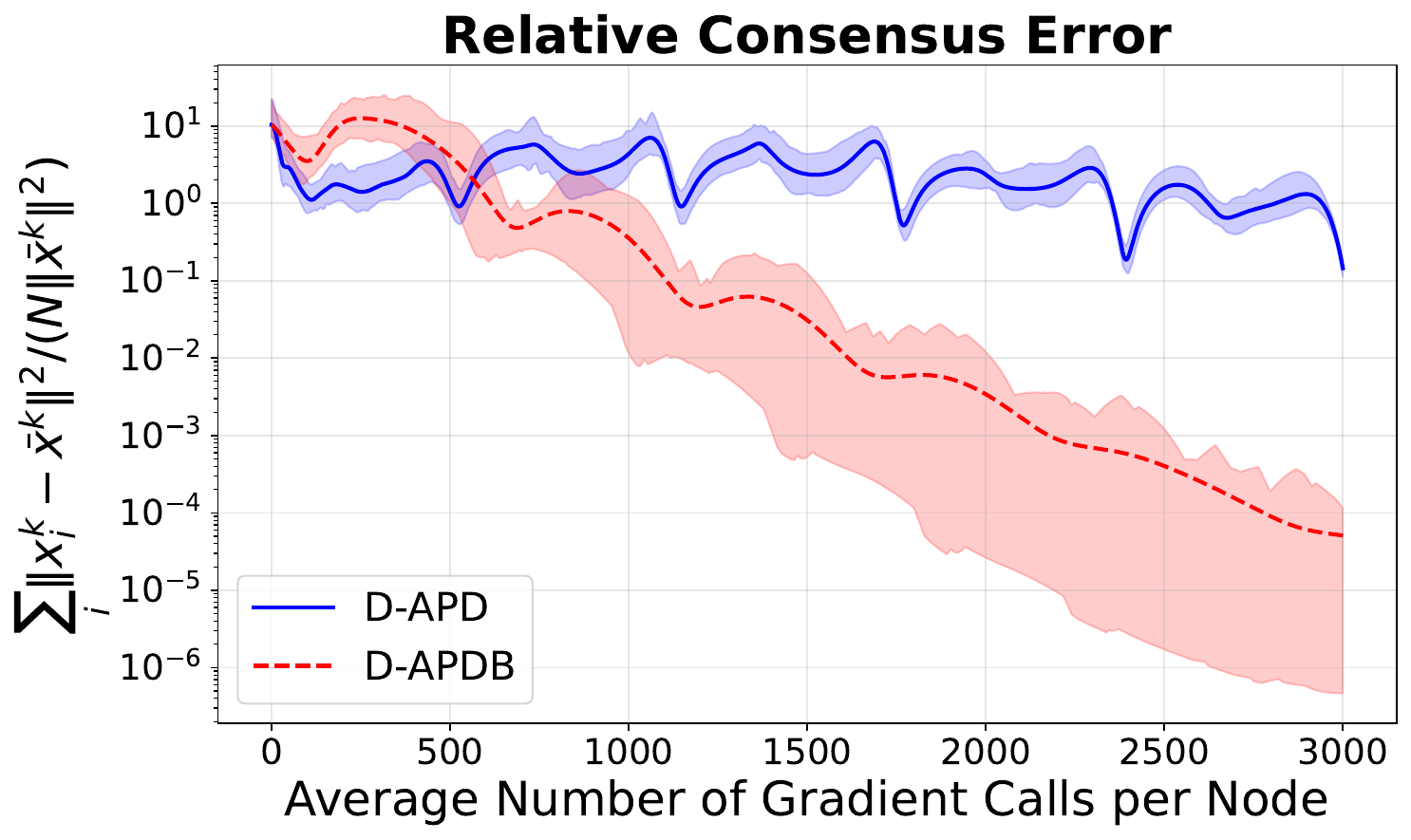}
    \captionsetup{font=scriptsize}
    \caption{Relative Consensus Error}
    \label{fig:qcqp-consensus}
  \end{subfigure}

  \begin{subfigure}[t]{0.48\textwidth}
    \centering
    \includegraphics[width=\linewidth]{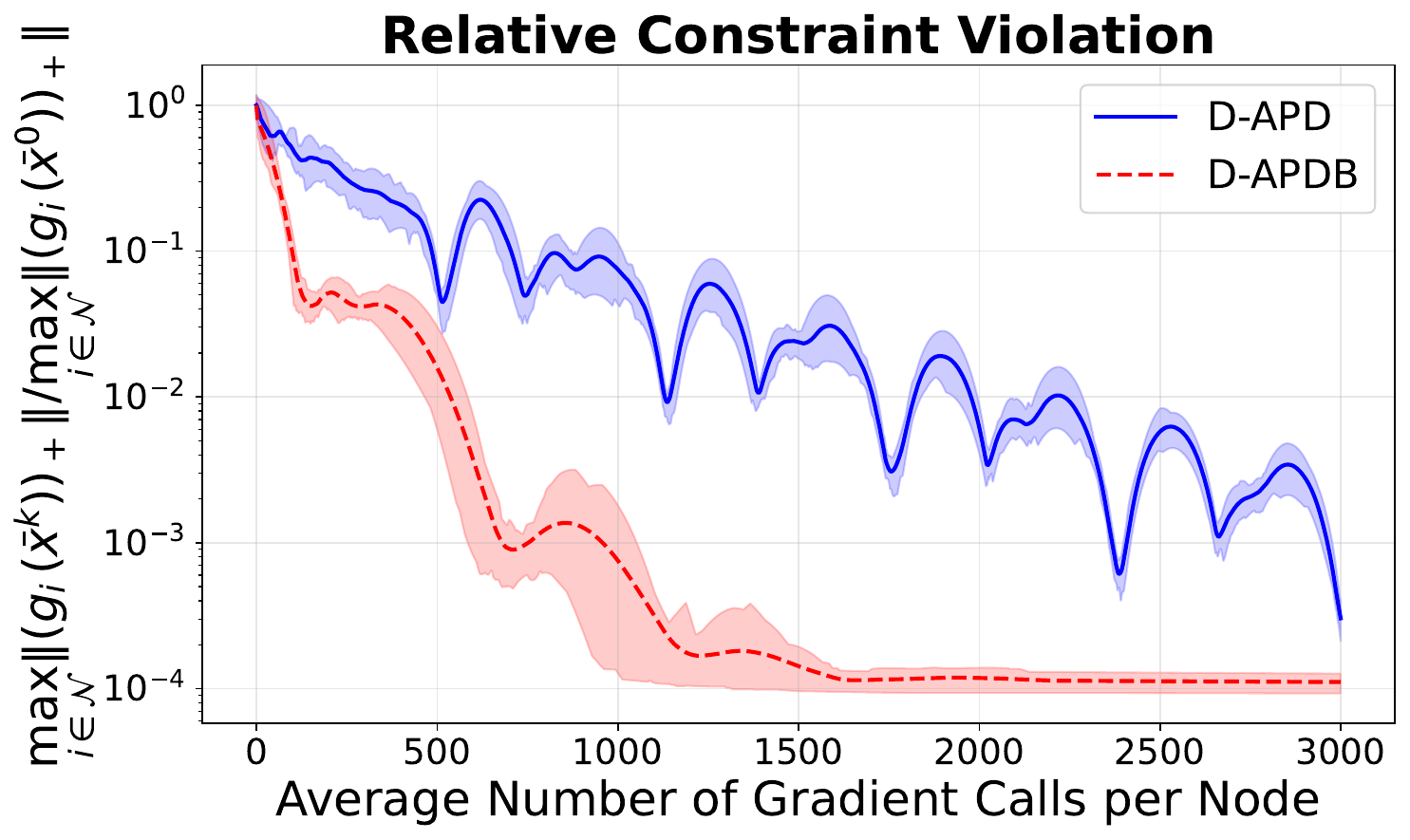}
    \captionsetup{font=scriptsize}
    \caption{Relative Constraint Violation}
    \label{fig:qcqp:constraint}
  \end{subfigure}
  \caption{Comparison of \alg{}~ against \texttt{D-APD} for solving \eqref{eq:qcqp} over 20 runs. {“Average number of gradient calls per node” means $\frac{1}{N}\sum_{i=1}^N n_i^k$ where $n_i^k$ counts how many calls for gradients for node $i$ at iteration $k$.}}
  \label{fig:qcqp-exp}
\end{figure}

\subsection{Unconstrained \mgc{$\ell_1$-norm-regularized} quadratic programming} We consider an unconstrained quadratic programming problem of the following form 
\begin{equation}\label{eq:qp}
    \varphi^* \triangleq \min_{x\in\cX} \norm{x}_1+\sum_{i\in\cN}
    \frac{1}{2} x^T Q_i x + q_i^T x + c_i,
\end{equation}
i.e., for $i\in\cN$, $f_i(x) =  \frac{1}{2} x^\top Q_i x + q_i^\top x + c_i$ for some randomly generated $Q_i\succeq 0$ and $\phi_i(x) = \frac{1}{N}\norm{x}_1$; hence, the local objective $\varphi_i = \phi_i + f_i$ is convex.

\sa{Let $\{{L}_{f_i}\}_{i\in\cN}$ be i.i.d. random variables sampled from the normal distribution with mean $1000$ and standard deviation $100$.} 
\mgb{The matrix $Q_i$ is generated randomly} where we set $Q_i = V_i \Gamma_i V_i^\top$ such that $V_i$ is a random orthonormal matrix satisfying $V_i V_i^\top = I$ and $\Gamma_i = \mathrm{diag}(\gamma_{i,1}, \gamma_{i,2},$ $ \dots, \gamma_{i,n})$ is a diagonal matrix
\sa{with the first diagonal element $\gamma_{i,1}={L}_{f_i}$, the last element $\gamma_{i,n}=0$}, 
and all the other elements are \sa{independently} sampled from \sa{$\mathcal{U}\Big[0, \min\big\{100, L_{f_i}\big\}\Big]$.} 
\sa{Moreover, every entry of $q_i$ is sampled from standard normal distribution independently} and $c_i$ is sampled from $\mathcal{U}[0,1]$. In this way, we can expect some 
\sa{significant variation among node-specific Lipschitz constants $\{{L}_{f_i}\}_{i\in\cN}$ --as $L_{f_i}=\norm{Q_i}_2$ due to $\grad f_i(x)=Q_i x + q_i$ for $i\in\cN$.}

For this class of randomly generated unconstrained $\ell_1$-norm-regularized QP problems, we compare our \texttt{D-APDB0} with \texttt{D-APD} and \texttt{global\_DATOS} \cite{chen2025parameter} in terms of log relative suboptimality, relative consensus error of the iterate sequence, i.e., in the two panels of Fig.~\ref{fig:qp-exp}, from left to right, we plot $\log(|\varphi(\bar{x}^k) -\varphi^*|/|\varphi^*| +1)$, $\sum_{i=1}^N\norm{x_i^k -\bar{x}^k}^2/(N\norm{\bar{x}^k}^2)$. Benchmarking against these methods allows us to evaluate the benefits of incorporating node-specific backtracking and employing node-specific stepsize choices. \sa{\texttt{global\_DATOS} uses two communication rounds among the neighboring nodes per iteration   while both \algz{} and \texttt{D-APD} require one communication round per iteration. In Fig.~\ref{fig:qp-exp} we plotted sub-optimality and consensus violation against the number of communication rounds in the x-axis.}

In the experiments, for the benchmark algorithm \texttt{D-APD}, we set the initial stepsizes to $\bar\tau_i = \frac{1}{2L_{f_i}}$ for each $i\in\cN$, and choose $c_\alpha = c_\varsigma = 0.4$ and $\eta_i^k = 1$ for all $k\geq 0$.
For \algz{}, because it can support larger stepsizes, 
we use a more aggressive initialization. Specifically, in Line~\ref{algzeq:init} of \algz{}, we set the initial stepsizes to $\bar\tau_i = \kappa \hat\tau_i$ with $\kappa = 5$, and choose the shrink factor $\rho = 0.9$, $c_\alpha = c_\varsigma = 0.4$, and $\delta = 0.1$. For \texttt{global\_DATOS}, node-specific stepsizes are not allowed, but backtracking enables the use of larger global stepsizes. Therefore, we set the initial stepsize for all nodes to $\bar\alpha = \kappa \hat{\alpha}$, where $\hat{\alpha} \triangleq \max_{i \in \cN} \hat\tau_i$ and again choose $\kappa = 5$ for a fair comparison between \algz{} and \texttt{global\_DATOS}. We follow the definition of Laplacian-based constant edge weight matrix $W=I-\frac{\Omega}{d_{\rm max} + 1}$ in \cite{shi2015extra} to generate the gossip matrix for \texttt{global\_DATOS}, where $\Omega$ is the Laplacian matrix of the graph $\cG$ and $d_{\rm max}$ is the max degree of the graph $\cG$. All other parameters follow the theoretical prescriptions. The results in Fig. \ref{fig:qp-exp} show that our method consistently outperforms both \texttt{D-APD} and \texttt{global\_DATOS} on average for 20 simulations, where the shaded regions indicate the standard deviation across runs.

\begin{figure}[htbp]
  \centering
  \begin{subfigure}[t]{0.48\textwidth}
    \centering
    \includegraphics[width=\linewidth]{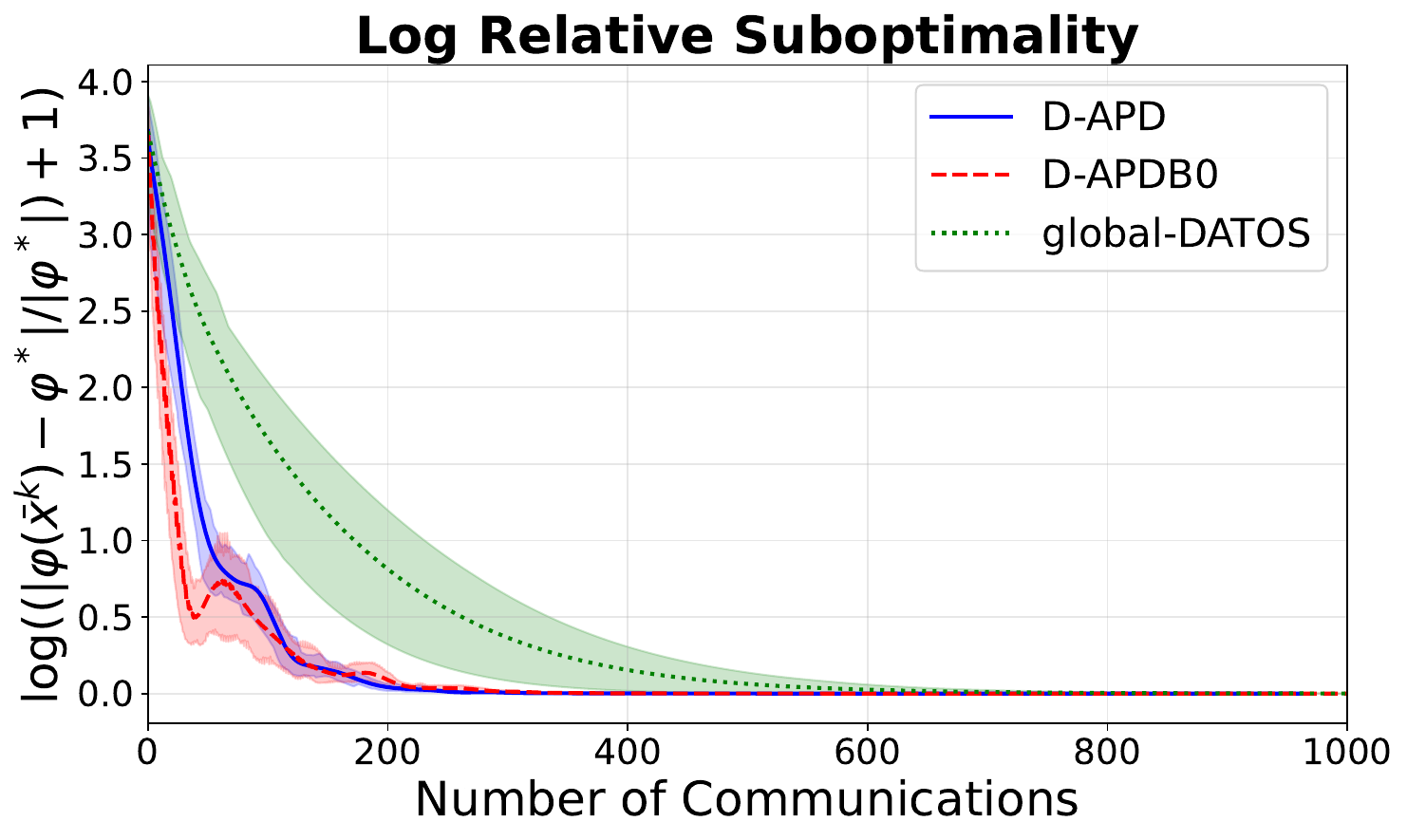}
    \captionsetup{font=scriptsize}
    \caption{Log Relative Suboptimality}
    \label{fig:qp-subopt}
  \end{subfigure}
  \hfill
  \begin{subfigure}[t]{0.48\textwidth}
    \centering
    \includegraphics[width=\linewidth]{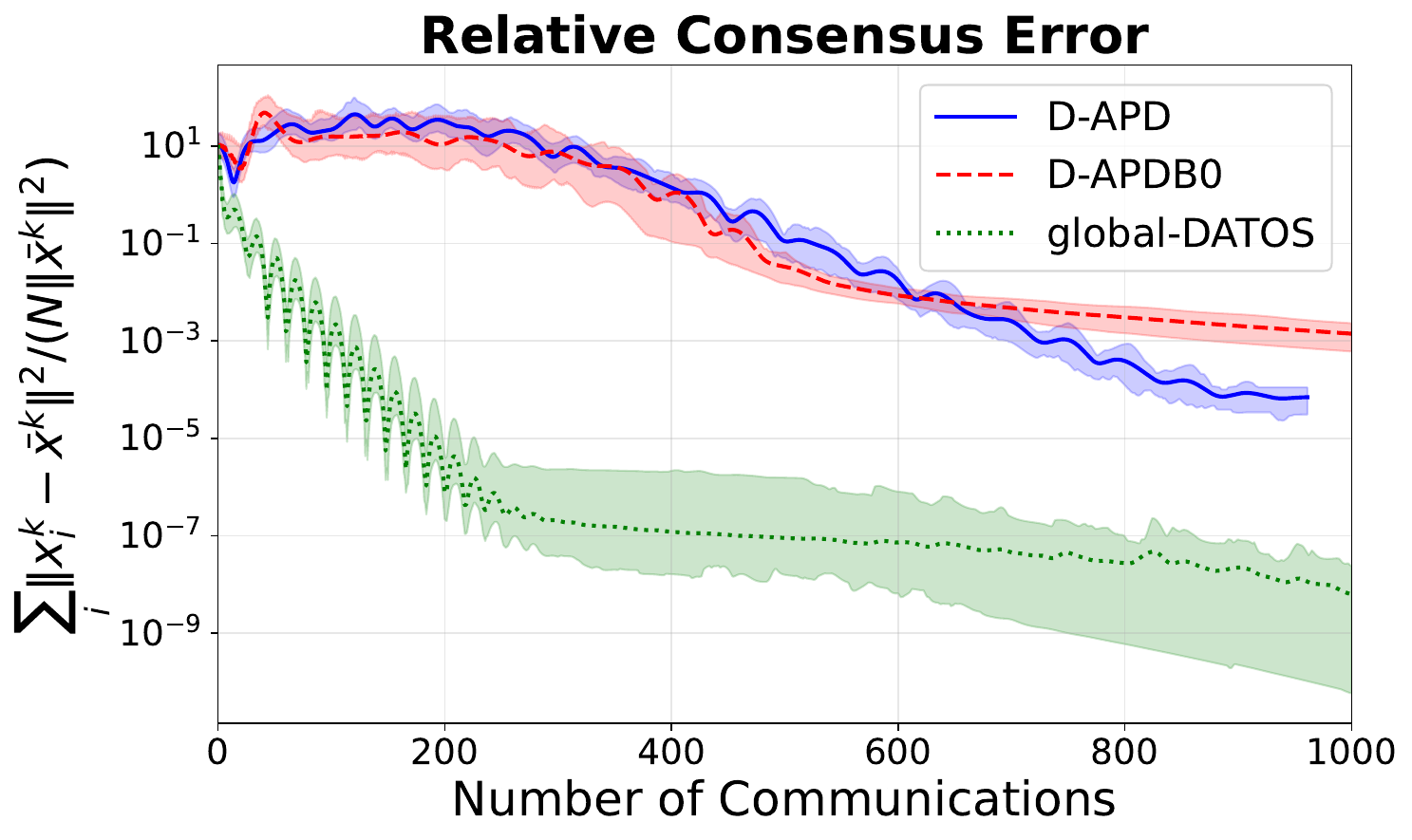}
    \captionsetup{font=scriptsize}
    \caption{Relative Consensus Error}
    \label{fig:qp-consensue}
  \end{subfigure}
  \caption{Comparison of \algz{}~ against \texttt{D-APD} and \texttt{global\_DATOS} for solving \eqref{eq:qp} with 20  simulation.}
  \label{fig:qp-exp}
\end{figure}

\begin{figure}[]
  \centering
  \begin{subfigure}[t]{0.32\textwidth}
    \centering
    \includegraphics[width=\linewidth]{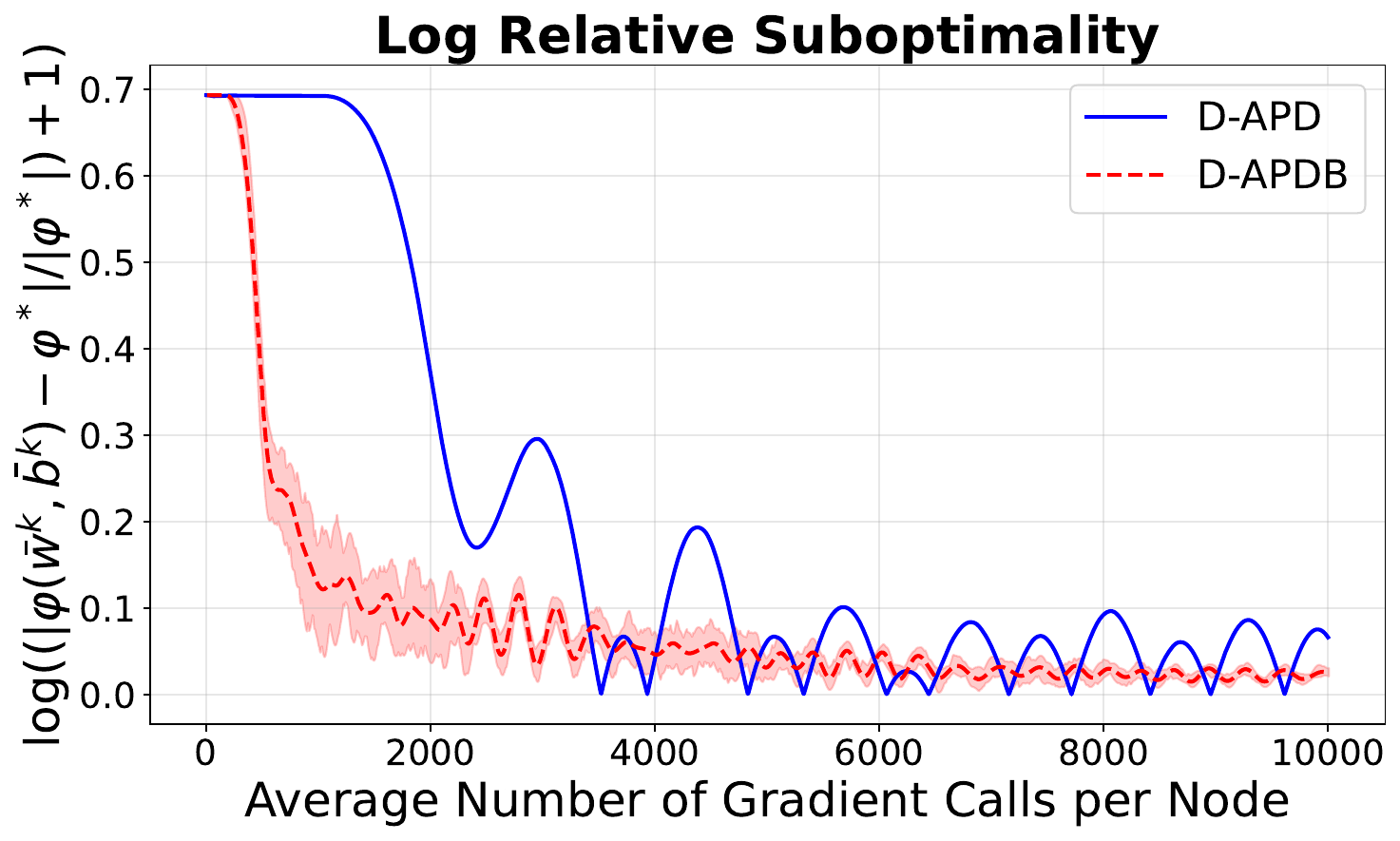}
    \captionsetup{font=scriptsize}
    \caption{Log Relative Suboptimality}
    \label{fig:svm-subopt}
  \end{subfigure}
  \hfill
  \begin{subfigure}[t]{0.32\textwidth}
    \centering
    \includegraphics[width=\linewidth]{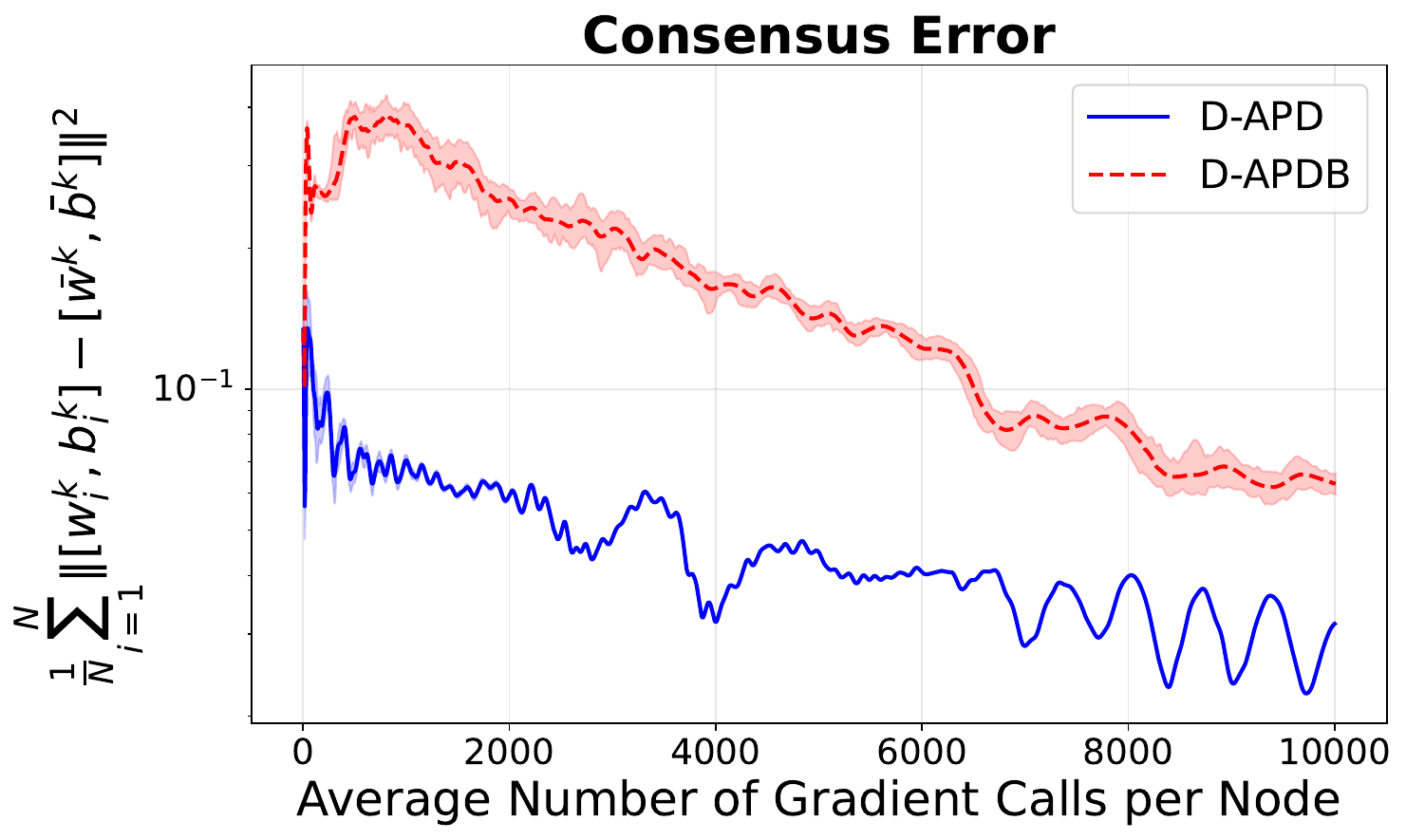}
    \captionsetup{font=scriptsize}
    \caption{Consensus Error}
    \label{fig:svm-consensus}
  \end{subfigure}
  \hfill
  \begin{subfigure}[t]{0.32\textwidth}
    \centering
    \includegraphics[width=\linewidth]{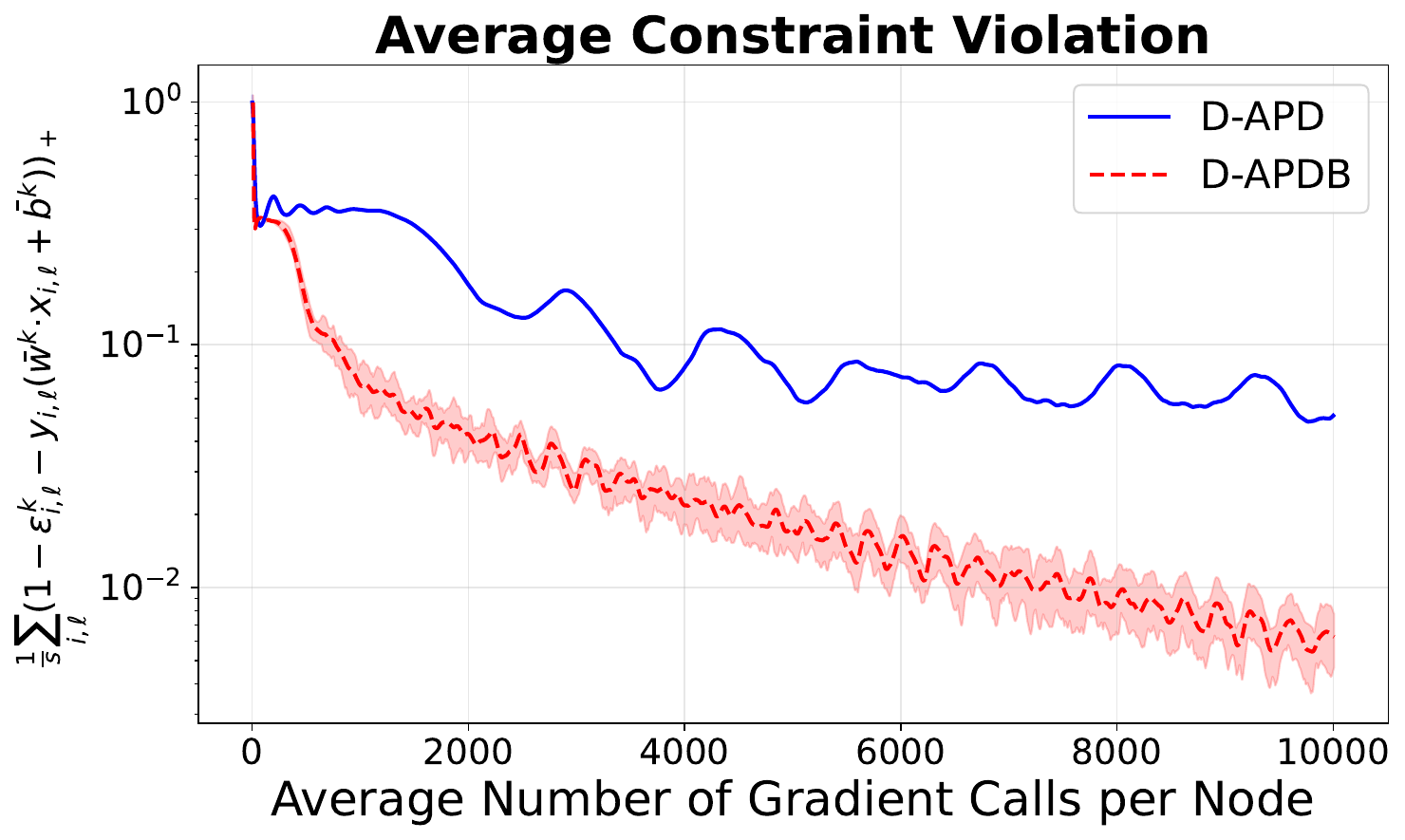}
    \captionsetup{font=scriptsize}
    \caption{Average Constraint Violation}
    \label{fig:svm:constraint}
  \end{subfigure}

  \begin{subfigure}[t]{0.32\textwidth}
    \centering
    \includegraphics[width=\linewidth]{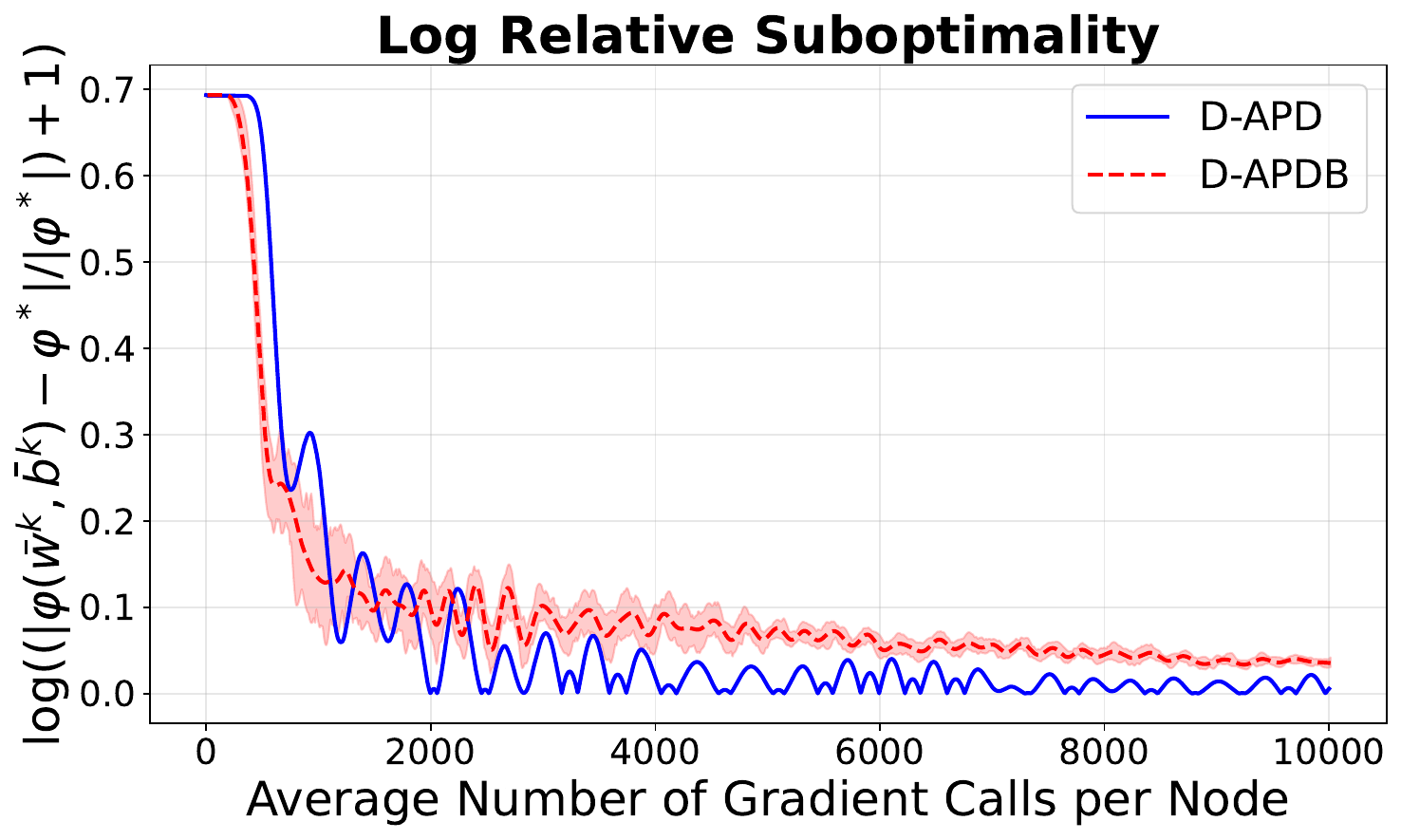}
    \captionsetup{font=scriptsize}
    \caption{Log Relative Suboptimality}
    \label{fig:svm-subopt_kappa_5}
  \end{subfigure}
  \hfill
  \begin{subfigure}[t]{0.32\textwidth}
    \centering
    \includegraphics[width=\linewidth]{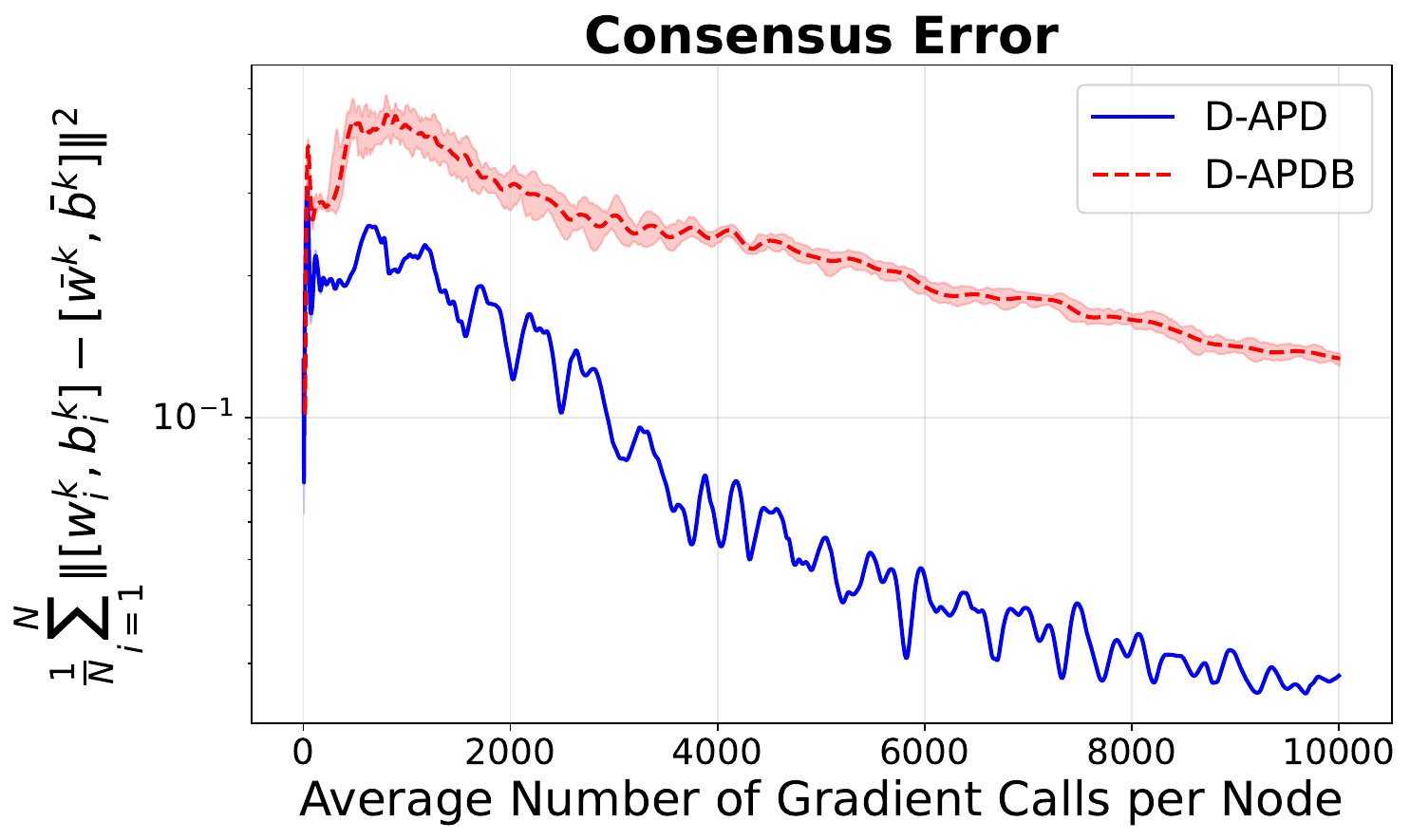}
    \captionsetup{font=scriptsize}
    \caption{Consensus Error}
    \label{fig:svm-consensus_kappa_5}
  \end{subfigure}
  \hfill
  \begin{subfigure}[t]{0.32\textwidth}
    \centering
    \includegraphics[width=\linewidth]{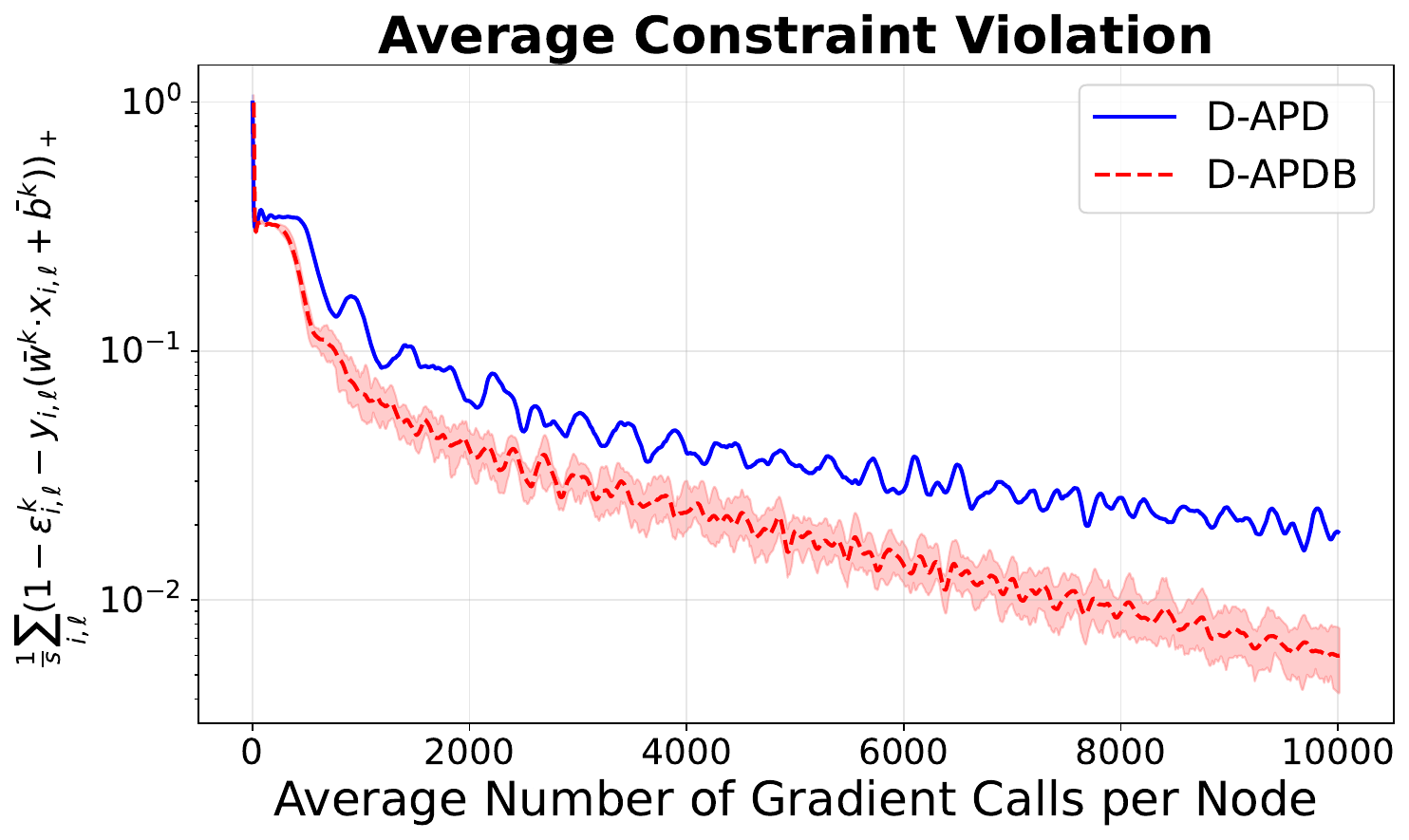}
    \captionsetup{font=scriptsize}
    \caption{Average Constraint Violation}
    \label{fig:svm:constraint_kappa_5}
  \end{subfigure}
  \caption{Comparison of \alg{}~ against \texttt{D-APD} for solving the linear SVM problem in \eqref{eq:svm} over 20 random problem instances. “Average number of gradient calls per node” means $\frac{1}{N}\sum_{i=1}^N n_i^k$ where $n_i^k$ counts how many calls for gradients for node $i$ at iteration $k$.}
  \label{fig:svm-exp}
\end{figure}
\vspace*{-3mm}
\subsection{Primal linear SVM problems.} We consider a primal linear SVM problem with data  distributed among computing nodes in $\cN$. 
Let $\cS\triangleq \{1,2,\cdots,s\}$ denote the set of indices for training data and $\cD \triangleq \{(x_\ell, y_\ell)\in\reals^n \times \{-1,+1\}: \ell\in\cS\}$ be a set of feature vector and label pairs. Let $\{\cS_i\}_{i\in\cN}$ be a partition of $\cS$ among the nodes $\cN$.

Define $\bw=[w_i]_{i\in \cN}$, $\bb=[b_i]_{i\in\cN}$, and $\varepsilon\in\reals^{s}$ such that $w_i\in \reals^n$ and $b_i\in\reals$ for $i\in\cN$. Consider the following distributed SVM problem:
\begin{equation}\label{eq:svm}
    \begin{aligned}
        \varphi(\bw,\bb)\triangleq\min_{\bw,\bb,\varepsilon} &\ \frac{1}{2}\sum_{i\in\cN}\norm{w_i}^2 + NC\sum_{i\in\cN} \sum_{\ell\in\cS_i} \varepsilon_\ell\\
        {\rm s.t.\ } & \quad y_\ell(w_i^\top x_\ell + b_i) \geq 1- \varepsilon_\ell, \quad\varepsilon_\ell\geq 0,\quad \ell\in\cS_i, \quad i\in\cN, \\
        & \quad w_i=w_j, \quad b_i=b_j, \quad (i,j)\in\cE.
    \end{aligned}
\end{equation}
Similar to \cite{forero2010consensus}, $\{x_\ell\}_{\ell\in\cS}$ is generated from two-dimensional multivariate Gaussian distribution with covariance matrix $\Sigma=[1, 0; 0, 2]$ and with mean vector either $m_1=[-1,-1]^\top$ or $m_2=[1,1]^\top$ with equal probability. The experiment was performed for $s = 600$ training points such that for each $i\in\cN$, we randomly sample $50$ training data points from $\cS$, i.e., $|\cS_i|=50$ for $i\in\cN$. For this class of randomly generated primal linear SVM problems, we compare our \texttt{D-APDB} with \texttt{D-APD} in terms of log relative suboptimality, consensus error and average constraint violation of the iterate sequence, i.e., in Fig.~\ref{fig:svm-exp}, we plot
$\log(|\varphi(\bar{w}^k,\bar{b}^k) -\varphi^*|/|\varphi^*| +1)$ in (A), $\frac{1}{N}\sum_{i\in\cN}\norm{[w_i^k, b_i^k]-[\bar{w}^k,\bar{b}^k]}^2$ in (B) and  $\frac{1}{s}\sum_{i\in\cN}\sum_{\ell\in\cS_i}(1-\epsilon^k_{\ell} - y_{\ell}(\bar{w}^{k^\top} x_{\ell} + \bar{b}^k))_+$ in (C), where $\bar w^k=\sum_{i\in\cN}w_i^k/N\in\reals^n$ and $\bar b^k=\sum_{i\in\cN}b_i^k/N\in\reals$ for $k\geq 0$.

For the benchmark algorithm \texttt{D-APD}, \mgc{which uses constant stepsize,} we set the step size $\bar\tau_i=0.005$ (the first row in Fig. \ref{fig:svm-exp}) to mimic the scenario where one picks a conservative stepsize in case the Lipschitz constants are not readily available. For the case when one knows the Lipschitz constants, we set $\bar{\tau}_i = \min\Big\{\frac{1-(\delta+c)}{2L_{f_i}},~\frac{1}{C_{g_i}}\sqrt{\frac{{c_\alpha(1-\delta)}}{\sa{2}\zeta_i}}\Big\}$ shown in the second row of Fig. \ref{fig:svm-exp}. For both cases we set $\zeta_i = 1$ (for dual step size choice), $c_\alpha,c_\beta, c_\varsigma = 0.1$ and $\eta^k = 1$ for $k\geq 0$. For \alg{}, because it can potentially support larger stepsizes,
we set the initial stepsizes to be larger than \mgc{$\bar\tau_i$} choices for \texttt{D-APD} stated above. More specifically, we set the initial stepsize of \alg{} (in \lin{algeq:init} of \alg{}) as 10 and 5 times larger than \mgc{$\bar\tau_i$} of \texttt{D-APD} for the first and second scenario, respectively.
\mgc{Moreover, we set}
the contraction factor $\rho = 0.9$, $\zeta_i = 1$, $c_\alpha,c_\beta, c_\varsigma = 0.1$, and $\delta=0.1$ for both settings. 
All the other parameters are set according to our theoretical results. The results in Fig. \ref{fig:svm-exp} shows that our method outperforms \texttt{D-APD} on average over 20 randomly generated problem instances, where the shaded areas demonstrate the standard deviation of the performance over the runs.
\section{Conclusion}
\mgb{
In this work, we consider cooperative multi-agent constrained consensus optimization problems over an undirected network of agents. Agents have local objective functions with a composite structure—each is the sum of smooth and nonsmooth terms—and each agent is subject to its own private convex conic constraints. We consider the setting when only agents connected by an edge can directly communicate to exchange large-volume data vectors via a high-speed, short-range protocol (e.g., WiFi), and when the network supports one-hop, low-rate information exchange beyond immediate neighbors, as in the LoRaWAN protocol. In this setting, unlike existing decentralized primal–dual methods that require knowledge of Lipschitz constants, we propose a new algorithm, \alg{}, which adapts to local smoothness via a distributed backtracking step-size search. We obtain $\mathcal{O}(1/K)$ convergence \sa{rate} guarantees for suboptimality, infeasibility and consensus \sa{violation}. When nodes have private constraints for which projections are expensive to compute, \alg{}~is, to the best of our knowledge, the first distributed method with backtracking that attains the optimal convergence rate for constrained composite convex optimization problems. We also provide a 
\sa{variant} of our algorithm, \algz{}, for problems \sa{that are either unconstrained or with simple constraint sets for which projections are cheap to compute}. \algz{} can also achieve the same $\mathcal{O}(1/K)$ rate for suboptimality and consensus violation. \mgc{Finally, we present numerical experiments demonstrating the performance of \algz{} and \alg{} on distributed $\ell_1$-norm-regularized QP and QCQP problems, and distributed primal SVM training.}}

\vskip 6mm
\noindent{\bf  Acknowledgements}

\noindent
This paper builds on momentum acceleration, and we dedicate it to Prof. Yurii
Nesterov (UCLouvain, Belgium), whose seminal contributions to the development and analysis
of optimization algorithms have been profoundly influential in this field. This research was
supported in part by the Office of Naval Research under award number N00014-24-1-2628 and
N00014-24-1-2666.

\end{document}